\documentclass{amsart}

\usepackage{stackrel}
\usepackage[utf8]{inputenc}
\usepackage{amssymb}
\usepackage{amsthm}
\usepackage{tikz}
\usepackage{mathrsfs}
\usepackage{epsf}
\usepackage{bookmark,hyperref}
\hypersetup{
     colorlinks=true,
     linkcolor=blue,
     filecolor=blue,
     citecolor = blue,
     urlcolor=cyan,
     }
\usepackage{pstricks}
\usepackage{mathtools}

\usepackage{amsmath,amssymb,amscd}
\usepackage{amsthm}
\usepackage{subfigure}
\usepackage{verbatim}
\usepackage{fancybox}
\usepackage{subfigure}
\usepackage[all]{xy}

\usepackage{accents}

\usepackage{tikz-cd}

\usepackage{bm}

\newcommand{\N}{{\mathbb N}}

\newcommand{\Z}{{\mathbb Z}}

\newcommand{\R}{{\mathbb R}}
\newcommand{\C}{{\mathbb C}}

\newcommand{\I}{{\mathcal{O}}}

\newcommand{\grad}{\mathsf{ grad}}
\newcommand{\Fi}{\mathrm{ F}}
\newcommand{\gi}{\mathrm{ g}}
\newcommand{\hi}{\mathrm{ h}}
\newcommand{\ind}{\mathrm{ ind}}
\newcommand{\Hess}{\mathrm{ Hess}}
\newcommand{\id}{\mathrm{ id}}

\newtheorem{teo}{Theorem}[section]
\newtheorem{lema}[teo]{Lemma}
\newtheorem{cor}[teo]{Corollary}

\newtheorem{prop}[teo]{Proposition}

\newtheorem{defi}{Definition}[section]
\newtheorem{rmrk}{Remark}[section]

\makeatletter
\@namedef{subjclassname@2020}{%
  \textup{2020} Mathematics Subject Classification}
\makeatother

\emergencystretch=\maxdimen
\begin{document}

\title[McGehee blowup and instability]{McGehee blowup for Lagrangian systems and instability of equilibria}

\author{J. M. Burgos}
\address{Departamento de Ciencias exactas y naturales, Universidad Católica del Uruguay, Av. 8 de Octubre 2738, 11200, Montevideo, Uruguay.}
\email{jburgos@ucu.edu.uy}

\begin{abstract}
We prove that total instability is a generic phenomenon in the real analytic class of electromagnetic Lagrangian systems under a weak magnetism hypothesis. The main object in the proof is an adaptation of the McGehee blowup for these systems. Together with this result, new criteria for total instability are introduced for both generic and non-generic cases.
\end{abstract}





\maketitle








\section{Introduction}

Since its appearance in \cite{McGehee}, the McGehee blowup has been extensively used, mainly in the context of celestial mechanics, where it was first applied. This paper adapts the original technique to certain Lagrangian systems and proves among other results that total instability, a stronger form of instability than Lyapunov's for non-minimum points of the potential, is a generic phenomenon for these systems. Specifically, we will consider electromagnetic systems verifying a weak magnetism hypothesis at some critical point of the potential and prove that, generically, the point is a totally unstable equilibrium.

Adaptations of the original blowup are not new. McGehee himself adapted his technique to produce an infinity manifold instead of blowing up the origin \cite{McGehee2}. In \cite{Devaney}, Devaney focused on the case of a degree $k$ isolated singularity of the potential of a mechanical system and proved that, generically in the Whitney differentiable topology, the singularity is Morse-Smale. In \cite{Lacomba}, Lacomba and Ibort adapted the technique to mechanical systems whose potential is a homogeneous polynomial. Recently, Knauf and Montgomery compactified the energy levels of the $N$-body problem by systematic application of McGehee blowups \cite{Knauf_Montgomery}. This list is by far non-exhaustive.

In adapatations of the McGehee blowup to mechanical systems in Euclidean space, the reason to consider homogeneous functions as potentials, either with a singularity or a critical point at the origin, is because the motion equations in the McGehee variables decouple from the radial one. This allows to study the global dynamical behaviour of the ambient space from the dynamics on the boundary, either the collision or the infinity manifold, resulting in a sort of holography.

In contrast to the previous paragraph, adaptations like the one in \cite{Devaney} where the potential is non-homogeneous, assumes additional dynamical structure on the singular manifold, for instance to be a Morse-Smale system, to derive results on the dynamics of the ambient space from the one on the boundary. Except for the transversality condition on the stable and unstable manifolds, this is achieved by assuming generic conditions on the first non-zero jet of the potential.


In this paper, under a weak magnetism hypothesis on the magnetic term, we adapt the McGehee blowup to an electromagnetic system with a possibly non-homogeneous function potential having some critical point $p$ in configuration space and study, from the dynamics on the critical and subcritical McGehee boundary, the original dynamics near the fixed point $(p,0)$ in phase space.

This study of the dynamics do not presuppose any additional dynamical structure on the boundary and requires a restricted notion of chain-recurrence suitable for the problem. Through this technique, we prove a mechanical tightness phenomenon with far reaching consequences. First, new non-generic criteria for total instability are derived, and second, after adding some generic dynamical structure at the critical boundary, it is established that total instability is a generic phenomenon for electromagnetic systems verifying a weak magnetism hypothesis. The mentioned hypothesis about magnetism will be stated in the following subsection.







Here, we deliberately decided to work in the real analytical class, mainly for two reasons. The first is that, for even degree values of the first non-zero jet of the potential at the critical point in question, the objects constructed here are naturally real analytic, a rich property that could eventually be exploited in future applications. The second and main reason is that the notion of genericity is much stronger with respect to the real analytic class than it is with respect to the differentiable class since the topology of the former is strictly finer than the one of the latter.

In contrast with the smooth class in which there are smooth Urysohn functions, in the real analytic class we will have to deal with real analytic germs at a point, that is, real analytic objects defined on a sufficiently small neighbourhood of the point that a priori cannot be further extended. Another consequence of the lack of these functions is that the genericity of the hypotheses in the local ring of real analytic germs, does not follow directly from general results regarding the Whitney differentiable topology, but rather from specific properties of the local ring of germs only.




A striking difference between the object built here and the classical McGehee blowup in $N\,$-\,body problem is the following. Contrary to celestial mechanics in which the configuration space is the Euclidean space minus the diagonals and therefore there is a natural set of coordinates, here the configuration space is a real analytic manifold and a local coordinate chart must be chosen to perform the constructions. However, in the appendix, the \emph{abstract McGehee blowup} will be constructed, which is an intrinsic object independent from its coordinates that naturally adopts the usual McGehee coordinates once a local chart on the configuration manifold is chosen. While the purpose of this intrinsic abstract object is to demonstrate the consistency of the constructions, it contributes nothing to our topic here, which is dynamics, and for this reason, it is relegated to the appendix.

\subsection{Statement of the results}

Consider a point $p$ in a real analytic Riemannian manifold $(E,g)$. Every real analytic function $U$ and real analytic one-form $\mu$ defined on some neighbourhood $A$ of the point determines a Lagrangian
$$L:TA\rightarrow\R,\qquad L(x,v)\,=\,g_x(v,v)/2\,+\,\mu_x(v)\,-\,U(x)$$
along with the corresponding Euler-Lagrange flow on $TA$ determined by the motion equations
$$\dot{x}\,=\, v,\qquad \nabla_{\dot{x}}\, v\,=\,-\grad_{g}\,U\,(x)\,-\,\mathcal{R}_x(\iota_{v}\,d\mu(x)).$$
Here, nabla denotes the Levi-Civita connection of the Riemannian metric $g$ and the operator $\mathcal{R}$ is the Riesz duality. This is not a Hamiltonian formulation and these are not the respective Hamilton equations of the system. However, this particular formulation is suitable for our purposes.

The function $U$ and one-form $\mu$ will be called the potential and magnetic potential respectively, while $A$ and $TA$ will be called the configuration and phase space respectively.


If $p$ is a critical point of the potential, then $(p,0)$ is a fixed point of the flow, that is, $p$ is an equilibrium point of the system. The energy function\footnote{We emphasize that this is the energy function, not the Hamiltonian of the system.} defined by
$$H:TA\rightarrow\R,\qquad H(x,v)\,=\,g_x(v,v)/2\,+\,U(x)$$
is constant along the orbits. Modifying the potential by an additive constant do not alter the flow thus, without loss of generality, we may assume $U(p)=0$ along the paper. Note that with this assumption, the energy is zero at $(p,0)$ and we will say that a point $z$ in phase space is subcritical if $H(z)<0$.

Throughout the paper, the magnetic potential will verify the following hypothesis called \emph{weak magnetism hypothesis}, namely
$$d\,-l/2\, \geq\,1,$$
where $l$ and $d$ are the degree values of the first non-zero jet of the potential and magnetic potential respectively.

A point $p$ is \emph{totally unstable} if there is a neighbourhood $B\subset A$ of the point such that for any subcritical point $z$ in $TB$, the positive orbit $I^+(z)$ is not contained in $TB$. That is, the orbit of any subcritical point in $TB$ eventually escapes from this set in the future.

Consider the local ring of real analytic germs $\I_p$ at the point $p$ and denote by $\mathfrak{m}$ its maximal ideal. The ideal $\mathcal{I}_k= \mathfrak{m}^k\lhd \I_p$ is naturally identified with the set of real analytic functions defined on some neighbourhood of $p$ whose first non-zero jet at the point has degree $k$.


\begin{teo}\label{main_theorem}
For any natural $k\geq 2$, there is an open and dense subset $\mathcal{A}$ of the ideal $\mathcal{I}_k$ such that, for every potential germ in $\mathcal{A}$ and magnetic potential verifying the weak magnetism hypothesis, the point $p$ is totally unstable.
\end{teo}

Note that the definition of total instability does not require the point $p$ to be critical. Actually, every point of an inclined plane is totally unstable. Even more, a strict minimum of the potential is totally unstable since the statement is vacuously true, this being a Lyapunov stable equilibrium point. It is in those cases where the point $p$ is critical and non-minimum, that the result is non-trivial and the point verifies a much stronger instability condition than Lyapunov's implying the latter.

Denote by $f$ the restriction to the unit sphere of the first non-zero jet of the potential $U$. From a stronger result proved in section \ref{Total_inst_section}, the \emph{mechanical tightness} phenomenon, the following total instability criterion follows, from which, after establishing the genericity of the hypotheses in sections \ref{Genericity} and \ref{Genericity_II}, Theorem \ref{main_theorem} is proved.

\begin{prop}[Total instability criterion]\label{total_inst_criterion_0}
Suppose that $f$ is a Morse function and zero is a regular value. Then, the point $p$ is totally unstable.
\end{prop}

Total instability is a local property in configuration space hence it only depends on the germ of the Lagrangian at $T_p E$. In this respect, considering the fact that in the smooth class, the stability question depends non-trivially on the kinetic term, it is quite interesting that at least in the real analytic class, a sufficient condition for total instability involves the first non-zero jet of the potential only without any mention neither to higher order terms of it nor to those of the kinetic and magnetic term.

In \cite{Palamodov}, Palamodov proved that in the purely mechanical case, that is without magnetic potential, if the first non-zero jet of the potential at $p$ has no critical point other than the origin, then the point $p$ is totally unstable. The main step of the proof consists in deforming a radial vector field in configuration space such that the resulting vector field $V$ verifies $V(U)=U$.

In the purely mechanical case, the total instability criterion, Proposition \ref{total_inst_criterion_0}, is weaker than Palamodov's Theorem since it requires the generic additional hypothesis that $f$ is Morse. However, \textbf{it is the first total instability criterion that covers magnetism}, something not achievable by arguing from the configuration space only, and proves total instability for generic electromagnetic systems verifying a weak magnetism hypothesis.

The inclusion of magnetism or any other non-trivial generalization to total instability arguments may have implications in other lines or work. For instance, Allaire and Rauch proved in \cite{Ref_crucial} the Earnshaw's Theorem asserting the absence of stable equilibrium configurations of conductors and dielectrics in an external electrostatic field. It is claimed in that paper that a total instability proof, not available at that time, would imply their instability theorems among others, see section 1.7 in \cite{Ref_crucial}.

As our knowledge goes, so far all of the instability criteria either in the total or Lyapunov sense, require the generic hypothesis that zero is a regular value of the function $f$. In view of this, a non-generic instability criterion is important since it bounds the field of possibilities in which to look for a counterexample, if any, to any of the instability conjectures regarding real analytic mechanical systems.

The following is \textbf{the first criterion covering the non-generic case of a zero critical value of $f$.} It will involve higher order terms of the potential and the metric tensor coefficients as well.
Although, at least in the smooth class, its non-trivial dependence on the stability of an equilibrium point is known, \textbf{this is the first time that the kinetic term is considered in an instability criterion.}

We now proceed to state the criterion. Suppose that $U$ is not a homogeneous polynomial and consider the first and second non-zero homogeneous terms in the Maclaurin series of the potential and the metric tensor coefficients,
$$U\,=\,U_{l_1}\,+\,U_{l_2}\,+\,\ldots,\qquad \gi_{ij}\,=\,\delta_{ij}\,+\,\gi^{(m)}_{ij}\,+\,\ldots.$$
We have assumed without loss of generality, taking a linear coordinate change at the origin if necessary, that $\gi^{(0)}_{ij}=\delta_{ij}$. In what follows, if the metric tensor is the Euclidean one, that is if the Lagrangian is Newtonian, then formally we can take $m$ as infinity.

It will be assumed that the degree value $d$ of the first non-zero jet of the magnetic potential at $p$ verifies a weaker magnetism hypothesis than the one before, namely
$$d\,-l_1/2\, >\,\mu\,=\,\min\,\{l_2-l_1,\,m\},$$
and the reason is simply because otherwise the criterion will not decide.

Define the \emph{criterion function} on the unit sphere given by
$$\mathcal{C}(q)\,=\,-\delta_\mu^{l_2-l_1}\,(l_2-l_1)\,U_{l_2}(q)\,+\,
\delta_\mu^m\,m\,U_{l_1}(q)\,\sum_{i,j=1}^n\,\gi^{(m)}_{ij}(q)\,q^i q^j,\qquad q\in S^{n-1}.$$

Recall that $f$ is the restriction of $U_{l_1}$ to the unit sphere.

\begin{prop}[Non-generic total instability criterion]
If $\mathcal{C}(q)>0$ on every critical point $q$ of $f$ such that $f(q)\leq 0$, then the point $p$ is totally unstable.
\end{prop}

\begin{cor}
In Newtonian mechanics, the point $p$ is totally unstable if the second non-zero homogeneous term in the Maclaurin series of the potential is negative on the critical locus of the the first non-zero homogeneous term restricted to the unit sphere where the latter term is less than or equal to zero.
\end{cor}

Regarding the study of asymptotic orbits, restoring the original weak magnetism hypothesis, we have \textbf{without any hypothesis on the first non-zero jet of the potential} the following result for electromagnetic systems.

\begin{prop}
Consider a non-equilibrium orbit converging to $(p,0)$ in phase space. Then, it is asymptotic in configuration space to some connected component of the critical locus of $f$ where the function is lower than or equal to zero.
\end{prop}

\subsection{A little history on instability}

To put the previous results into perspective, we will describe below some relevant facts in the chronology of the instability problem. Again, the list is by far non-exhaustive.

The instability problem dates back to Lyapunov \cite{Lyapunov} when he asked about the converse of the Lagrange-Dirichlet Theorem \cite{Lagrange, Dirichlet}, that is, whether a Lyapunov stable equilibrium point is necessarily a strict minimum of the potential function. However, a series of counterexamples appearing later showed that the converse is false in the smooth class of potentials and very counterintuitive phenomena can arise.

The first examples where the Lagrange-Dirichlet converse fails were due to Painlev\'e in dimension one (see Wintner's version \cite{Wi}) and Laloy \cite{Laloy1} in dimension two. In constrast with the first example which has trapping zones, the latter is particularly interesting since it has escape routes to infinity starting from the non-minimum critical point that a priori a particle could follow. However, in both examples the critical point turns out to be Lyapunov stable.

Another counterintuitive result is the non-trivial dependence of the kinetic term of the Lagrangian on the stability of an equilibrium. Indeed, Bertotti and Bolotin \cite{BB} showed that in dimension two, there exist a smooth potential $U$ together with smooth kinetic terms $Q$ and $\tilde{Q}$ such that the equilibrium point is stable for the system $L=Q-U$ but unstable for the system $\tilde{L}=\tilde{Q}-U$. Concrete examples of this phenomenon were given by Garcia, Tal and Bortolatto in \cite{GT2} and \cite{BGT}.

On the other hand, it is well known since the work of Hagedorn in \cite{Ha} that a local strict maximum of a smooth potential is Lyapunov unstable for any smooth kinetic term. This result was later improved by Taliaferro in \cite{Taliaferro} for local maxima and Sofer in \cite{Sofer} for local maxima and generalized systems whose natural framework is Finsler geometry instead of Riemannian geometry. See also the variational approach by Hagedorn and Mawhin in \cite{Hagedorn2}.

In view of the previous results, it is natural to ask for the Lagrange\,-\,Dirichlet converse in some restricted class of potentials, for example the class of real analytic potentials. In dimension two, the result follows from a collection of independent results \cite{Pa3}, \cite{Ta}, \cite{Ko}, \cite{Laloy}. There is an elegant proof given by Brunella in \cite{Br} using desingularization results for singular vector fields in three dimensional space developed by Cano in \cite{Cano1} and \cite{Cano2}.

In dimension greater than two, the problem remains open and only partial results are known. Most of these partial results prove the existence of some asymptotic orbit assuming some non-degeneracy condition on the first non-zero jet of the potential \cite{Lyapunov}, \cite{Cetaev}, \cite{Ha}, \cite{Hagedorn2}, \cite{Ref2}, \cite{Ref3}, \cite{Ref4}, \cite{Ref5}, \cite{Ref6}, \cite{Sofer}, \cite{Taliaferro}, \cite{Ref7}.

In \cite{Arnold}, Arnold proposed the problem of instability of an isolated non-minimum critical point of a real analytic potential in Newtonian mechanics. Even this simplified version of the Lagrange-Dirichlet converse remains open. The case of a non-strict minimum of the potential was treated in \cite{BMP} for Newtonian mechanics and in \cite{BMP} for Lagrangian systems.

It is worth to emphasize that all of the previous partial results are in the context of Lyapunov instability and proving the existence of an asymptotic orbit is not enough to assure total instability. As it was mentioned in the previous part, Palamodov proved total instability for real analytic potential whose first non-zero jet has no critical points other than the origin \cite{Palamodov} and this had implications for the work on the Earnshaw's Theorem by Allaire and Rauch in \cite{Ref_crucial}.

Another counterintuitive result regarding instability in the smooth class is the following. A fixed point in phase space is Siegel-Moser unstable if there is a neighbourhood of the point such that the set consisting of the solely point is the only invariant set in the neighbourhood. In the context of Newtonian mechanics in the plane, Ureña gave an example of a smooth potential whose strict maximum point $p$ verifies that $(p,0)$ is Siegel-Moser stable \cite{Urena}.

Introducing magnetism in the picture is quite subtle. Actually, even in the case of a zero potential in the plane, a transverse non-zero magnetic field makes every point in the plane Lyapunov stable. In contrast with the purely mechanical case, in the electromagnetic one the previous example leaves no room for the converse of Routh's Theorem \cite{Routh}, stating that any strict minimum of a potential function is Lyapunov stable.

In \cite{Hagedorn} and its generalizations \cite{BK}, \cite{Fu}, \cite{Ref6}, \cite{So}, the condition required on the magnetic term for the Lyapunov instability to hold is very similar to the weak magnetism hypothesis considered here. The mentioned condition necessarily implies the vanishing of the magnetic field at the point. In contrast to the previous, instability conditions in which the magnetic field is non-zero at the point were treated in \cite{BN},\cite{BN2}, \cite{BP} and \cite{So2}. Other conditions were treated in \cite{Kozlov_dissipation} under the presence of dissipative forces, something we do not assume in this paper.

As in the purely mechanical case, it is worth to notice that all of the previous results involving magnetism are in the context of Lyapunov instability. The theorem and criteria presented in this paper are the first total instability results for electromagnetic systems.

\section{McGehee blowup for Lagrangian systems}\label{blowup_definition_section}

In this section we adapt the McGehee blowup for the case of an equilibrium point of a real analytic Lagrangian system, specifically, an electromagnetic Lagrangian with a weak magnetism hypothesis with respect to the potential. These systems include the real analytic mechanical systems. The development parallels the approach taken by Montgomery in \cite{Montgomery}.

As mentioned in the introduction, in contrast with the smooth class in which there are smooth Urysohn functions, in the real analytic class we will have to deal with real analytic germs at a point, that is real analytic objects defined on a sufficiently small neighbourhood of the point that a priori cannot be further extended.

For the construction of the objects in this section, we will take coordinates from the beginning via some coordinate map. However, in appendix \ref{abstract_blowup_appendix}, the \emph{abstract McGehee manifold} will be constructed, which is an intrinsic object independent from its coordinates that naturally adopts the usual McGehee coordinates once a local chart on the configuration manifold is chosen. In that appendix, the intrinsic analogue of the object in Proposition \ref{McGehee_blowup_mechanical_system}, the \emph{abstract McGehee blowup}, will also be constructed. Therefore, although a priori it will seem that the main objects constructed in this section depend on the chosen coordinates, a posteriori this will not be the case.

In particular, for the purposes of this section, the approach taken here for the construction of the objects via the mentioned coordinate map is sufficient. This allows the reader to skip the mentioned appendix, at least on a first reading, and has the advantage of going directly to the central object of this paper which is the \emph{McGehee blowup for Lagrangian systems}.


Consider a real analytic Riemannian manifold $(E,g)$, that is a real analytic manifold $E$ with a compatible real analytic Riemannian metric $g$. Consider a point $p$ in $E$ and a real analytic function germ $[\,U]$ at the point, such that $p$ is a critical point of the germ. That is, there is a neighbourhood $A$ of $p$ in $E$ whereat the real analytic function $U$ is defined and $p$ is a critical point,
$$U:\, A\rightarrow \R,\qquad d_p\, U\,=\,0,\qquad p\in A\subset E.$$
Similarly, consider a real analytic one-form germ $[\mu]$ at the point $p$. Taking a smaller neighbourhood $A$ if necessary, we may suppose that $\mu$ is a real analytic one-form on $A$,
$$\mu\,\in\,\Omega^1(A).$$

The function will be called the \emph{potential energy} or simply the \emph{potential} \footnote{Not to be confused with the potential action which is the opposite.} and the form will be called the \emph{magnetic potential}. On the tangent bundle $TA$, consider the real analytic Lagrangian
\begin{equation}\label{Mechanical system}
L:\,TA\rightarrow\R, \qquad L=K+\mu-U
\end{equation}
where $K$ is the \emph{kinetic term}, given by
$$K(x,v)\,=\,g_x(v,v)/\,2,\qquad (x,v)\in TA.$$

Taking a small enough neighbourhood $A$, we may suppose that $(A,\eta)$ is a coordinate neighbourhood centered at the origin,
$$\eta:A\rightarrow \R^n,\qquad \eta\,(p)=0.$$
Taking a linear coordinate change if necessary, we may suppose that the pullback by $\eta$ of the usual inner product at the origin coincides with the metric at $p$, that is
$$\eta^*\,\langle\cdot,\cdot\rangle_0\,=\,g_p.$$
In terms of the metric tensor coefficients, this condition reads as follows
$$(\eta^{-1})^*g\,=\,\sum_{a,b=1}^n\,\gi_{ab}\,dx^a\otimes dx^b,\qquad \gi_{ab}(0)\,=\,\delta_{\it{ab}}.$$

Moreover, since the problem we address here is of local nature, we may even suppose without loss of generality that $A$ is the manifold $E$ itself and is identified with its image in $\R^n$ by the map $\eta$. That is, \textbf{$E$ is a neighbourhood of the origin in $\R^n$}, and we have
\begin{equation}\label{Mechanical system_2}
L:\,TE\rightarrow\R, \qquad 0\in E\subset\R^n.
\end{equation}

The space $E$ will be called the \emph{configuration space} and its tangent bundle $TE$ will be called the \emph{phase space}. In view of the previous identification, it will be convenient to distinguish notationally the origin $p=0$ in configuration space from the origin $(p,0)=\mathbf{0}$ in phase space.


The Euler-Lagrange equations resulting from the Lagrangian \eqref{Mechanical system}, are equivalent to the motion equation
\begin{equation}\label{Newton_eq}
\nabla_{\dot{x}}\, \dot{x}\,=\,-\grad_{g}\,U\,(x)\,-\,\mathcal{R}_x(\iota_{\dot{x}}\,d\mu(x))
\end{equation}
where nabla denotes the Levi-Civita connection of the Riemannian metric $g$ and $\mathcal{R}$ is the Riesz duality operator. Equivalently, equation \eqref{Newton_eq} can be written as a dynamical system in phase space given by the vector field $X$ on $TE$ defined by the equations
\begin{equation}\label{Newton_eq2}
\dot{x}\,=\, v,\qquad \nabla_{\dot{x}}\, v\,=\,-\grad_{g}\,U\,(x)\,-\,\mathcal{R}_x(\iota_{v}\,d\mu(x)).
\end{equation}
As mentioned in the introduction, this is not a Hamiltonian formulation and these are not the respective Hamilton equations of the system. However, this particular formulation is suitable for our purposes. The induced flow in phase space is the \textit{Euler-Lagrange flow}. It is clear from \eqref{Newton_eq2}, that the critical point $p$ of the potential $U$ is an equilibrium point and $(p,0)$ is a fixed point of the Euler-Lagrange flow.


Without loss of generality, we may suppose that \textbf{the potential germ is non-null and is zero at $p$},
$$[\,U]\neq 0,\qquad [\,U]\,(p)=0.$$
In particular, $U$ has a first non-zero jet $U_l$ at $p$ with degree $l$. The function $U_l$ is a homogeneous polynomial on $\R^n$ with degree $l\geq 2$, since $p$ is a critical point of the germ $[\,U]$ by hypothesis,	
$$U_l:\R^n\rightarrow\R,\qquad U_l\in S_l(\R^n),\qquad l\geq 2,$$
where $S_l(\R^n)$ denotes the space of $l$-homogeneous polynomials on $\R^n$.


Along the paper, \textbf{we will assume the weak magnetism hypothesis} on the magnetic potential $\mu$. It is defined as follows.
\begin{defi}
The magnetic potential $\mu$ verifies the weak magnetism hypothesis if the degree $d$ of its first non-zero jet at $p$ verifies
\begin{equation}\label{weak_magnetism}
\Delta\,=\,d\,-l/2\, \geq\,1.
\end{equation}
\end{defi}


Define the \emph{McGehee coordinates} $(r,\,q,\,y)\in (0,+\infty)\times S^{n-1}\times \R^{n}$ as follows
\begin{equation}\label{Coordinates_McGehee}
0\neq x\,=\,r\,q,\quad r>0,\ q\in S^{n-1},\qquad v\,=\,r^{l/2}\,y.
\end{equation}
Extending these coordinates to $r=0$, we have the continuously differentiable map
\begin{equation}\label{map_pi}
\pi:\,M\,=\,[0,+\infty)\times S^{n-1}\times \R^{n}\rightarrow \R^n\times \R^n\qquad (r,q,y)\mapsto (x,v).
\end{equation}
The \emph{McGehee manifold} $M$ is a manifold with boundary and $\pi$ is a real analytic diffeomorphism on its interior,
$$\partial M\,=\,\{0\}\times S^{n-1}\times \R^{n}\,=\,\pi^{-1}\left(\mathbf{0}\right),\qquad \pi:\,M-\partial M\xrightarrow{\cong}\,\left(\R^n-\{0\}\right)\times \R^n.$$

Note that, unless $l$ is an even number, the map $\pi$ in \eqref{map_pi} cannot be real analytically extended to the boundary. This is because of the $r^{l/2}$ factor in the definition of the rescaled velocity $y$ in \eqref{Coordinates_McGehee}.


It will be convenient to decompose the rescaled velocity $y$ in its tangent and normal component to the sphere
\begin{equation}\label{function_nu}
y\,=\,y_{\,tg}\,+\,\nu\,q,\qquad \nu:M\rightarrow\R,\qquad \nu\,(r,q,y)\,=\,\left\langle q,\,y\right\rangle.
\end{equation}
Note that $\nu$ is globally defined on $M$ and it verifies
\begin{equation}\label{y_tg_identity}
\Vert y\Vert^{\,2}\,=\,\Vert y_{\,tg}\Vert^{\,2}\,+\,\nu^{\,2}.
\end{equation}


Recall that $E$ is a neighbourhood of zero in $\R^n$. Taking a smaller neighbourhood if necessary, we may suppose without loss of generality that $E$ \textbf{is actually a centered ball at the origin with radius $r_E>0$},
$$E\,=\,B(0,\,r_E)\subset \R^n,\qquad r_E>0.$$
We will abuse of notation by restricting the domain $M$ of the map $\pi$ to $\pi^{-1}(TE)$ and preserving the names of the objects, that is
$$M\,=\,\pi^{-1}(TE)\,=\,[0,r_E)\times S^{n-1}\times \R^{n},\qquad \pi:M\rightarrow TE.$$

\begin{prop}\label{Extension_vector}
On $M-\partial M$, consider the vector field
$$\tilde{X}\,=\, r^{1-l/2}\,\left(\pi^{-1}\right)_*(X).$$
Then, $\tilde{X}$ is real analytic and extends uniquely and continuously differentiable to the manifold $M$. The boundary $\partial M$ is an invariant set of the corresponding dynamical system and the extension of $\tilde{X}$ is real analytic on this set. Moreover, if $l$ is even, then the extension is real analytic.
\end{prop}
\begin{proof}
The radial rescaling in the definition of $\tilde{X}$ can be seen as the time rescaling $$d\tau\,=\,r^{l/2-1}\,dt.$$
For the time derivative with respect to $\tau$ we will use prime instead of dot.

Using Einstein summation convention on repeated indices, equations \eqref{Newton_eq2} read as follows
\begin{equation}\label{Newton_eq3}
\dot{x}^k\,=\, v^k,\qquad \dot{v}^k\,=\,-\gi^{ka}(x)\,\partial_a U\,(x)\,-\,\Gamma^k_{ij}(x)\,v^i\,v^j\,+\gi^{ka}(x)\,\Fi_{ab}(x)\,v^b
\end{equation}
where $(\gi^{ab})=(\gi_{ab})^{-1}$, $\Gamma^k_{ij}$ are the Christoffel symbols and $\Fi_{ab}$ are the magnetic field strength coefficients
$$\mu=\sum_i\,A_i dx^i,\qquad d\mu=\sum_{i<j}\,\Fi_{ij} dx^i\wedge dx^j,\qquad \Fi_{ij}=\partial_i A_j -\partial_j A_i.$$

After a tedious but straightforward calculation, by Lemma \ref{extension_Lemma}, on $M-\partial M$ these equations read as follows
\begin{equation}\label{McGehee_eq}
\begin{array}{r@{}l}
\medskip
r'\,&=\,\nu\,r,\\
\medskip
q'\,&=\,y\,-\,\nu\,q,\\
y'\,&=\,-\nabla U_l\,(q)\,-\,\frac{l}{2}\,\nu\,y\,+\,r\,F_X(r,q,y)\,+\,r^\Delta\,\mathbf{F}_X(r,q,y),
\end{array}
\end{equation}
where nabla denotes here the usual operator in Euclidean space and $F_X$ together with $\mathbf{F}_X$ are globally defined real analytic vector-valued function on $M$. Specifically,
\begin{equation}
\begin{array}{r@{}l}
\bigskip

F_X^k(r,q,y)\,&=\,-\, V^k_{>l}\,(r,q)\,-\,\hi^{ka}(r,q)\,\partial_a U_l(q)\,-\,\Gamma^k_{ij}(r\,q)\,y^i\,y^j,\\

\mathbf{F}_X^k(r,q,y)\,&=\,\gi^{ka}(r\,q)\,\Fi_{ab,\,\geq d}(r,q)\,y^b
\end{array}
\end{equation}
where $\hi^{ka}$, $V_{>l}$ and $\Fi_{ab,\,\geq d}$ are globally defined real analytic functions on $M$ defined by the equations
$$\gi^{ka}(r\,q)\,=\,\delta^{ka}\,+\,r\,\hi^{ka}(r,q),\qquad
\nabla U(r\,q)\,=\,r^{l-1}\,\nabla U_l(q)\,+\,r^l\,V_{>l}(r,q),$$
$$\,\Fi_{ab}(r\,q)\,=\,r^{d-1}\,\Fi_{ab,\,\geq d}(r,q).$$
Since $\Delta\geq 1$ by the magnetism hypothesis \eqref{weak_magnetism} and the boundary $\partial M$ corresponds to the region $r=0$, the assertion follows immediately from \eqref{McGehee_eq} and we have the result.
\end{proof}

Note that the vector field $\tilde{X}$ in applications to the $N$-body problem is recovered from \eqref{McGehee_eq} with $l=-1$ and the $F$ corrections identically zero. However, in our case $l\geq 2$ and the $F$ corrections will be different from zero in general.

Recall the energy function on phase space
$$H(x,v)\,=\,g_x(v,\,v)/\,2\,+\,U(x),\qquad (x,v)\in TE.$$

\begin{prop}\label{Extension_energy}
On $M-\partial M$, consider the function
$$\tilde{H}\,=\, r^{-l}\,\pi^*(H)\,=\,r^{-l}\,H\circ \pi.$$
Then, $\tilde{H}$ is real analytic and extends uniquely and real analytically to $M$.
\end{prop}
\begin{proof}
By Lemma \ref{extension_Lemma}, on $M-\partial M$ the function $\tilde{H}$ reads as follows
\begin{equation}\label{Energy_rescaled}
\tilde{H}(r,q,y)\,=\,\frac{\Vert y\Vert^2}{2}\,+\,U_l(q)\,+\,r\,F_H(r,q,y)
\end{equation}
where $F_H$ is a globally defined real analytic function on $M$. In effect,
$$F_H(r,q,y)\,=\,u_{>l}(r,q)\,+\,\frac{\hi_{ab}(r,q)}{2}\,y^a\,y^b$$
where $\hi_{ab}$ and $u_{>l}$ are globally defined real analytic functions on $M$ defined by the equations
$$\gi_{ab}(r\,q)\,=\,\delta_{ab}\,+\,r\,\hi_{ab}(r,q),\qquad
U(r\,q)\,=\,r^l\,U_l(q)\,+\,r^{l+1}\,u_{>l}(r,q).$$
This completes the proof.
\end{proof}

For every point $\zeta$ in $M$, consider the unique solution $\varphi_\zeta$ of the following ODE defined on its maximal interval,
$$\varphi_\zeta: (\omega_-(\zeta),\,\omega_+(\zeta))\rightarrow M,\qquad \frac{d\,\varphi_\zeta}{d\tau} \,=\,\tilde{X}(\varphi_\zeta).$$
These solutions constitute a flow $\varphi(\tau, \zeta)\,=\,\varphi_{\zeta}(\tau)$ where it is defined and we will refer to it as the \emph{McGehee flow}. We will also denote $\varphi_\tau(\zeta)\,=\,\varphi(\tau, \zeta)$. Also, given a smooth function $w$ on $M$, we will make the usual abuse of notation by writing
$$w'\,=\,\partial_\tau\,(w\,\circ\,\varphi)\,|_{\tau=0}.$$

Recall that, by Proposition \ref{Extension_vector}, the boundary $\partial M$ is an invariant set of the flow.

\begin{lema}\label{Energy_calculation}
On the boundary $\partial M$, we have $\tilde{H}'\,=\,-l\,\nu\,\tilde{H}$.
\end{lema}
\begin{proof}
Since $\partial M$ is invariant, the calculation follows from \eqref{McGehee_eq} and \eqref{Energy_rescaled} at $r=0$:
$$\tilde{H}'\,=\,\langle y,\,y'\rangle\,+\,\langle\nabla U_l(q),\,q'\rangle\,=\,-\frac{l}{2}\,\nu\,\Vert y\Vert^2\,-\,\nu\,\langle\nabla U_l(q),\,q\rangle\,=$$
$$=\,-\frac{l}{2}\,\nu\,\Vert y\Vert^2\,-\,l\,\nu\,U_l(q)\,=\,-l\,\nu\,\tilde{H}$$
where the terms $\pm\,\langle\nabla U_l(q),\,y\rangle$ cancel out and the Euler theorem on homogeneous functions gives $\langle\nabla U_l(q),\,q\rangle\,=\,l\,U_l(q)$. This concludes the calculation.
\end{proof}

\begin{cor}\label{compact_invariant_boundaries}
Relative to the boundary $\partial M$, the (possibly empty) regions defined by $\tilde{H}=0$ and $\tilde{H}\leq 0$,
$$\partial M_0\,=\,\partial M \cap \tilde{H}^{-1}(0),\qquad \partial M_{\leq 0}\,=\,\partial M \cap \tilde{H}^{-1}(-\infty,0],$$
are compact invariant sets.
\end{cor}
\begin{proof}
From Lemma \ref{Energy_calculation}, it is clear that $\partial M_0$ is invariant. Thus, by uniqueness of the solution of the ODE in the mentioned lemma, $\partial M_{\leq 0}$ is invariant as well. Since $\tilde{H}$ is continuous, both of these regions are closed sets in $\partial M$. Moreover, they are contained in the compact set
$$C\,=\,\{0\}\times S^{n-1}\times \overline{B(\mathbf{0},\, R)},\qquad R^2/2\,=\,\max\{\,-m_U,\,0\}$$
where $m_U$ is the minimum of $U_l$ on the sphere $S^{n-1}$ and the assertion follows.
\end{proof}

\begin{defi}
We will refer to $\partial M_{< 0}\,=\,\partial M_{\leq 0}\,-\,\partial M_0$ as the subcritical boundary and to $\partial M_0$ as the critical boundary.
\end{defi}


It is worth to remark that the critical boundary $\partial M_0$ contrasts with its analogous \emph{collision manifold} counterpart in $N$-body problem. Indeed, while the collision manifold is non-compact, the critical boundary is, although not a manifold in general. As an example, consider a Newtonian mechanical system in the plane with potential $U(x_1,\,x_2)\,=\,-x_1^2$ whose critical boundary of the blowup at the origin is a doubly-pinched torus, see Figure \ref{perspectiva2_degenerado}.


\begin{lema}\label{boundary_non_empty}
Suppose that the critical point $p$ of the potential $U$ is not a strict minimum. Then, $\partial M_{\leq \rm 0}$ is non-empty.
\end{lema}
\begin{proof}
Suppose that $\partial M_{\leq \rm 0}$ is empty. Hence, by definition $\tilde{H}>0$ on $\partial M$. In particular, $U_l>0$ on the unit sphere $S^{n-1}$ hence its minimum $m_U>0$.

By Lemma \ref{extension_Lemma}, there is a real analytic function $u_{>l}$ on $[0,r_E)\times S^{n-1}$ such that,
$$U(rq)\,=\,r^l\,\left(U_l(q)\,+\,r\, u_{>l}(r,q)\right),\qquad
u_{>l}:[0,r_E)\times S^{n-1}\rightarrow\R.$$
In particular, $\vert u_{>l} \vert$ is uniformly continuous on $[0,r_E/2]\times S^{n-1}$ and has a maximum $M_U$ there. Denote by $M_U'$ the maximum of $\{1,\,M_U\}$ and define
$$\varepsilon\,=\,\min\,\{\,r_E/2,\,m_U/2M_U'\}.$$
Then, we have
$$U(rq)\,\geq \,r^l\,\left(m_U\,-\,r\,M_U\right)\,\geq\,r^{l}\,\frac{m_U}{2},\qquad (r,q)\in [0,\varepsilon]\times S^{n-1}.$$
Therefore, $p=0$ is a strict minimum of $U$ and we have the result.
\end{proof}

The converse is false in general. Indeed, although the origin is a strict minimum of the potential $x_1^2+x_2^4$ on the plane, the first non-zero jet $x_1^2$ is zero at $(0,\pm 1)\in S^1$ hence the critical boundary is non-empty.

We are in position now to define the main object of this paper.

\begin{prop}\label{McGehee_blowup_mechanical_system}
Suppose that the critical point $p$ of the potential $U$ is not a strict minimum and $U(p)=0$. Then, the restriction of the map $\pi$ to the critical and subcritical energy regions
$$\pi_{\leq 0}:\,\tilde{H}^{-1}(-\infty,0]\rightarrow H^{-1}(-\infty,0]$$
is a surjective continuous proper map with the following properties:
\medskip

\begin{enumerate}
\item\label{item_nonempty_boundary} $\pi_{\leq 0}^{-1}(\mathbf{0})\,=\, \partial M_{\leq \rm 0}\,\neq\,\emptyset$.
\medskip

\item\label{item_homeo_McGehee_blowup} The following restriction is a homeomorphism
$$\pi_{\leq 0}:\,\tilde{H}^{-1}(-\infty,0]\,-\,\partial M_{\leq 0}\rightarrow H^{-1}(-\infty,0]\,-\,\{\mathbf{0}\}.$$

\item The following restriction is a real analytic diffeomorphism
$$\pi_{< 0}:\,\tilde{H}^{-1}(-\infty,0)\,-\,\partial M_{< 0}\,\rightarrow\, H^{-1}(-\infty,0).$$

\item Dynamically, it maps an orbit to another orbit preserving their orientations and the resulting orbit map is surjective. In particular, the map in item \ref{item_homeo_McGehee_blowup} is a topological equivalence of flows.
\end{enumerate}
\end{prop}
\begin{proof}
Relative to the interior of $M$, by definition $\tilde{H}\,=\, r^{-l}\,H\circ \pi$ hence
$$\pi^{-1}(H^{-1}(-\infty,\,0])\cap \accentset{\circ}{M}\,=\,\tilde{H}^{-1}(-\infty,\,0])\cap \accentset{\circ}{M}.$$
Because, relative to the interior of $M$, the map $\pi$ is a diffeomorphism with image the complement of $T_p E$ in $TE$, we have
\begin{equation}\label{identidad_util_1}
\tilde{H}^{-1}(-\infty,\,0])\cap \accentset{\circ}{M}\,=\,\pi^{-1}\left(H^{-1}(-\infty,\,0]\cap (T_p E)^c\right).
\end{equation}
In particular, $$\pi\left(\tilde{H}^{-1}(-\infty,\,0])\cap \accentset{\circ}{M}\right)\subset H^{-1}(-\infty,0]$$
and because
$$\tilde{H}^{-1}(-\infty,\,0])\cap \partial M\subset \partial M\,\xrightarrow{\pi}\,\{\mathbf{0}\}\subset H^{-1}(-\infty, 0]$$
we conclude that the restriction $\pi_{\leq 0}$ is well defined and continuous.

Now, we prove the properties listed in the items.
\medskip

\begin{enumerate}
\item It follows directly from the definitions and Lemma \ref{boundary_non_empty},
$$\pi_{\leq 0}^{-1}(\mathbf{0})\,=\, \tilde{H}^{ -1} (-\infty,0]\cap \pi^{ -1}(\mathbf{0}) \,=\, \tilde{H}^{ -1}(-\infty,0]\cap \partial M\,=\,\partial M_{\leq  0}\,\neq\,\emptyset.$$
\bigskip

\item By \eqref{identidad_util_1} and the fact that $\pi$ is a diffeomorphism outside $\partial M$, the following map is a homeomorphism
$$\pi_{\leq 0}:\,\tilde{H}^{-1}(-\infty,0]\,\cap \,\accentset{\circ}{M}\rightarrow H^{-1}(-\infty,0]\,\cap (T_p E)^c$$
and the item follows directly from the equality
$$H^{-1}(-\infty,0]\,\cap T_p E\,=\,\{\mathbf{0}\}.$$
This is because $H(p,y)\,=\,y^2/2\,+\,U(p)\,=\,y^2/2\geq 0$.
\bigskip

\item A completely similar argument as the one before for the interval $(-\infty,0)$, shows that the following map is a real analytic diffeomorphism
$$\pi_{< 0}:\,\tilde{H}^{-1}(-\infty,0)\,\cap \,\accentset{\circ}{M}\rightarrow H^{-1}(-\infty,0)\,\cap (T_p E)^c$$
and the result immediately follows from the inclusion
$$H^{-1}(-\infty,0)\,\subset\,( T_p E)^c.$$

\smallskip

\item By item \ref{item_nonempty_boundary} and the fact that $\mathbf{0}$ is a fixed point in $TE$, every orbit in $\partial M_{\leq 0}$ is mapped to the constant orbit $\{\mathbf{0}\}$. Relative to the interior of $M$, by definition $r^{l/2-1}\,\pi_*(\tilde{X})\,=\,X$ and due to the fact that
$$\tilde{H}^{-1}(-\infty,\,0]\cap \accentset{\circ}{M}\,\subset\,M,\qquad
H^{-1}(-\infty,0]\,-\,\{\mathbf{0}\}\,\subset\,TE,$$
are invariant sets, the result follows from Lemma \ref{flow_equivalence}.
\end{enumerate}
\medskip

It rest to show that the map $\pi_{\leq 0}$ is surjective and proper. Surjectivity of the map follows directly from Lemma \ref{boundary_non_empty} and items \ref{item_nonempty_boundary} and \ref{item_homeo_McGehee_blowup}. For properness, consider a compact set $K\subset H^{-1}(-\infty,0]$ and a sequence $(\zeta_n)$ in $K'=\pi_{\leq 0}^{-1}(K)$. It will be enough to show that there is a convergent subsequence of $(\zeta_n)$ in $K'$. Denote by $(z_n)$ the sequence $z_n=\pi_{\leq 0}(\zeta_n)\in K$. Taking a subsequence if necessary, we may suppose that $z_n\rightarrow z\in K$. If $z\neq 0$, then by item \ref{item_homeo_McGehee_blowup} we have the result. Suppose instead that $z_n\rightarrow 0\in K$. In particular,
$$r_n\rightarrow 0,\qquad \zeta_n=(r_n,q_n,y_n),\qquad \tilde{H}(\zeta_n)\leq 0.$$
Because $q_n\in S^{n-1}$, taking a subsequence we may suppose that $q_n\rightarrow q_*\in S^{n-1}$. Therefore, by equation \eqref{Energy_rescaled} we have
$$0\,\leq\, \frac{\Vert y_n \Vert^2}{2}\,\leq\, -U_l(q_n)\,-\,r_n\,F_H(r_n,\,q_n)\,\rightarrow\,-U_l(q_*)$$
hence the sequence $(y_n)$ is bounded. Thus, taking a further subsequence if necessary, we may suppose that $y_n\rightarrow y_*$ and we have proved that
$$\zeta_n=(r_n,q_n,y_n)\rightarrow (0,q_*,y_*).$$
Since $K'$ is closed, $(0,q_*,y_*)\in K'$ and this finishes the proof.
\end{proof}

\begin{defi}
We will refer to the map in Proposition \ref{McGehee_blowup_mechanical_system} as the McGehee blowup for the Lagrangian system at $p$.
\end{defi}

As an example, Figure \ref{pendulum_blowup} shows the McGehee blowup for the simple pendulum at the unstable equilibrium point. The subcritical boundary is the disjoint union of two open intervals while its boundary, the critical boundary, consists of four points. Note that the blowup has replaced the unstable single equilibrium point by the disjoint union of two segments.

Dynamically, the subcritical boundary consists of two heteroclinic orbits connecting pairs of fixed points in the critical boundary. The original homoclinic orbits have become heteroclinic after the blowup. The map $\pi_{\leq 0}$ identifies the critical and subcritical boundary into a single point and recovers the original space along with its dynamics in the energy region $H\leq 0$.

\begin{figure}[htb]
\centering
\includegraphics[width=0.5\textwidth]{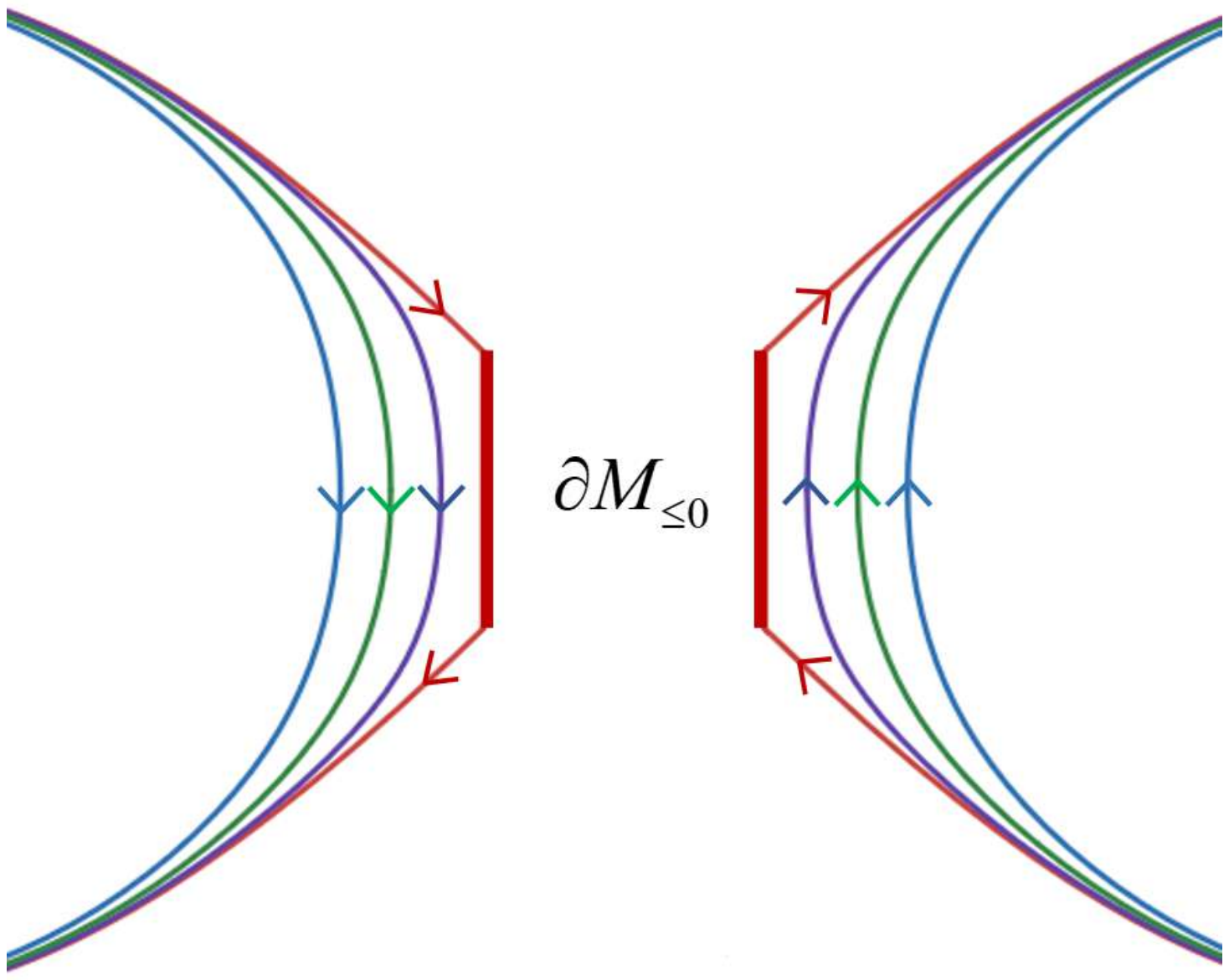}
\caption{McGehee blowup for the simple pendulum at the unstable equilibrium point. The thick red lines represent the critical and subcritical boundary.}
\label{pendulum_blowup}
\end{figure}

\section{Dynamics on the subcritical and critical boundary}\label{Dynamics_criticalboundary_section}

In this section, it will be proved that the McGehee flow on the critical and subcritical boundary is a rich dynamical object, actually, a Lyapunov dynamical system.

\begin{lema}\label{Lyapunov_calculation}
On the boundary $\partial M$, we have $\nu'\,=\,\left(1+l/2\right)\,\Vert y_{\,tg}\Vert^2\,-\,l\,\tilde{H}$.
\end{lema}
\begin{proof}
Since $\partial M$ is invariant, the calculation follows from \eqref{McGehee_eq} and \eqref{Energy_rescaled} at $r=0$. Actually,
$$\nu'\,=\,\langle q',\,y\rangle\,+\,\langle q,\,y'\rangle\,=\,\Vert y\Vert^2\,-\,\nu^2\,-\,\langle q,\,\nabla U_l(q)\rangle\,-\,\frac{l}{2}\nu^2\,=$$
$$=\,\Vert y_{\,tg}\Vert^2\,-\,l\, U_l(q)\,-\,\frac{l}{2}\nu^2\,=\,\left(1+l/2\right)\,\Vert y_{\,tg}\Vert^2\,-\,l\,\tilde{H}$$
where we have used \eqref{y_tg_identity} and the Euler theorem on the homogeneous function $U_l$. This concludes the calculation.
\end{proof}

\begin{prop}\label{Lyapunov_prop}
On the compact invariant set $\partial M_{\leq 0}$, the fixed point set is contained in $\partial M_0$ and the function $\nu$ is a strict Lyapunov function outside it.
\end{prop}
\begin{proof}
By Lemma \ref{Lyapunov_calculation} and the fact that $l\geq 2$, the function $\nu$ is increasing on orbits hence it is a Lyapunov function. Thus, it is enough to show that the function is strict outside the fixed point set in $\partial M_0$.

Suppose it is not. Then by definition, there is a point
$$\zeta\,\in\,\partial M_{\leq 0}\,-\,\mbox{Fix}\,(\partial M_0)$$
and there is an $\varepsilon>0$ such that $\nu(\zeta)\,=\,\nu(\varphi_{\varepsilon}(\zeta))$. Because $\nu$ is increasing on orbits, $\tau\mapsto\nu(\varphi_{\tau}(\zeta))$ is constant on the interval $(0,\varepsilon)$. Therefore $\nu'$ is zero on this interval and by Lemma \ref{Lyapunov_calculation} again, we have that $y_{\,tg}$ is null and $\tilde{H}$ is zero along the orbit on the aforementioned interval. By the motion equations \eqref{McGehee_eq} at $r=0$,
$$q'\,=\,y\,-\,\nu\,q\,=\,y_{\,tg},\qquad
y'\,=\,-\nabla U_l\,(q)\,-\,\frac{l}{2}\,\nu\,y,$$
we have that $q$ is constant hence so is $y\,=\,y_{\,tg}\,+\,\nu\,q$ along the orbit on $(0,\varepsilon)$. Since the solution is unique, $\zeta$ must be a fixed point such that $\tilde{H}(\zeta)=\rm 0$, that is
$$\zeta\,\in\,\mbox{Fix}\,(\partial M_0)$$
and this is a contradiction since it was assumed not by hypothesis.
\end{proof}

In particular, the fixed point set in $\partial M_{\leq 0}$ coincides with the one in $\partial M_0$ and the $\alpha$ and $\omega$ limit of every orbit in $\partial M_{\leq 0}$ are contained in different connected components of this fixed point set. In view of this and the previous proposition, it will be useful to have a concrete description of this set.

\begin{lema}\label{concrete_description_fixpoint}
Relative to $\partial M_0$, the point $\zeta$ is a fixed point of the flow if and only if $\zeta\,=\,\rm (\,0,\it \,q,\,\nu\,q)$ such that $q$ is a critical point of $U_l$ relative to the unit sphere with Lagrange multiplier $-l\,\nu^2/2$,
$$\nabla U_l\,(q)\,=\,-\frac{l\,\nu^2}{2}\,q.$$
Equivalently, for every critical point $q$ of $U_l$ relative to the unit sphere verifying $U_l(q)\leq 0$, we have the fixed points in $\partial M_0$
\begin{equation}\label{fixedpoints_associatedto_criticalpoint}
\zeta\,=\,\rm \left(\,0,\it \,q,\,\rm \pm(-2\,\it U_l(q) \rm )^{1/2}\, \it q\right)
\end{equation}
and these are all.
\end{lema}
\begin{proof}
This follows immediately from equations \eqref{McGehee_eq} and \eqref{Energy_rescaled} at $r=0$ and the fact that $U_l$ is homogeneous.
\end{proof}

In general, the Lyapunov function of a Lyapunov dynamical system does not need to be constant on the connected components of the chain-recurrent set since it could take different values on different chain-recurrent classes. As an easy example, consider the flow $\phi_t(x,\,y)\,=\,(x,\,e^t\,y)$ on the plane and the complete Lyapunov function $(x,y)\mapsto x+y^2$.

\begin{cor}\label{Lyapunov_constant}
Restricted to the fixed point set in $\partial M_0$, the function $\nu$ is locally constant and has only a finite set of values.
\end{cor}
\begin{proof}
Denote by $f$ the restriction of $U_l$ to the sphere $S^{n-1}$. Because $f$ is a real analytic function on the compact real analytic manifold $S^{n-1}$, by general theory the function $f$ is locally constant on its critical locus and has a finite set of critical values. From the fact that $\nu$ is continuous and the previous lemma we have
$$\nu(0,\,q,\,y)\,=\,\pm(-2\, f(q) )^{1/2},\qquad (0,\,q,\,y)\,\in\,\mbox{Fix}\,(\partial M_0),$$
the result follows.
\end{proof}

A few words about the previous proof. The fact that a real analytic function is locally constant on its critical locus is an immediate consequence of the non-trivial Lojasiewicz inequality. It was proved in \cite{Soucek} that a real analytic function on a compact set takes a finite set of values and since the sphere can be covered by a finite amount of precompact real analytic charts, the result holds for the sphere as well. An alternative proof using the curve selection lemma can be found in \cite{LeDungTrang}.

In Lemma \ref{concrete_description_fixpoint}, we will call the critical boundary fixed points in \eqref{fixedpoints_associatedto_criticalpoint} by the name of \emph{fixed points associated to the critical point q}.

\begin{lema}\label{heteroclinic_orbit}
For every critical point $q$ of $U_l$ relative to the unit sphere with $U_l(q)< 0$, there is a subcritical heteroclinic orbit connecting its associated fixed points. Concretely, the function
$$\tau\,\mapsto\,(0,\,q,\,\mathscr{Y}(\tau)\,q),\qquad \mathscr{Y}'=-l\,\left(\mathscr{Y}^2/2\,+\,U_l\,(q)\right)$$
is a heteroclinic orbit.
\end{lema}
\begin{proof}
This follows by direct substitution from equations \eqref{McGehee_eq} at $r=0$.
\end{proof}

\subsection{Example in the plane}\label{example_plane}

Consider a Newtonian Lagrangian on $T\R^2$ with potential $U:\R^2\rightarrow\R$ such that the origin is a critical point of the latter. Suppose that the potential is zero at the origin and its first non-zero jet at this point has degree two. Consider the energy region $H\leq 0$ in phase space $T\R^2$.

The McGehee blowup of this mechanical system at the origin consists of replacing the origin in the mentioned energy region with one of the critical and subcritical boundaries shown below.

Denote by $f:S^1\rightarrow\R$ the restriction to the unit circle of the first non-zero jet $U_2$ of the potential at the origin. Up to some rotation on the circle, without loss of generality we may suppose that
$$f(x_1,x_2)\,=\,a\, x_1^2\,+\,b\, x_2^2,\qquad (x_1,x_2)\in S^1.$$

Up to interchanging $x_1$ with $x_2$, there are six cases to distinguish. The case $0<a\leq b$ has empty critical and subcritical boundary while this set consist of two single fixed points in the case $0=a< b$. There are four remaining cases, two of which are non-degenerated in the sense that $f$ is a Morse function with zero regular value.

Because the degree of the first non-zero jet is even, by equations \eqref{McGehee_eq} at $r=0$, the dynamical system on the critical and subcritical boundary has the $\Z_2$ symmetry given by $(x,y)\mapsto \pm (x,y)$.

The case $a<0<b$ is non-degenerated in the sense explained above and its critical and subcritical boundary is a disjoint union of two compact balls. The dynamics on the critical boundary is conjugated to a gradient system on each of the spherical boundaries, with only one source and one sink on each of the spheres. This is illustrated in Figure \ref{perspectiva2_disjunto}. Now, besides the fixed points, the orbits in the critical boundary are heteroclinic orbits connecting associated fixed points (blue).

\begin{figure}[htb]
\centering
\includegraphics[width=0.425\textwidth]{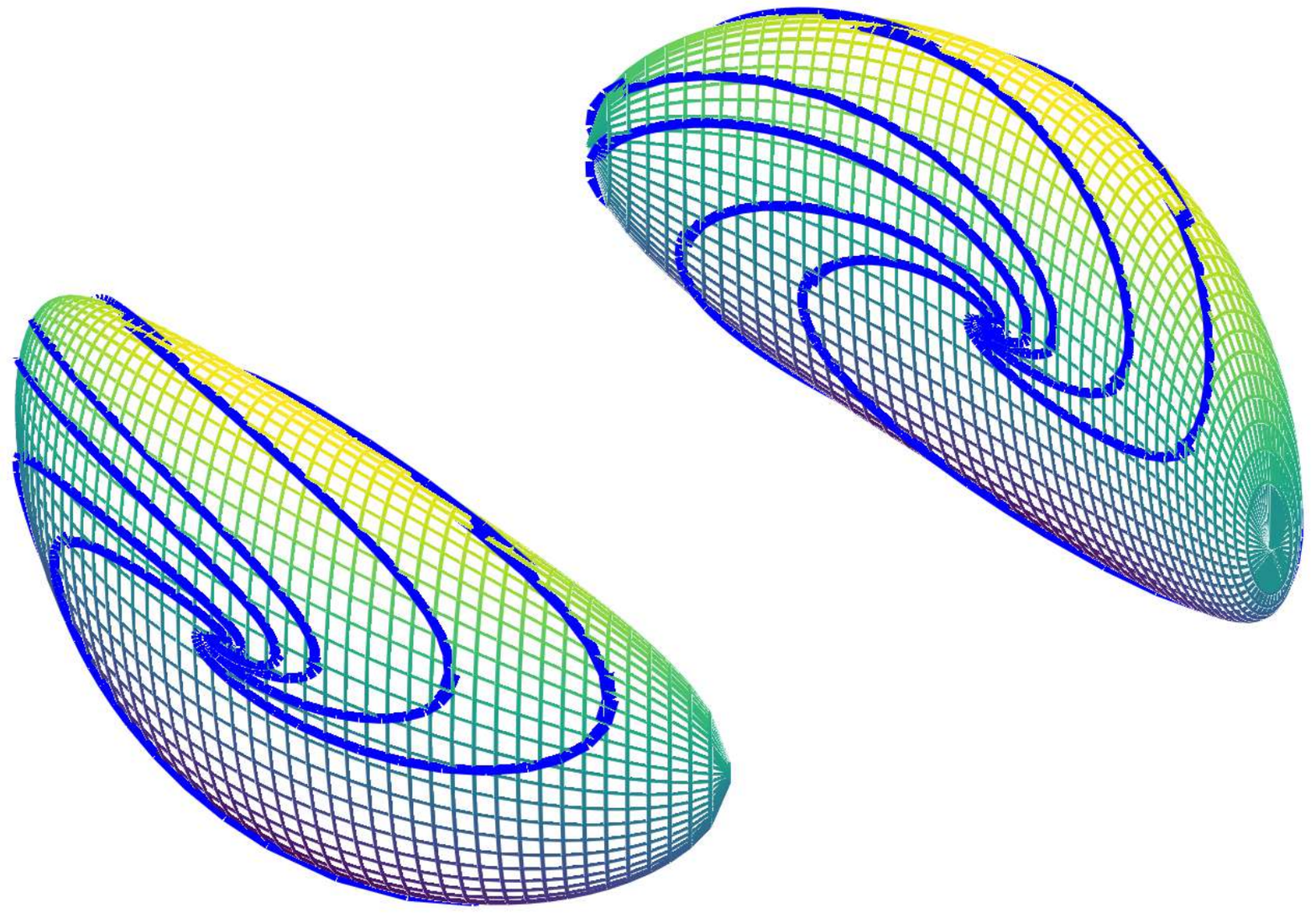}
\caption{Critical boundary. The blue orbits are heteroclinic orbits connecting associated fixed points.}
\label{perspectiva2_disjunto}
\end{figure}

The case $a<b=0$ is degenerated and its critical and subcritical boundary is a solid doubly pinched torus. There are regions clearly delimited by separatrices and these regions are filled with heteroclinic orbits connecting associated fixed points. The separatrices are the only heteroclinic orbits connecting non-associated fixed points. This is illustrated in Figure \ref{perspectiva2_degenerado}.

\begin{figure}[htb]
\centering
\includegraphics[width=0.4\textwidth, height=0.28\textwidth]{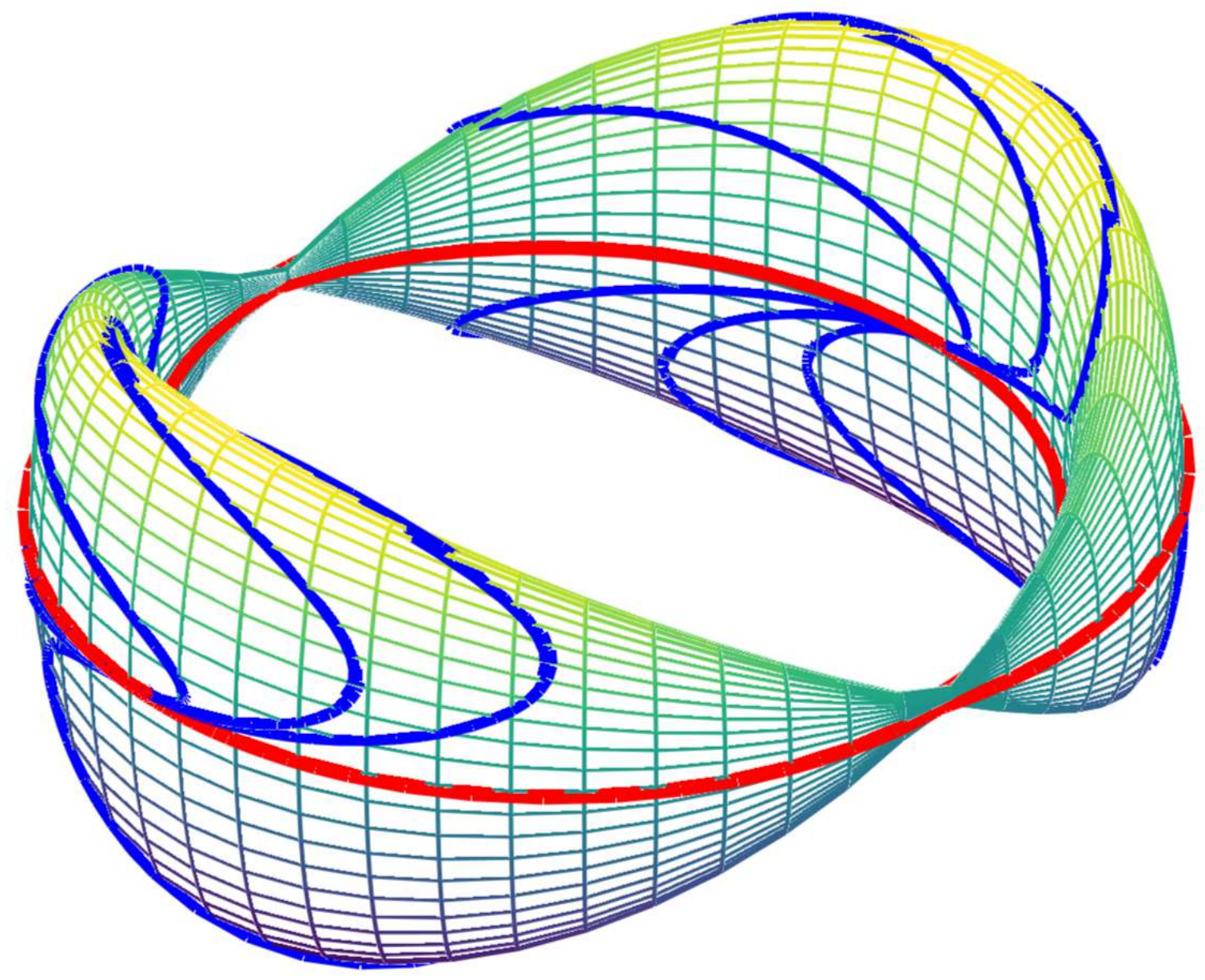}
\caption{Critical boundary. The red orbits are separatrices and the blue orbits are heteroclinic orbits connecting associated fixed points.}
\label{perspectiva2_degenerado}
\end{figure}

The case $a<b<0$ is also non-degenerated in the sense explained before and the critical and subcritical boundary is a solid torus. The dynamics on the critical boundary is conjugated to a double cover of the typical example of Morse-Smale system on the torus. This is illustrated in Figure \ref{perspectiva2_toro}. As in the previous case, there are regions clearly delimited by separatrices and these regions are filled with heteroclinic orbits. Up to the $\Z_2$ symmetry described before, a couple of these regions contains heteroclinic orbits connecting associated fixed points while the other couple contains heteroclinic orbits connecting non-associated fixed points.

\begin{figure}[htb]
\centering
\includegraphics[width=0.4\textwidth]{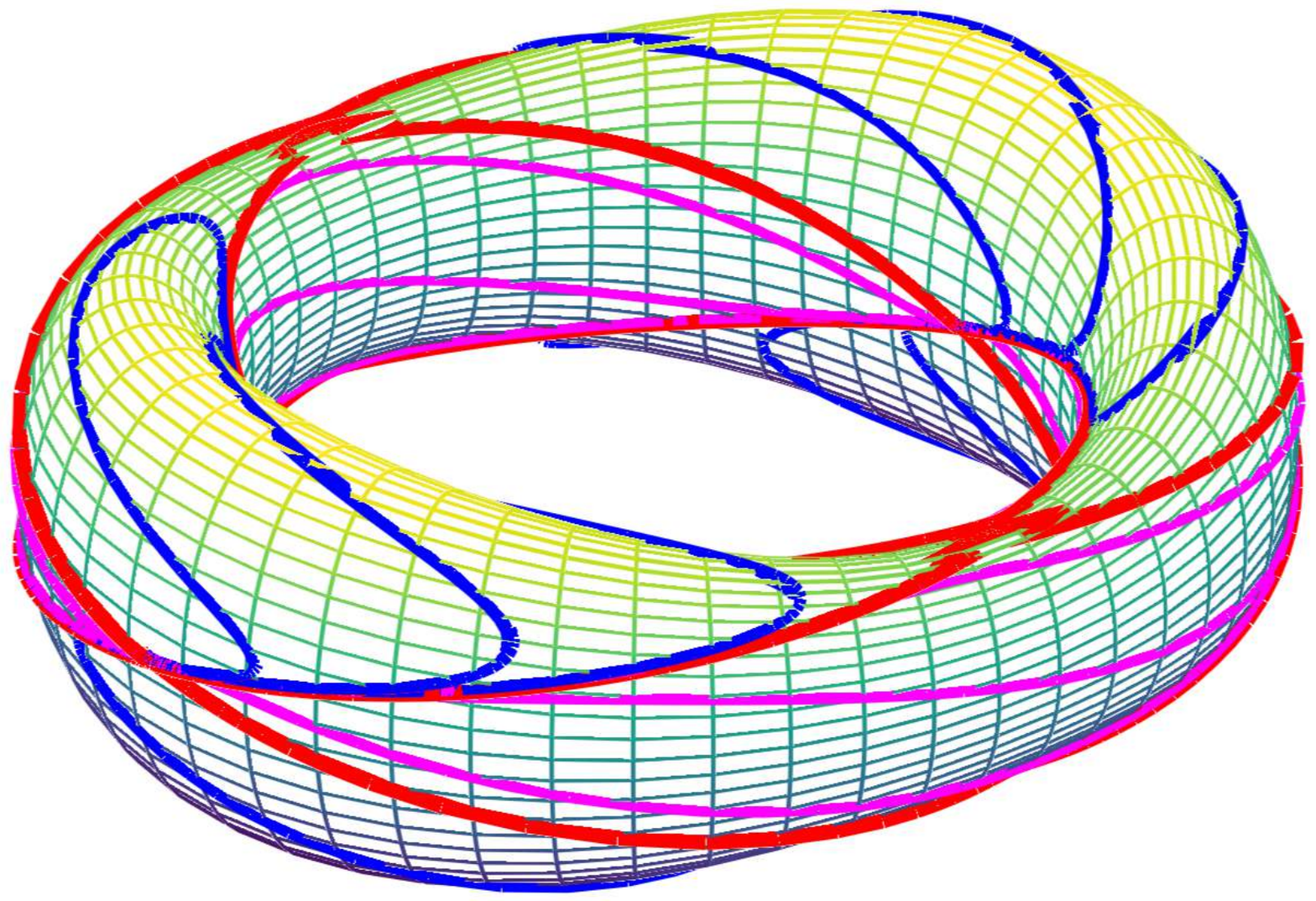}
\caption{Critical boundary. The red orbits are separatrices, the blue orbits are heteroclinic orbits connecting associated fixed points and the magenta orbits are heteroclinic orbits connecting non-associated fixed points.}
\label{perspectiva2_toro}
\end{figure}

Finally, the case $a=b<0$ is degenerated and its critical and subcritical boundary is a solid torus. Now, there is a pair of continua of equilibrium points, each one an embedded copy of the circle, both connected by heteroclinic orbits. This is illustrated in Figure \ref{perspectiva2_continuo}.

\begin{figure}[htb]
\centering
\includegraphics[width=0.43\textwidth, height=0.21\textwidth]{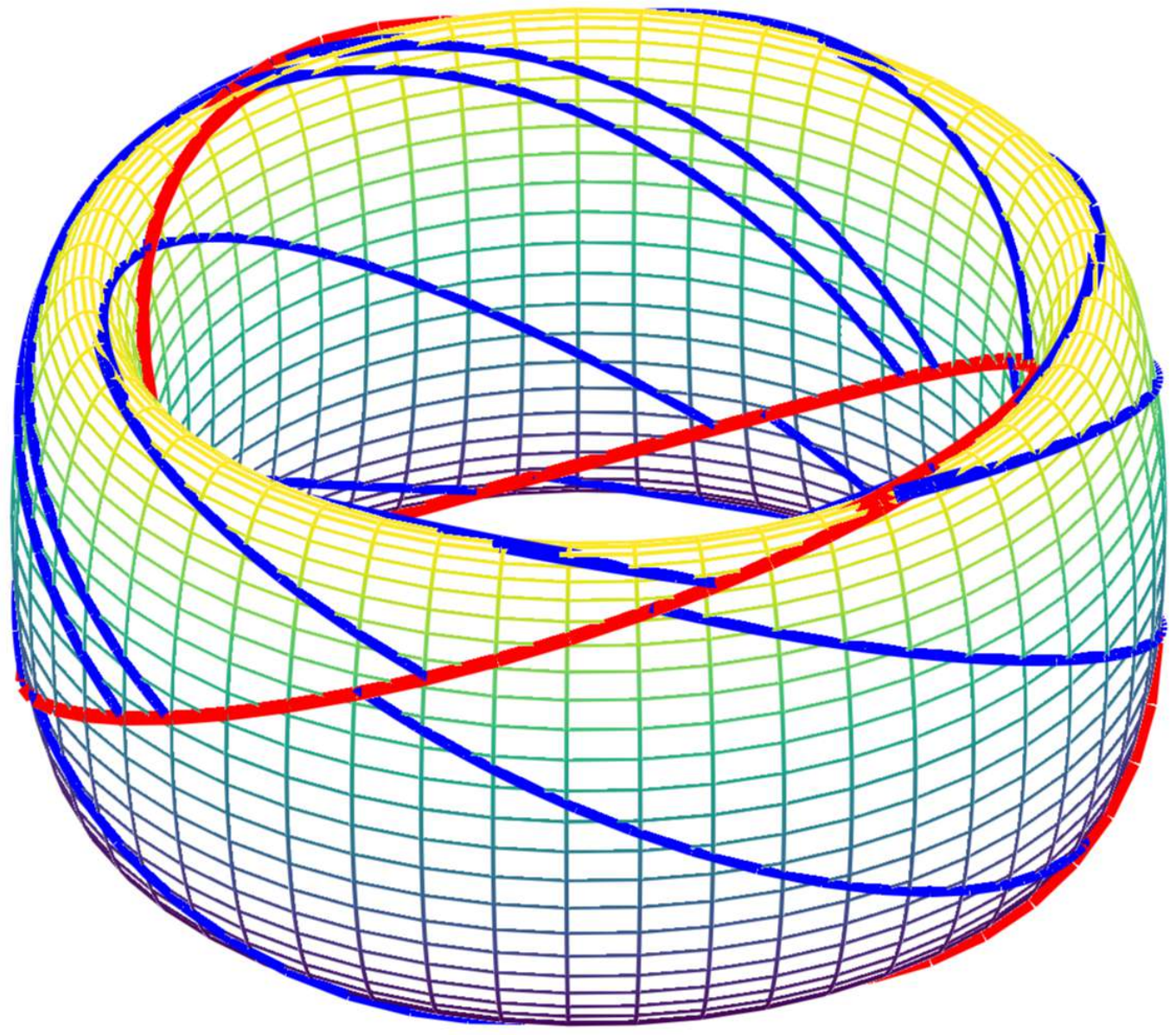}
\caption{Critical boundary. The red lines are fixed points continua and the blue orbits are heteroclinic orbits connecting these.}
\label{perspectiva2_continuo}
\end{figure}

\section{Superior limit of sets: topological preliminaries}\label{limsup}

In what follows, we will abuse of notation and define the superior limit of a sequence of subsets $(A_n)_{n\in \N}$ of a topological space as
\begin{equation}\label{limsup_definition_eq}
\limsup_n\,A_n\,=\,\bigcap_{k\in\N}\,\overline{\bigcup_{n\geq k}\,A_n}.
\end{equation}
The difference with the usual notion is taking the closure before the intersection.

This section is preliminary, and the results presented here are well-known in the literature, see \cite{Cech}. However, some of these results may be difficult to trace, either because they are found in different sources or because they are embedded in very general formalisms. Since the object in question, the superior limit, plays a key role in the proof of the main theorem, except for the standard one regarding the Hausdorff metric, the results presented here will be proved. The presentation is self-contained and suitable for our purposes.

\begin{lema}\label{general_properties}
Consider a sequence of subsets $(A_n)_{n\in \N}$ of a metric space $(X,\,d)$ and its superior limit
$$A\,=\,\limsup_n\,A_n.$$
Then, the following assertions hold.
\medskip

\begin{enumerate}
\item $A$ is a closed set.
\medskip

\item\label{subseq} $x$ belongs to $A$ if and only if there is a strictly increasing sequence $(n_k)_{k\in\N}$ such that for every natural $k$ there is $x_k$ in $A_{n_k}$ such that $x_k\rightarrow x$.
Equivalently, $A$ consists of all the limit points of all the sequences $(x_n)_{n\in\N}$ such that $x_n\in A_n$ for every natural $n$.
\medskip

\item\label{nonempty_compact} If there is a compact set $K$ in $X$ containing every set $A_n$ and every $A_n$ is non-empty, then $A$ is a non-empty compact set in $K$.
\medskip

\item\label{connected} If there is a compact set $K$ in $X$ containing every set $A_n$, every $A_n$ is connected and there is a converging sequence $(x_n)_{n\in\N}$ such that $x_n\in A_n$ for every natural $n$, then $A$ is connected.
\end{enumerate}
\end{lema}
\begin{proof}
\begin{enumerate}
\item $A$ is intersection of closed sets.
\bigskip

\item This is a diagonal argument. Indeed, suppose that $x\in A$, that is for every natural $j$ the point $x$ is an accumulation point of $\bigcup_{n\geq j}\,A_n$. Let $B_m=B(x,\,1/m)$. We will construct the sequence by induction.

\noindent For $m=1$, there is
$$x_1\,\in\, B_1\cap\left(\bigcup_{n\geq 1}\,A_n\right)\,=\,\bigcup_{n\geq 1}\,(B_1\cap A_n)$$
hence there is a natural $n_1$ such that $x_1\in B_1\cap A_{n_1}$.

\noindent Now suppose that we have constructed the first $m$ terms. There is
$$x_{m+1}\,\in\, B_{m+1}\cap\left(\bigcup_{n\geq n_{m}+1}\,A_n\right)\,=\,\bigcup_{n\geq n_{m}+1}\,(B_{m+1}\cap A_n)$$
hence there is a natural $n_{m+1}>n_{m}$ such that $x_{m+1}\in B_{m+1}\cap A_{n_{m+1}}$.

\noindent It is clear that $x_k\in A_{n_k}$ for every natural $k$ and $x_k\rightarrow x$.
\smallskip

\noindent For the converse we have that for every natural $L$,
$$x\in\overline{\{x_k\ |\ k\geq L\}}\subset\overline{\bigcup_{k\geq L}\,A_{n_k}}\subset \overline{\bigcup_{n\geq n_{L}}\,A_{n}}$$
hence since $(n_L)$ is increasing, we have that $x$ is in $A$.
\bigskip

\item Since the metric space is Hausdorff, $K$ is closed hence it clearly contains the set $A$. Since $A$ is closed and contained in a compact set, $A$ is compact as well. Take any sequence $(x_n)$ such that $x_n\in A_n$. Since the sequence is contained in the compact $K$, it has a limit point that belongs to $A$ by the previous item hence it is non-empty.
\bigskip

\item Suppose that $A$ is not connected. Then, there are open sets $V_1$ and $V_2$ in $X$ whose union contains $A$ and their intersection with $A$ is non-empty. Without loss of generality, we may suppose that the converging sequence $(x_n)$ in the hypotheses converges to some point $x$ in $A\cap V_1$. There is a point $x'$ in $A\cap V_2$ since it is non-empty. By item \ref{subseq}, there is a strictly increasing sequence $(n_k)$ such that for every natural $k$ there is $x'_k$ in $A_{n_k}$ such that $x'_k\rightarrow x'$. Taking a subsequence if necessary, we may suppose that $x_{n_k}$ is in $V_1$ and $x'_k$ is in $V_2$ for every natural $k$. In particular, the intersections $A_{n_k}\cap V_1$ and $A_{n_k}\cap V_2$ are non-empty for every natural $k$. Thus, since $A_{n_k}$ is connected, for every natural $k$ there is a point $x''_k$ in $A_{n_k}-(V_1\cup V_2)$. Because the sequence $(x''_k)$ is in the compact $K$, taking a further subsequence if necessary we may suppose that the sequence $(x''_k)$ converges to some point $x''$ necessarily contained in $A-(V_1\cup V_2)$ and this is absurd. This completes the proof.
\end{enumerate}
\end{proof}

Denote by $X^*$ the set of non-empty compact sets of $X$ and recall the Hausdorff distance
$$d_H\,(C_1,\,C_2)\,=\,\max\left\lbrace\,\max_{x\in C_1}\,d(x,C_2),\ \max_{y\in C_2}\,d(y,C_1)\,\right\rbrace,\qquad C_1,\,C_2\,\in\,X^*.$$
Here, $d(x,\,C)$ denotes the minimum of the distances from $x$ to some point in the compact set $C$. The following proposition is a standard result in the theory and we omit the proof.

\begin{prop}\label{Hausdorff_metric}
The space $(X^*,\,d_H)$ is a metric space. If $X$ is complete, then so is $X^*$. Moreover, if $X$ is compact, then so is $X^*$.
\end{prop}

\begin{lema}\label{Hausdorff}
Suppose that $X$ is compact. If $(C_n)_{n\in\N}$ is a sequence in $X^*$ converging to $C$ under the Hausdorff distance, then $C$ coincides with the set of limit points of converging sequences $(x_n)_{n\in\N}$ such that $x_n\in C_n$ for every natural $n$.
\end{lema}
\begin{proof}
Denote by $L$ the set of limit points of converging sequences $(x_n)$ such that $x_n\in C_n$ for every natural $n$. The set $L$ is closed since $y_m\xrightarrow{d} y$ with $y_m$ in $L$ implies that $y$ is also in $L$ by a diagonal argument. Since $L$ is a subset of $X$ and the last is compact by hypothesis, $L$ is compact as well.

We claim that $C\subset L$. In particular, $L$ is a non-empty compact set hence it is a point in $X^*$. Let $x$ in $C$. Because each $C_n$ is a non-empty compact set, for every $n$ there is a point $x_n$ in $C_n$ such that $d(x,x_n)\,=\,d(x,C_n)$. By definition,
$$d(x,\,x_n)\,=\,d(x,\,C_n)\,\leq\,d_H(C,\,C_n)\xrightarrow{n}\, 0$$
hence the sequence $(x_n)$ converges to $x$ and we conclude that $x$ belongs to $L$. This proves the claim.

Suppose that $L\neq C$. Then, $d\,=\,d_H(C,L)>0$ and there is a natural $n_0$ such that $d_H(C,C_n)<d/2$ for every $n\geq n_0$. Because $C\subset L$, we have that
$$d\,=\,d_H(C,L)\,=\,\max_{x\in L}\,d(x,C)\,=\,d(x_*,C)$$
for some point $x_*$ in $L$. For every natural $n$ there are points $y_n\,\in\,C_n$ and $z_n\,\in\,C$ such that $d(x_*,C_n)\,=\,d(x_*,y_n)$ and $d(y_n,C)\,=\,d(y_n,z_n)$ hence

$$d(x_*,C_n)\,=\,d(x_*,y_n)\,\geq\,d(x_*,z_n)\,-\,d(y_n,z_n)\,\geq\,d(x_*,C)\,-\,d(y_n,C)$$
$$\,\geq\,d(x_*,C)\,-\,d_H(C_n,C)\,>\,d/2$$
for every $n\geq n_0$. In particular, by definition of $L$, the point $x_*$ does not belong to $L$ and this is a contradiction.
\end{proof}

\section{Dynamics on the superior limit of sets}

Suppose that we have a flow $\varphi_t$ of homeomorphisms on the metric space $(X,d)$. In this section, all of the dynamical statements will be with respect to this flow. As defined in section \ref{limsup}, consider a sequence of subsets $(A_n)_{n\in \N}$ of $X$ and denote by $A$ its superior limit,
$$A\,=\,\limsup_n\,A_n.$$

In this section, we will be concerned with the dynamical properties the superior limit $A$ inherits from those of the sets $A_n$.

\begin{lema}\label{invariant}
If every $A_n$ is a (positive) invariant set, then so is $A$.
\end{lema}
\begin{proof}
Let $x$ in $A$. By item \ref{subseq} in Lemma \ref{general_properties}, there is a strictly increasing sequence $(n_k)$ such that for every natural $k$ there is $x_k$ in $A_{n_k}$ such that $x_k\rightarrow x$. Therefore, for every (positive) real number $t$ we have that $\varphi_t(x_k)$ belongs to $A_{n_k}$ and $\varphi_t(x_k)\rightarrow \varphi_t(x)$. Because of the aforementioned item, we conclude that $\varphi_t(x)$ belongs to $A$.
\end{proof}

Even if every set $A_n$ in Lemma \ref{invariant} were recurrent, the superior limit $A$ does not need to be so. For instance, take a sequence of subcritical periodic orbits in the simple pendulum phase space whose sequence of energies converge to the critical Mañé value and note that the superior limit is the union of the unstable equilibrium point with the pair of homoclinic orbits and this is clearly not a recurrent set.

However, if every set $A_n$ is recurrent, then $A$ is chain-recurrent. We say that an invariant set is chain-recurrent if every one of its points is so.

\begin{lema}[Weak statement]\label{weak_statement}
If every $A_n$ is recurrent, then $A$ is chain-recurrent.
\end{lema}
\begin{proof}
Let $x$ in $A$ and consider arbitrary $\varepsilon,\,T\,>0$. Since $\varphi_{-T}$ is continuous, there is $0<\varepsilon_2\leq 2\,\varepsilon$ such that $\varphi_{-T}\left(B(\varphi_T(x),\,\varepsilon_2)\right)\subset B(x,\,\varepsilon)$. By item \ref{subseq} in Lemma \ref{general_properties}, there is a point $y$ in $A_n$ for some natural $n$ such that $d(\varphi_T(x),\, y)\,<\,\varepsilon_2/2$. Because $A_n$ is recurrent, there is $t_*\geq 2\,T$ such that $d(y,\,\varphi_{t*}(y))\,<\,\varepsilon_2/2$. In particular, $\varphi_{t*}(y)\,\in\,B(\varphi_T(x),\,\varepsilon_2)$ hence
$$\varphi_{t*-T}(y)\,=\,\varphi_{-T}\left(\varphi_{t*}(y)\right)\,\in\,\varphi_{-T}\left(B(\varphi_T(x),\,\varepsilon_2)\right)\subset B(x,\,\varepsilon).$$
We have proved that $(x,\,y,\,x;\,T,\,t*-T)$ is an $(\varepsilon,\,T)$-chain from $x$ to itself. Since $x$ was arbitrary, we conclude that every point in $A$ is chain-recurrent.
\end{proof}

We will need the superior limit to verify an additional property. Denote by $\mathcal{R}(\varphi_t)$ the chain-recurrence set of the flow $\varphi_t$, that is the set of all the chain-recurrent points of the flow.

\begin{defi}
An invariant set $Z\subset X$ is chain-recurrent with the restriction property if $\mathcal{R}\left(\varphi_t|_Z\right)\,=\,Z$.
\end{defi}

It is well known that the chain-recurrence set of a flow is chain-recurrent with the restriction property. However, there are chain-recurrent sets not verifying the restriction property. For instance, consider a pair of points $a,\,b$ and a pair of heteroclinic orbits, one going from $a$ to $b$ and the other going from $b$ to $a$. Now, consider the invariant subset consisting of the union of $a$ and $b$ together with one of the heteroclinic orbits. Clearly the invariant subset is chain-recurrent since every one of its points is so. Nevertheless, it is not chain-recurrent as a dynamical system itself, that is it does not verify the restriction property.

\begin{lema}\label{restricted_chain_recurrence_limit}
If there is a compact set $K$ in $X$ containing every set $A_n$ and every $A_n$ is chain-recurrent with the restriction property, then so is $A$.
\end{lema}
\begin{proof}
It is a non-trivial fact that a point $x$ is chain-recurrent for the flow $\varphi_t$ if and only it is so for the associated time one map $\varphi\,=\,\varphi_1$. Therefore, it will be enough to prove the result for the time one map $\varphi$.

Let $x$ in $A$. By item \ref{subseq} in Lemma \ref{general_properties}, there is a strictly increasing sequence $(n_k)$ such that for every natural $k$ there is $x_k$ in $A_{n_k}$ such that $x_k\rightarrow x$. For each natural $k$, there is a $k^{-1}$-chain from $x_k$ to itself contained in $A_{n_k}$,
\begin{equation}\label{chain}
x_k\,=\,x_{k,0},\,x_{k,1},\,\ldots,\,x_{k,j_k-1},\,x_{k,j_k}\,=\,x_k.
\end{equation}
Denote by $C_k$ the set whose elements are those of the chain \eqref{chain}. These are non-empty compact sets contained in the compact set $K$ hence by Proposition \ref{Hausdorff_metric}, taking a subsequence if necessary, there is a non-empty compact set $C\subset K$ which is the limit of the sequence $(C_k)$ under the Hausdorff metric.

By Lemma \ref{Hausdorff}, we have that $x\in C$ and $C\subset A$. Thus, it will be enough to prove that for every $\varepsilon>0$ there is an $\varepsilon$-chain from $x$ to itself in $C$.

Let $\varepsilon>0$. Since $\varphi$ is uniformly continuous on $K$, there is $0<\delta<\varepsilon/3$ such that $d(\,\varphi(y_1),\,\varphi(y_2))<\varepsilon/3$ for every pair of points $y_1,\,y_2$ in $K$ verifying $d(y_1,\,y_2)<\delta$.
Consider a natural $m\,>\,3/\varepsilon$ such that $d_H(\,C_m,\,C)\,<\delta$ and $d(x_m,\,x)\,<\,\delta$.

For each $1\leq\,i\,\leq j_m-1$, there is a point $y_i\in C$ such that $d(x_{m,i},\,y_i)<\delta$. We claim that
$$x\,=\,y_{0},\,y_{1},\,\ldots,\,y_{j_m-1},\,y_{j_m}\,=\,x$$
is an $\varepsilon$-chain. In effect, for every $0\leq i\leq j_m-1$ we have that
\begin{equation}\label{cuenta1}
d\left(\varphi(y_i),\,\varphi(x_{m,i})\right)\,<\,\frac{\varepsilon}{3}
\end{equation}
because $d(x_{m,i},\,y_i)<\delta$ by construction. We also have that
\begin{equation}\label{cuenta2}
d\left(\,\varphi(x_{m,i}),\,x_{m,i+1}\right)\,<\,\frac{1}{m}\,<\,\frac{\varepsilon}{3}
\end{equation}
because \eqref{chain} is an $m^{-1}$-chain when $k=m$. Finally, we have that
\begin{equation}\label{cuenta3}
d\left(x_{m,i+1},\,y_{i+1}\right)\,<\,\delta\,<\,\frac{\varepsilon}{3}
\end{equation}
by construction again. Summing up \eqref{cuenta1}, \eqref{cuenta2} and \eqref{cuenta3} we conclude that
$$d\left(\varphi(y_{i}),\,y_{i+1}\right)\,<\,\varepsilon$$
and this proves the claim. Since the choice of $x\in A$ was arbitrary, we have the result.
\end{proof}

Taking the constant sequence, the previous example also shows that if $A_n$ does not verify the restriction property, then in general neither does the set $A$.

Because of the fact that every recurrent set is chain-recurrent with the restriction property, we immediately have the strong version of Lemma \ref{weak_statement}.

\begin{cor}\label{chain_rec_restr_property}
If there is a compact set $K$ in $X$ containing every set $A_n$ and every $A_n$ is recurrent, then $A$ is chain-recurrent with the restriction property.
\end{cor}

\section{Dynamics near the subcritical and critical boundary}

In this section, it will be shown that if the McGehee flow on $M$ has recurrent sets arbitrarily close to the subcritical and critical boundary, then these sets must be localized in very particular regions. All of the dynamical statements will be with respect to this flow. It is worth noting that \textbf{no hyperbolicity} or any other dynamical structure on the McGehee flow \textbf{is assumed}.


\begin{prop}\label{near_recurrency}
Consider a compact set $K$ in $M$ containing $\partial M_{\leq 0}$ and a sequence $(\omega_n)_{n\in \N}$ of non-empty chain-recurrent sets with the restriction property contained in $K$, whose superior limit $\omega$ is contained in $\partial M_{\leq 0}$. Then, $\omega$ is a non-empty compact set contained in the fixed point set of $\partial M_{0}$.
\end{prop}
\begin{proof}
By item \ref{nonempty_compact} in Lemma \ref{general_properties}, $\omega$ is a non-empty compact set. By Lemma \ref{restricted_chain_recurrence_limit}, the set $\omega$ is chain-recurrent with the restriction property. Therefore, since the critical and subcritical boundary is a Lyapunov system with strict Lyapunov function outside its fixed point set, by Proposition \ref{Lyapunov_prop} we have
$$\omega\,=\,\mathcal{R}\left(\varphi_\tau|_\omega\right)\,\subset\,\mathcal{R}\left(\varphi_\tau|_{\partial M_{\leq 0}}\right)\,=\,\mbox{Fix}\,(\partial M_0),$$
and this proves the result.
\end{proof}

\begin{cor}\label{near_recurrency_cor}
Under the hypotheses of Proposition \ref{near_recurrency}, we have:
\medskip

\begin{enumerate}
\item For every neighbourhood in $M$ of the fixed point set of $\partial M_{0}$, the sequence $(\omega_n)_{n\in \N}$ is eventually contained in the neighbourhood.
\medskip

\item If additionally every $\omega_n$ is minimal, then for any connected component $C$ of the fixed point set of $\partial M_{0}$ such that $C\cap \omega$ is non-empty and any neighbourhood of $C$ in $M$, there is a subsequence of $(\omega_n)_{n\in \N}$ contained in the neighbourhood.
\end{enumerate}
\end{cor}
\begin{proof}
\begin{enumerate}
\item Consider a neighbourhood $V$ in $M$ of the fixed point set $F$ of $\partial M_{0}$. Suppose that $(\omega_n)_{n\in \N}$ is not eventually contained in $V$. Then, there is a strictly increasing sequence $(n_k)$ such that for every natural $k$ there is $\zeta_{\,k}$ in $\omega_{n_k}-V$. Taking a further subsequence if necessary, we may suppose that the sequence $(\zeta_{\,k})$ converges to some point $\zeta$ necessarily in $\omega-V$ which is absurd.
\bigskip

\item Let $\zeta$ in $C\cap \omega$. By item \ref{subseq} in Lemma \ref{general_properties}, there is a strictly increasing sequence $(n_k)$ such that for every natural $k$ there is $\zeta_{\,k}$ in $\omega_{n_k}$ such that $\zeta_{\,k}\rightarrow \zeta$. Consider
$$\omega' \,=\,\limsup_k\,\omega_{n_k}.$$
It is clear that $\omega'\subset \omega$ hence by the previous proposition $\omega'\subset F$. Since every $\omega_{n_k}$ is minimal, they are connected. Thus, by item \ref{connected} in Lemma \ref{general_properties}, the set $\omega'$ is also connected and contains the point $\zeta\in C$. We conclude that $\omega'\subset C$.
Consider a neighbourhood $V'$ in $M$ of $C$. A similar argument as the one in the previous item shows that the sequence $(\omega_{n_k})$ is eventually in $V'$ and this proves the result.
\end{enumerate}
\end{proof}

\begin{prop}\label{near_recurrency_II}
Suppose that there is a point $\zeta_*$ in $M$ whose omega-limit set $\omega$ is contained in the critical and subcritical boundary $\partial M_{\leq 0}$. Then, $\omega$ is a non-empty compact connected set contained in the fixed point set of $\partial M_{0}$.
\end{prop}
\begin{proof}
Since $\partial M_{\leq 0}$ is compact by Corollary \ref{compact_invariant_boundaries}, the orbit $\I^+(\zeta_*)$ is eventually contained in a compact set hence, by general theory, the maximal interval of definition of $(\tau\mapsto\varphi_{\zeta_*}(\tau))$ contains $[0,\,+\infty)$ and also $\omega(\zeta_*)$ is a non-empty compact connected invariant set.

Consider the fixed point set $F$ of $\partial M_{0}$ and denote by $\omega_n$ the tail of the $\zeta_*$-orbit starting at time $\tau=n$,
$$F\,=\,\mbox{Fix}(\partial M_0),\qquad \omega_n\,=\,\varphi_{\zeta_*}\,[n,+\infty).$$
Note of course that $\omega$ is the superior limit of $\omega_n$.

Suppose that there is a point $\zeta$ in $\omega\,-\,F$. As an immediate corollary of Proposition \ref{Lyapunov_prop}, there are different connected components $C_1$ and $C_2$ of $F$ containing the alpha and omega-limit sets of $\zeta$ respectively. By Corollary \ref{Lyapunov_constant}, the Lyapunov function $\nu$ is constant on these components and takes the values $\nu_1$ on $C_1$ and $\nu_2$ on $C_2$. Note that $\nu_1<\nu_2$. Again by Corollary \ref{Lyapunov_constant}, there is $\delta>0$ such that $\delta\,<\,(\nu_2-\nu_1)/3$ and $\nu_2-\delta\notin \nu(F)$.

Consider the following neighbourhoods in $M$ of $C_1$ and $C_2$ respectively
$$V_1\,=\,\nu^{-1}\left(B(\nu_1,\,\delta)\right),\qquad V_2\,=\,\nu^{-1}\left(B(\nu_2,\,\delta)\right),$$
and note that their closures are disjoint. By item \ref{subseq} in Lemma \ref{general_properties}, there is a strictly increasing sequence $(n_k)$ such that for every natural $k$ there is $\zeta_{\,k}\in\omega_{n_k}$ such that $\zeta_{\,k}\rightarrow \zeta$. Taking a further subsequence if necessary, we may suppose that $\I^+(\zeta_{\,k})\,\cap\,V_i\,\neq\,\emptyset$ for every natural $k$ and $i=1,\,2$.

For every natural $k$ there is an interval $(\tau_{k2},\,\tau_{k1})$ such that
$$\nu\left(\varphi_{\zeta_{\,k}}(\tau_{k2})\right)\,=\,\nu_2-\delta,\qquad \nu\left(\varphi_{\zeta_{\,k}}(\tau_{k1})\right)\,=\,\nu_1+\delta,$$
$$\nu\left(\varphi_{\zeta_{\,k}}\left(\tau_{k2},\,\tau_{k1})\right)\right)\,\subset\,(\nu_1+\delta,\,\nu_2-\delta).$$

\noindent \emph{Claim:} There is $c>0$ independent of $k$ such that $\tau_{k1}\,-\,\tau_{k2}\geq c$.
\begin{proof}[Proof of the claim]
Indeed, by the motion equations \eqref{McGehee_eq}, a similar calculation as the one in Lemma \ref{Lyapunov_calculation} gives a $C^1$ function $\Xi$ such that \footnote{The function $\Xi$ is defined explicitly by the formula
\begin{equation*}
\begin{array}{r@{}l}
\medskip
\Xi(r,q,y)\,=\,\left(1+l/2\right)\,&(\Vert y\Vert^2\,-\,\nu(q,y)^2)\,-\,l\,\tilde{H}(r,q,y)\,\\
\medskip
&+\,r\,\left(\langle q,\,F_X(r,q,y)\rangle\,+\,l\,F_H(r,q,y)\right)
+\,r^\Delta\,\langle q,\,\mathbf{F}_X(r,q,y)\rangle.
\end{array}
\end{equation*}
}
$$\nu'\,=\,\partial_\tau\,(\nu\,\circ\,\varphi)\,|_{\tau=0}\,=\,\Xi,\qquad \Xi:M\rightarrow \R.$$
Consider the maximum $M_{\Xi}$ of the function $\vert\Xi\vert$ on the compact set $K$ and note that it is greater than zero since otherwise $\nu$ would be constant along $\I(\zeta)$ which implies, by Proposition \ref{Lyapunov_prop}, that $\zeta$ is a fixed point which is absurd by hypothesis.

By Lagrange Theorem, there is $\tau_*\in (\tau_{k2},\,\tau_{k1})$ such that
$$\nu_1-\nu_2+2\delta\,=\, \nu'\left(\varphi_{\zeta_{\,k}}(\tau_*)\right)\,(\tau_{k1}\,-\,\tau_{k2}).$$
Thus, taking absolute values on both sides we have
$$\delta\,<\,\nu_2-\nu_1-2\delta\,=\,\vert \nu'\left(\varphi_{\zeta_{\,k}}(\tau_*)\right)\vert\,(\tau_{k1}\,-\,\tau_{k2})\,\leq\,M_\Xi\,(\tau_{k1}\,-\,\tau_{k2})$$
hence $\tau_{k1}\,-\,\tau_{k2}\,>\,M_\Xi^{-1}\,\delta\,>\,0$ and taking $c=M_\Xi^{-1}\,\delta$ the claim is proved.
\end{proof}
\medskip

Now, we continue with the proof. Define $\zeta'_{\,k}\,=\,\varphi_{\zeta_{\,k}}(\tau_{k2})$ and note that for every natural $k$ we have $\nu(\zeta'_{\,k})\,=\,\nu_2-\delta$ and 
$\nu\left(\varphi_{\zeta'_{\,k}}(\varepsilon)\right)\,<\,\nu_2-\delta$ for every $\varepsilon\in (0,c)$. Taking a further subsequence if necessary, we may suppose that the sequence $(\zeta'_{\,k})$ converges to some point $\zeta'$ necessarily in $\partial M_{\leq 0}$. In particular, we have $\nu(\zeta')\,=\,\nu_2-\delta$ and 
$\nu\left(\varphi_{\zeta'}(\varepsilon)\right)\,\leq\,\nu_2-\delta$ for every $\varepsilon\in (0,c)$. Since $\nu$ is increasing on $\partial M_{\leq 0}$ and this is an invariant set, we conclude that $\nu\left(\varphi_{\zeta'}(\varepsilon)\right)\,=\,\nu_2-\delta$, constant for $\varepsilon\in (0,c)$. Therefore, $\zeta'$ is a fixed point such that $\nu(\zeta')\,=\,\nu_2-\delta$ which is absurd since $\nu_2-\delta\notin \nu(F)$ by hypothesis. This finishes the proof.
\end{proof}

A verbatim argument as the one before, proves Proposition \ref{near_recurrency} avoiding chain-recurrence theory at the cost of assuming a sequence $(\omega_n)$ of recurrent sets instead.

\section{Hyperbolicity at the critical boundary}\label{HyperbolicSet}

Recall from section \ref{blowup_definition_section} that, except for even values of the degree $l$, the map $\pi$ in expression \eqref{map_pi} cannot be even smoothly extended to the boundary $\partial M$ and it does so only continuously. However, the McGehee manifold does admit a real analytic extension even for negative radial parameter. We will refer to this extension as the \emph{extended McGehee manifold},
$$\hat{M}\,=\, (-r_E,\,r_E)\times S^{n-1}\times \R^n.$$

A verbatim argument as in the proof of Proposition \ref{Extension_vector} shows that, for even values of $l$, the result holds for the extended McGehee manifold as well. Instead, if the value $l$ is odd, then $\Delta$ is no longer an integer and because of the $r^\Delta$ factor, equations \eqref{McGehee_eq} do not even real analytically extend to the boundary. Nevertheless, replacing this factor simply by $|r|^\Delta$ results in a continuously differentiable term since necessarily $\Delta>1$ because of the weak magnetism hypothesis \eqref{weak_magnetism}. Similarly as before, the resulting equations define a continuously differentiable vector field whose flow restricts to the original McGehee flow on $M$.

We will abuse of notation and refer to the flow on $\hat{M}$ also as the McGehee flow. The advantage of this new manifold over $M$ is that $\hat{M}$ is boundaryless and $\partial M$ is a submanifold.

As it was mentioned in section \ref{blowup_definition_section}, the critical boundary $\partial M_0$ is not a manifold in general and we provided an example whose critical boundary was a doubly-pinched torus. It was shown in section \ref{Dynamics_criticalboundary_section} that the dynamics on the subcritical and critical boundary is a Lyapunov system whose chain-recurrent set is the fixed point set. However, without additional structure on the fixed point set, nothing can be said about its relation with the dynamics on the McGehee manifold.

In this section, conditions are given to assure that the fixed point set is a hyperbolic subset, not only of the critical boundary, but of the extended McGehee manifold as well.

As before, denote by $f$ the restriction to the unit sphere of the first non-zero jet $U_l$ of the potential $U$ at the critical point $p$. Note that $U_l$ is a degree $l$ homogeneous polynomial on $\R^n$ while $f$ is a smooth function on the sphere. 


\begin{lema}\label{Techincal_lemma}
If $q_*$ is a critical point of $f$, then $\langle\, q_* \rangle\,\oplus\, T_{\,q_*}S^{n-1}$ is an orthogonal invariant subspace decomposition of $\Hess_{\,q*}( U_l)$. Moreover, with respect to the induced Euclidean metric on the sphere, we have
$$\Hess_{\,q_*}(f)\,=\,\Hess_{\,q*}( U_l)\,-\,l\,U_l\,(q_*)\,\id _{n-1},$$
where the right hand side is restricted to $T_{\,q_*}S^{n-1}$.
\end{lema}
\begin{proof}
Since $U_l$ is a degree $l$ homogeneous function, by Euler's theorem
\begin{equation}\label{Euler_Theorem}
\langle\,\nabla U_l (x),\,x\rangle\,=\,l\,U_l(x)
\end{equation}
for every $x$ in $\R^n$. The directional derivative of the above expression along a vector $v$ in $\R^n$ reads as
$$\langle\,\Hess_{\,x}(U_l)(v),\,x\rangle\,+\,\langle\,\nabla U_l (x),\,v\rangle\,=\,l\,\langle\,\nabla U_l (x),\,v\rangle.$$
Suppose that $q_*$ is a critical point of $f$. Then, $\nabla U_l (q_*)\,=\,\lambda\,q_*$ for some real number $\lambda$, the Lagrange multiplier. In particular, grouping the terms of the previous expression, for every vector $v$ in $T_{\,q_*}S^{n-1}$ we have
$$\langle\,\Hess_{\,q_*}(U_l)(v),\,q_*\rangle\,=\,(l-1)\,\langle\,\nabla U_l (q_*),\,v\rangle\,=\,(l-1)\,\lambda\,\langle\,q_*,\,v\rangle\,=\,0$$
and we conclude that $T_{\,q_*}S^{n-1}$ is an invariant subspace of $\Hess_{\,q_*}(U_l)$. Because $\Hess_{\,q_*}(U_l)$ is a symmetric operator, for every vector $v$ in $T_{\,q_*}S^{n-1}$ we have
$$\langle\,v,\,\Hess_{\,q_*}(U_l)(q_*)\rangle\,=\,\langle\,\Hess_{\,q_*}(U_l)(v),\,q_*\rangle\,=\,0$$
therefore $\langle\, q_* \rangle$ is an invariant subspace as well.

The gradient of $f$ at $q\in S^{n-1}$ with respect to the induced Euclidean metric on the sphere, is the orthogonal projection on $T_{\,q}S^{n-1}$ of the usual gradient of $U_l$ at $q$. Therefore, by Euler's theorem \eqref{Euler_Theorem} we have
$$\nabla U_l (q)\,=\,l\,U_l(q)\,q\,+\,\grad_q(f),\qquad q\in S^{n-1}.$$

The Levi-Civita connection on $S^{n-1}$ with respect to the induced Euclidean metric, is the orthogonal projection on the corresponding tangent space of the usual derivative. Thus, deriving the last expression along a vector $v$ in $T_{\,q_*}S^{n-1}$ we have
\begin{eqnarray*}
\Hess_{\,q_*}(U_l)(v)\,&=&\,l\,U_l(q_*)\,v\,+\,l\,\langle\,\nabla U_l (q_*),\,v\rangle\, q_*\,+\,D_v\,\grad(f) \\
&=&\,\,l\,U_l(q_*)\,v\,+\,l\,\lambda\,\langle\, q_*,\,v\rangle\, q_*\,+\,D_v\,\grad(f) \\
&=&\,\,l\,U_l(q_*)\,v\,+\,D_v\,\grad(f).
\end{eqnarray*}
Because $T_{\,q_*}S^{n-1}$ is an invariant subspace of $\Hess_{\,q_*}(U_l)$, the derivative $D_v\,\grad(f)$ coincides with its orthogonal projection on the tangent space hence, for every $v$ in $T_{\,q_*}S^{n-1}$,
$$\Hess_{\,q_*}(f)(v)\,=\,\nabla_v\,\grad(f)\,=\,D_v\,\grad(f)\,=\,\Hess_{\,q*}( U_l)(v)\,-\,l\,U_l\,(q_*)\,v$$
where $\nabla_v$ denotes the Levi-Civita connection along the vector $v$ and we have the result.
\end{proof}

\begin{prop}\label{Morse_Proposition}
Suppose that $f$ is a Morse function and zero is a regular value. Then, the following assertions hold.
\medskip

\begin{enumerate}
\item $\partial M_{\leq 0}$ is a manifold with boundary $\partial M_{0}$.
\medskip

\item The fixed point set
$$Fix(\,\partial M_{\leq 0}\,)\,=\,Fix(\,\partial M_0\,)$$
is a hyperbolic subset of the extended McGehee manifold $\hat{M}$.
\medskip

\item If $\zeta\,=\,(0,\,q,\,\nu\,q)$ is one of the fixed points in the previous item and $\nu<0$ $(\nu<0)$, then there is an orbit in the interior of $M$ converging to $\zeta$ as time goes to $+\infty$ $(-\infty)$ and the function $\tilde{H}$ is zero along the orbit.
\medskip

\item In the previous item, if additionally $q$ is a local minimum of $f$, then the unstable (stable) manifold of $\zeta$ is an open set in $\partial M_{\leq 0}$ while the stable (unstable) one is an immersed one-submanifold and these intersect transversally at $\zeta$.
\end{enumerate}
\end{prop}
\begin{proof}
Suppose that $f$ is a Morse function without zero critical values. Consider the smooth function \eqref{Energy_rescaled} restricted to the McGehee boundary,
$$\tilde{H}_{\partial M}:\,\partial M\rightarrow \R,\qquad \tilde{H}_{\partial M}(0,q,y)\,=\,\frac{\Vert y\Vert^2}{2}\,+\,f(q).$$
Its critical points are of the form $(0,q_*,0)$ such that $q_*$ is a critical point of $f$. Thus, zero is a regular value of $\tilde{H}_{\partial M}$ hence $\tilde{H}_{\partial M}^{-1}(-\infty,0]\,=\,\partial M_{\leq 0}$ is a manifold with boundary $\tilde{H}_{\partial M}^{-1}(0)\,=\,\partial M_{0}$. This proves the first item.
\medskip

Consider a fixed point $\zeta$ in $\partial M_{0}$. By Lemma \ref{concrete_description_fixpoint},
$$\zeta\,=\,(0,\,q_*,\,\nu_*\,q_*),\qquad \nu_*\,=\,\pm(-2\,\it f(q_*) \rm )^{1/2}$$
such that $q_*$ is a critical point of $f$ verifying $f(q_*)\leq 0$. Because zero is a regular value of $f$, we actually have $f(q_*)< 0$ and $\nu_*\neq 0$. By definition of the extended McGehee manifold, the tangent space of $\hat{M}$ at $\zeta$ splits as follows
\begin{equation}\label{Splitting1}
T_\zeta\,\hat{M}\,=\,T_0\,\R\,\oplus\,T_{q_*}S^{n-1}\,\oplus\,T_{\nu_* q_*}\R^n,\qquad v\,=\,v_r\,+\,v_q\,+\,v_y,
\end{equation}
where the first factor is naturally isomorphic to $\R$. We can further decompose the last factor as follows
\begin{equation}\label{Splitting2}
T_{\nu_* q_*}\R^n\,=\,\R^n\,=\,\ker\,(\partial_y\,\nu\,|_\zeta)\,\oplus \langle\,q_*\rangle,\qquad \partial_y\,\nu\,|_\zeta\, (v_y)\,=\,\langle q_*,\, v_y\rangle.
\end{equation}
In view of the latter expression, it is natural to identify the first factor in \eqref{Splitting2} with $T_{q_*}S^{n-1}$ and the second with $\R$. Thus, we have the splitting
\begin{equation}\label{Splitting3}
T_\zeta\,\hat{M}\,=\,\R\,\oplus\,T_{q_*}S^{n-1}\,\oplus\,T_{q_*}S^{n-1}\,\oplus\,\R.
\end{equation}

Recall the extension of the vector field $\tilde{X}=(\tilde{X}_r,\,\tilde{X}_q,\,\tilde{X}_y)$ to the extended McGehee manifold $\hat{M}$ discussed at the beginning of the section,
\begin{equation}\label{extended_X_tilde}
\begin{array}{r@{}l}
\bigskip

\tilde{X}_r(r,q,y)\,&=\,\nu\,r, \\
\bigskip

\tilde{X}_q(r,q,y)\,&=\,y\,-\,\nu\,q, \\

\tilde{X}_y(r,q,y)\,&=\,-\nabla U_l\,(q)\,-\,\frac{l}{2}\,\nu\,y\,+\,r\,F_X(r,q,y)\,+\,\mathcal{F}_X(r,q,y).
\end{array}
\end{equation}
such that the new magnetic term $\mathcal{F}_X$ is given by
$$\mathcal{F}_X(r,q,y)=
\begin{cases}
\bigskip

    \quad r^\Delta\,\mathbf{F}_X(r,q,y) & \text{if $\Delta\in\N$,} \\
    \quad |r|^\Delta\,\mathbf{F}_X(r,q,y) & \text{if $\Delta-1/2\in\N$,}
  \end{cases}.$$
As it was mentioned at the beginning of the section, \eqref{extended_X_tilde} is continuously differentiable. With respect to the splitting \eqref{Splitting3}, the differential of $\tilde{X}$ at $\zeta$ reads as follows
\medskip

\begin{equation}\label{Linear_part}
d_\zeta\,\tilde{X}\,=\,\begin{pmatrix}
\nu_* && 0 && 0 && 0\\
&& && && \\
0 && -\nu_*\,\id_{n-1} && \id_{n-1} && 0\\
&& && && \\

F_X^{tg}\,(\zeta)+\delta^\Delta_1\,\mathbf{F}_X^{tg}\,(\zeta) && -\Hess_{q_*}\, (U_l) && -\frac{l}{2}\nu_*\,\id_{n-1} && 0 \\
&& && && \\

F_X^{||}\,(\zeta)+\delta^\Delta_1\,\mathbf{F}_X^{||}\,(\zeta) && 0 && 0 && -l\,\nu_*
\end{pmatrix}.
\end{equation} 
\medskip

Again, with respect to the splitting \eqref{Splitting3}, it is clear that $(0,0,0,1)$ is an eigenvector of $d_\zeta\,\tilde{X}$ with non-zero eigenvalue $-l\,\nu_*$.

Because $\Hess_{q_*}\, (U_l)$ is a symmetric operator, by Lemma \ref{Techincal_lemma} there is an orthonormal basis $\{v_i\}$ of $T_{q_*}S^{n-1}$ consisting of eigenvectors of the operator $\Hess_{q_*}\, (U_l)$ such that
$$\Hess_{q_*}\, (U_l)\,(v_i)\,=\,\lambda_i\,v_i,\qquad i\,=\,1,\,2,\,\ldots,\,n-1.$$
By direct inspection, the vector $(0,v_i,\alpha_i\,v_i,0)$ is an eigenvector of $d_\zeta\,\tilde{X}$ with eigenvalue $k_i$ such that
$$\alpha_i\,-\,\nu_*\,=\,k_i,\qquad \alpha_i^2\,+\,\nu_*\left(\frac{l}{2}-1\right)\,\alpha_i\,+\,\lambda_i\,=\,0.$$
In particular, by Lemma \ref{Techincal_lemma} again, we have the following expression for the eigenvalues of $d_\zeta\,\tilde{X}$ in terms of those of $\Hess_{\,q_*}(f)$,
\begin{equation}\label{eigenvalue_expression}
k_i\,=\,-\frac{1}{2}\nu_*\left(\frac{l}{2}+1\right)\,\pm\,\frac{1}{2}\left(\nu_*^2\left(\frac{l}{2}+1\right)^2-4\hat{\lambda}_i\right)^{1/2}
\end{equation}
where we have denoted by $\hat{\lambda}_i$ the eigenvalues of the Hessian operator of $f$ at $q_*$,
\begin{equation}\label{eigenvalue_expression_2}
\Hess_{q_*}\, (f)\,(v_i)\,=\,\hat{\lambda}_i\,v_i.
\end{equation}
Since $\hat{\lambda}_i\neq 0$  due to the fact that $f$ is a Morse function, if
$$\nu_*^2\left(\frac{l}{2}+1\right)^2\,\neq\,4\hat{\lambda}_i,$$
then there are two linearly independent eigenvectors with different eigenvalues whose real part is non-zero. Otherwise, if
$$\nu_*^2\left(\frac{l}{2}+1\right)^2\,=\,4\hat{\lambda}_i,$$
then $\{\,(0,0,v_i,0),\,(0,v_i,\alpha_i\,v_i,0)\,\}$ is a Jordan basis of an invariant subspace of $d_\zeta\,\tilde{X}$, with respect to which the operator is the Jordan block
\begin{equation}
\begin{pmatrix}
k_i & 0 \\
1 & k_i
\end{pmatrix},\qquad k_i\,=\,-\frac{1}{2}\nu_*\left(\frac{l}{2}+1\right)\neq 0.
\end{equation}

It is clear from the orthonormality of the basis $\{v_i\}$ that the union of all of the previous eigenvectors and Jordan bases is a linearly independent set $S$. It is also clear that adding an adequate linear combination of the vectors in $S$ to the vector $(1,0,0,0)$, results in a linearly independent eigenvector of $d_\zeta\,\tilde{X}$ with eigenvalue $\nu_*\neq 0$. The union of this new vector with $S$ constitutes a Jordan basis for the operator $d_\zeta\,\tilde{X}$.

We have proved that all of the eigenvalues of the Jordan canonical form of the operator $d_\zeta\,\tilde{X}$ have non-zero real part. In particular, the operator $d_\zeta\,\tilde{X}$ is hyperbolic. This proves the second item.
\medskip

Consider the resulting invariant subspace decomposition of $d_\zeta\,\tilde{X}$,
\begin{equation}\label{Splitting5}
T_\zeta\,\hat{M}\,=\,\R\,\oplus\,T_\zeta\left(\partial M_{\leq 0}\right)\,=\,\R\,\oplus\,T_\zeta\left(\partial M_0\right)\,\oplus\,\R.
\end{equation}
Suppose that $\nu_*<0$. Then, it follows immediately from \eqref{Linear_part} that the first $\R$ factor in \eqref{Splitting5} is a stable invariant subspace of the operator while the last $\R$ factor is an unstable one. An analogue statement holds for $\nu_*>0$ inverting the roles of the stable and unstable invariant subspaces. A direct application of the Hartman-Grobman Theorem gives the desired orbit. By Corollary \ref{Corollary_omegalimit_constraint}, the function $\tilde{H}$ is zero along the orbit and this proves the third item.
\medskip

Finally, if $q_*$ is a local minimum of $f$, then every eigenvalue $\hat{\lambda}_i$ in \eqref{eigenvalue_expression_2} is greater than zero, thus by \eqref{eigenvalue_expression} we have that $T_\zeta\left(\partial M_0\right)$ is an unstable invariant subspace. If $\nu_*>0$, then inverting the roles of the stable and unstable invariant subspaces a similar result holds. The fourth item follows then as a direct application of the Hartman-Grobman Theorem and this completes the proof.
\end{proof}

\begin{rmrk}
The eigenspace with eigenvalue $-l\,\nu_*$ of the operator $d_\zeta\,\tilde{X}$ in the previous proof, corresponds to the heteroclinic orbit in Lemma \ref{heteroclinic_orbit}.
\end{rmrk}

\section{Genericity of the Hypothesis I}\label{Genericity}

In this section it will be proved that the hypothesis in Proposition \ref{Morse_Proposition} is generic. Moreover, it will be proved that it is an open and dense subset of the degree $l$ homogeneous polynomial space. The difficulty and also the reason why general results in the theory cannot be directly applied here, lies in the impossibility of using Urysohn functions in the argument.

Consider the space $S_l(\R^n)$ of $l$-homogeneous polynomials on $\R^n$ with the compact-open topology and denote by $\mathcal{H}$ the subset of functions whose restriction to the unit sphere $S^{n-1}$ is a Morse function such that zero is not a critical value. The following is the main result of this section.

\begin{prop}
The set $\mathcal{H}$ is an open and dense subset of $S_l(\R^n)$.
\end{prop}

For the proof, define the set $\mathcal{E}$ consisting of those $l$-homogeneous polynomials whose restriction to the unit sphere $S^{n-1}$ is a Morse function. In particular,
$$\mathcal{H}\,\subset\,\mathcal{E}\,\subset\,S_l(\R^n).$$
First, we will prove that $\mathcal{H}$ is an open and dense subset of $\mathcal{E}$ and then we will do the same for the other inclusion. The conclusion will follow immediately from Propositions \ref{Prop1_generic} and \ref{Prop2_generic}.

\begin{lema}
Consider a compact manifold $M$ and a $C^2$ function
$$f:\,M\times B_m(0,\,\varepsilon)\,\rightarrow\,\R,\qquad B_m(0,\,\varepsilon)\subset \R^m$$
such that $f(\cdot,\,0)$ is a Morse function with critical points $p_1,\,p_2,\,\ldots,\,p_r$. Then, there is a real number $\delta>0$ together with $C^1$ functions
$$g_i:\,B_m(0,\,\delta)\rightarrow M,\qquad g_i(0)\,=\,p_i,\qquad i=1,2,\ldots,r$$
such that for each $a\in B_m(0,\,\delta)$, the function $f(\cdot,\,a)$ is Morse with critical points $g_1(a),\,g_2(a),\,\ldots,\,g_r(a)$ preserving the respective Morse indices; that is
$$\ind\ g_i(a)\,=\,\ind\ p_i,\qquad a\,\in\,B_m(0,\,\delta),\qquad i=1,2,\ldots,r.$$
\end{lema}
\begin{proof}
For each $p_i$ take a coordinate chart $(U_i,\,\eta_i)$ centered at $p_i$ with $U_i$ small enough such that $U_i\cap U_j$ is empty if $i\neq j$. Define
\begin{align*}
f_i:&\,\eta_i(U_i)\times B_m(0,\,\varepsilon)\rightarrow \R,\qquad &f_i(y,a)\,=&\,f(\eta_i^{-1}(y),a), \\
h_i:&\,\eta_i(U_i)\times B_m(0,\,\varepsilon)\rightarrow \R^n,\qquad &h_i(y,a)\,=&\,\left(\partial_1 f_i(y,a),\,\ldots,\,\partial_n f_i(y,a)\right),
\end{align*}
where $n$ is the dimension of the manifold $M$.
The functions $h_i$ are $C^1$ and verify
$$h_i(0,0)=0,\qquad D_y h_i\,(0,0)=\left(\partial^2_{\alpha\beta} f_i\,(0,0)\right)_{1\leq \alpha,\beta\leq n},\qquad i=1,2,\ldots,r$$
where the latter differential is non singular since $f_i(\cdot,\,0)$ are Morse functions.

By the implicit function theorem, taking $U_i$ small enough, for every $i=1,\ldots,r$ there is $\delta_i>0$ and a unique function 
$$\hat{g}_i:\,B_m(0,\,\delta_i)\rightarrow \eta_i(U_i),$$
verifying $h_i(\,\hat{g}_i(a),\,a)\,=\,0$ for every $a$. Moreover, these functions are $C^1$ and the differential $D_y h_i\,(\,\hat{g}_i(a),\,a)$ is non-singular for every $a$.
Due to the uniqueness of these functions, we have that $\hat{g}_i(0)\,=\,0$ for $h_i(0,0)=0$ and $\hat{g}_i(a)$ is the only critical point of $f_i(\cdot,\,a)$ which is non-degenerate because
$$\left(\partial^2_{\alpha\beta}\, f_i\,(\,\hat{g}_i(a),\,a)\,\right)_{1\leq \alpha,\beta\leq n}\,=\,D_y h_i\,(\,\hat{g}_i(a),\,a)$$
and the latter was non-singular.

Take any Riemannian metric on $M$ and consider the smooth function
$$s:\,M\times B_m(0,\,\varepsilon)\,\rightarrow\,\R,\qquad s(x,a)\,=\,\Vert\,\grad_x\,f(\cdot,\,a)\,\Vert^2.$$
Because $s(\cdot,\,0)>0$ on the compact set $M-\bigcup_i\,U_i$, it has a minimum $s_*>0$ on that set. In particular, there is $\delta_0>0$ such that $s\geq s_*/2$ on $\left(M-\bigcup_i\,U_i\right)\times B_m(0,\,\delta_0)$. We have proved that for $a$ in $B_m(0,\,\delta_0)$, there are no other critical points of $f_i(\cdot,\,a)$ except for those contained in the sets $U_i$. Defining $\delta$ as the minimum of $\delta_0,\,\delta_1,\,\ldots\,,\delta_r$ and $g_i\,=\,\eta_i^{-1}\circ \hat{g}_i$, we have the first part of the result.

Finally, based on the above we have that $f$ is a Morse-Bott function on the manifold $M\times B_m(0,\,\delta)$ and a direct application of the Morse-Bott Lemma \cite{Morse-Bott_Lemma} gives that the Morse indexes of the critical points are preserved, that is
$$\ind\ g_i(a)\,=\,\ind\ p_i$$
for every $a$ in $B_m(0,\,\delta)$ and every $i=1,\ldots,r$. This finishes the proof.
\end{proof}

The proof that the inclusion $\mathcal{H}\subset\mathcal{E}$ is dense in the case where $l$ is even is trivial. Indeed, given any function in $\mathcal{E}$, adding the $l$-homogeneous polynomial
$$\varepsilon\,\Vert\,x\,\Vert^{l}\,=\,\varepsilon\,\left(x_1^2\,+\,x_2^2\,+\,\ldots\,+\,x_n^2\right)^{l/2}$$
to the function has the effect of adding the constant $\varepsilon$ to the restriction to the sphere $S^{n-1}$. Since the set of critical values of the function is finite, the new function will be in $\mathcal{H}$ for an arbitrary small $\varepsilon$.

This is no longer possible in the case where $l$ is odd. Actually, in this case every function in $\mathcal{E}$ has a continuum of zeros in the sphere $S^{n-1}$ and a priori there is no reason why perturbing this function in $\mathcal{E}$ will not preserve the zero value of a critical point or create new ones with this property along the process.

Note that, since the locally uniformly convergence in $S_l(\R^n)$ coincides with the coefficient-wise one, the space $S_l(\R^n)$ with the compact-open topology is naturally homeomorphic to an Euclidean space. This fact will be used in the next proposition.

\begin{prop}\label{Prop1_generic}
The set $\mathcal{H}$ is an open and dense subset of $\mathcal{E}$.
\end{prop}
\begin{proof}
Let $f$ be in $\mathcal{H}$. We claim that $f$ is an interior point of $\mathcal{H}$. In effect, consider the smooth function
$$\hat{f}:\,S^{n-1}\times S_l(\R^n)\,\rightarrow\,\R,\qquad \hat{f}(x,\,h)\,=\,f(x)\,+\,h(x).$$
By definition of $\mathcal{H}$, the function $f$ is Morse with critical points $p_1,\,p_2,\,\ldots,\,p_r$ and $f$ is non-zero on any of them. By the previous lemma, there is $\delta>0$ together with $C^1$ functions
$$g_i:\,B_N(0,\,\delta)\rightarrow S^{n-1},\qquad i=1,2,\ldots,r$$
such that for each $h\in B_N(0,\,\delta)$, the function $\hat{f}(\cdot,\,h)$ is Morse with critical points $g_1(h),\,g_2(h),\,\ldots,\,g_r(h)$. Define the $C^1$ function
$$\mu:\,B_N(0,\,\delta)\,\rightarrow\,\R,\qquad \mu(h)\,=\,\hat{f}(g_1(h)\,\,h)\,\ldots\, \hat{f}(g_r(h)\,\,h).$$
Because $\mu(0)\,=\,f(p_1)\,\ldots\,f(p_r)\,\neq\,0$, there is $0<\delta_1 <\delta$ such that $\mu\neq 0$ on $B_N(0,\,\delta_1)$. We conclude that $f\,+\,B_N(0,\,\delta_1)$ is contained in $\mathcal{H}$ and this proves the claim.

To prove that $\mathcal{H}$ is actually dense in $\mathcal{E}$, define the set $\mathcal{H}_j$ consisting of those functions in $\mathcal{E}$ whose restriction to $S^{n-1}$ has at most $j$ critical points with zero value. Thus, we have a chain of inclusions
$$\mathcal{H}\,=\,\mathcal{H}_0\,\subset\, \mathcal{H}_1\,\subset\,\ldots\,\subset\,\bigcup_{j\in\N}\,\mathcal{H}_j\,=\,\mathcal{E}.$$
Therefore, it will be enough to prove that for every $j\geq 1$, the inclusion $\mathcal{H}_{j-1}\subset\mathcal{H}_j$ is dense.

Let $f$ be in $\mathcal{H}_{j-1}-\mathcal{H}_j$. Its restriction to $S^{n-1}$ has critical points $p_1,\,p_2,\,\ldots,\,p_r$ and without loss of generality we may suppose that
$$f(p_1)\,=\,\ldots\,=\,f(p_j)\,=\,0,\qquad f(p_{j+1})\neq 0,\,\ldots,\,f(p_{r})\neq 0.$$
For every $k=j+1,\,\ldots,\,r$, there is a neighbourhood $V_k$ of $p_k$ in $S^{n-1}$ such that $\vert f\vert\,>\,\vert f(p_k)\vert/2\,>\,0$ on the neighbourhood. Consider the function
$$\tilde{f}:\,S^{n-1}\times \R\rightarrow\R,\qquad \tilde{f}(x,t)\,=\,f(x)\,+\,t\,\langle\, x,\,p_1\rangle^l.$$
By the previous lemma, there is $\delta_2>0$ together with $C^1$ functions
$$\tilde{g}_i:\,(-\delta_2,\,\delta_2)\rightarrow S^{n-1},\qquad \tilde{g}_i(0)\,=\,p_i,\qquad i=1,2,\ldots,r$$
such that for each $t\in (-\delta_2,\,\delta_2)$, the function $\tilde{f}(\cdot,\,t)$ is Morse with critical points $\tilde{g}_1(t),\,\tilde{g}_2(t),\,\ldots,\,\tilde{g}_r(t)$. Take $\delta_2>0$ small enough such that, for any of the indices $k=j+1,\,\ldots,\,r$, we have the inclusion $\tilde{g}_k(-\delta_2,\,\delta_2)\,\subset\, V_k$. Note that the function $\tilde{g}_1$ is constant and equal to $p_1$ hence
$$\tilde{f}(\tilde{g}_1(t),\,t)\,=\,\tilde{f}(p_1,\,t)\,=\,t.$$
Therefore, for every $t\in (-\delta_2,\,\delta_2)-\{0\}$, the function $f\,+\,t\,\langle\, \cdot,\,p_1\rangle^l$ belongs to $\mathcal{H}_{j-1}$ and we have the result.
\end{proof}

The following is a well-known elementary result in Morse theory and its proof can be found in \cite{GuilleminPollack}, section 1.7.

\begin{lema}\label{Morse_linear}
Consider an embedded manifold $X\subset \R^N$ (not necessarily compact) and a smooth real valued function $f$ on $X$. Then, the set of linear functionals $L$ in $(\R^N)^*$ such that $f+L$ is a Morse function on $X$ is a full Lebesgue measure set.
\end{lema}

\begin{prop}\label{Prop2_generic}
The set $\mathcal{E}$ is an open and dense subset of $S_l(\R^n)$.
\end{prop}
\begin{proof}
The set $\mathcal{E}$ is the intersection of the set of Morse functions with the finite dimensional subspace of $l$-homogeneous polynomials $S_l(\R^n)$. Since the set of Morse functions is open in the space of smooth functions, it immediately follows that $\mathcal{E}$ is an open subset of $S_l(\R^n)$.

To prove that the inclusion $\mathcal{E}\subset S_l(\R^n)$ is actually dense, consider the Veronese map
\begin{equation}\label{Veronese}
\mathcal{V}: S^{n-1}\rightarrow \R^N,\qquad N\,=\,\binom {n+d} n.
\end{equation}
It is a well known result that this map factors through the projective space and it is an embedding there. A careful inspection shows that the same proof proves that the Veronese map \eqref{Veronese} is an immersion.

Take a finite covering $\{\,U_j\ |\ j\in J\}$ of $S^{n-1}$ by small enough open sets such that the Veronese map \eqref{Veronese} is an embedding on any of these sets. This is possible since the map is an immersion and the sphere is compact. By Lemma \ref{Morse_linear}, the subset $Z_j$ of linear functionals in $(\R^N)^*$ which are not Morse functions restricted to $\mathcal{V}(U_j)$ is a Lebesgue measure zero set. In particular, the set of linear functionals in $(\R^N)^*$ whose composition with the Veronese map is not a Morse function on $S^{n-1}$, is the union $Z\,=\,\bigcup_{j\in J}\,Z_j$ which is a Lebesgue measure zero set since $J$ is finite.

Because every function in $S_l(\R^n)$ is uniquely expressed as the composition of a linear functional in $(\R^N)^*$ with the Veronese map, we have proved that $\mathcal{E}$, the set of functions in $S_l(\R^n)$ whose restriction to the sphere $S^{n-1}$ is a Morse function, is a full Lebesgue measure set hence dense since the measure is a Borel measure. This concludes the proof.
\end{proof}

\section{Genericity of the Hypothesis II}\label{Genericity_II}

This section is a direct continuation of section \ref{Genericity}. Here instead, we will be concerned with the genericity of the hypothesis with respect to the stalk of real analytic germs $\I_p$ at the point $p$ in $\R^n$. In contrast with section \ref{Genericity}, the results here will follow almost directly by the previous genericity results and the definition of the objects involved.

The ring $\I_p$ is a local ring with maximal ideal $\mathfrak{m}$ and denote by $\mathcal{I}_k$ the ideal of its $k$ powers, that is
$$\mathcal{I}_k\,=\,\mathfrak{m}^k,\qquad \mathfrak{m}\,\lhd\,\I_p.$$
The ring $\I_p$ can be naturally identified with the ring of convergent power series at the origin, converging on some polydisc with positive radius of convergence. Without loss of generality, we have assumed that $p$ is the origin in the previous identification. Under this identification, the ideal $\mathcal{I}_k$ is the one generated by the $k$-powers of the linear terms $x_1,\,x_2,\,\ldots\,x_n$,
$$\mathcal{I}_k\,=\,\langle\,x_1,\,x_2,\,\ldots\,x_n\,\rangle^k.$$

Consider the canonical epimorphism $\rho_k$ onto the space of $k$-covectors at $p$,
\begin{equation}\label{canonical_projection}
\rho_k:\,\mathcal{I}_k\,\rightarrow\,\mathcal{I}_k\,/\,\mathcal{I}_{k+1}\,\cong\,S_k(\R^n).
\end{equation}
Under the previous identification with the convergent power series ring, the morphism reads explicitly as follows,
\begin{equation}\label{canonical_projection_II}
\rho_k:\,\sum_{\vert \alpha\vert\geq k}\,c_\alpha\,x^\alpha\,\mapsto\,\sum_{\vert \alpha\vert= k}\,c_\alpha\,x^\alpha,\qquad \alpha\in \N_0^n,
\end{equation}
where we have used multi-index notation in the previous expression.

Recall the open and dense subset $\mathcal{H}$ of $S_l(\R^n)$ in section \ref{Genericity}. The following is the main result of this section.

\begin{prop}\label{generic_prop_II}
The set $\rho_l^{-1}(\mathcal{H})$ is an open and dense subset of $\mathcal{I}_l$.
\end{prop}

Consider the usual coordinate-wise embedding $\R^n\subset \C^n$. Now, the point $p$ can be regarded as a point in $\C^n$ whereat the stalk of holomorphic germs contains the one of real analytic germs under the induced monomorphism,
\begin{equation}\label{germs_isomorphism}
\begin{tikzcd}        
    \I_{p,\,\R^n} \arrow[hook]{r} & \I_{p,\,\C^n}^{\,hol}
\end{tikzcd}.
\end{equation}
Indeed, under the identification with the convergent power series ring, the induced morphism is actually the identity.

By definition, the ring of holomorphic germs at $p$ is the direct limit of the rings of holomorphic functions on neighbourhoods of $p$ in $\C^n$,
$$\I_{p,\,\C^n}^{\,hol}\,=\,\varinjlim_{p\in A}\,H(A)$$
where $H(A)$ is the space of holomorphic functions on $A$. The spaces $H(A)$ have the compact-open topology and the local ring $\I_{p,\,\C^n}^{\,hol}$ inherits the direct limit topology. The topology on the ring $\I_{p,\,\R^n}$ is defined as the initial one induced by the monomorphism \eqref{germs_isomorphism}, with which \eqref{germs_isomorphism} becomes an embedding. The topology on every ideal $\mathcal{I}_k$ is the relative one.

Concretely, a sequence of germs $([f_n])_{n\in\N}$ in $\I_{p,\,\R^n}$ converges to the germ $[f]$ if there is a neighbourhood $A$ of $p$ in $\C^n$ whereat the unique extensions $\tilde{f}_n$ and $\tilde{f}$ of the respective functions $f_n$ and $f$ are defined and the sequence $(\tilde{f}_n)_{n\in\N}$ converges locally uniformly to $\tilde{f}$ in $H(A)$, that is
\begin{equation}\label{germs_convergence}
\xymatrix{ [f_n] \ar[r]^{\I_{p,\R^n}} & [f]}
\qquad {\rm if}\qquad
\xymatrix{ \tilde{f}_n \ar[r]^{H(A)} & \tilde{f}},\quad p\in A\subset \C^n.
\end{equation}

In what follows, since the extension of a real polynomial to the complex domain is unique, we will abuse of notation and identify the polynomial with its extension. In particular,
$$S_k(\R^n)\,\subset\,H(\C^n).$$

\begin{lema}\label{convergencias_lema}
A sequence of homogeneous polynomials in $S_k(\R^n)$ locally uniformly converges if and only if it does so in $H(\C^n)$.
\end{lema}
\begin{proof}
This immediately follows from the observation that both convergences are equivalent to the coefficient-wise one in $S_k(\R^n)$.
\end{proof}

\begin{lema}\label{rho_continuous_lemma}
In expression \eqref{canonical_projection}, the quotient topology on the space of homogeneous polynomials $S_k(\R^n)$ is the compact-open topology and the canonical projection $\rho_k$ is continuous.
\end{lema}
\begin{proof}
Because of the Cauchy integral formula in the context of several complex variables, the uniqueness of the extension and expression \eqref{germs_convergence}, the coefficients of the power series representation of a germ are continuous. Indeed, if $[f_n] \xrightarrow{\I_{p}} [f]$, then $\tilde{f_n} \xrightarrow{H(A)} \tilde{f}$ for some neighbourhood $A$ of $p$ in $\C^n$, hence
$$c_\alpha(f_n)\,=\,c_\alpha(\tilde{f}_n)\,\rightarrow\,c_\alpha(\tilde{f})\,=\,c_\alpha(f),\qquad \forall\ \alpha\in\N_0^n.$$
Considering those multi-indexes with $\vert \alpha\vert\,=\,k$, by expression \eqref{canonical_projection_II} we have that $\rho_k$ is continuous with respect to the compact-open topology $\mathcal{T}_{co}$ on $S_k(\R^n)$.

Since by definition, the quotient topology $\mathcal{T}_q$ on $S_k(\R^n)$ is the finest one such that $\rho_k$ is continuous, we have $\mathcal{T}_{co}\subset \mathcal{T}_{q}$.

It rest to show the opposite inclusion. To this end, consider the canonical monomorphism
$$\begin{tikzcd}        
\iota_k:\,S_k(\R^n) \arrow[hook]{r} & \I_{p}
\end{tikzcd},\qquad f\,\mapsto\,[f].$$
We claim that this morphism is continuous with respect to the compact-open topology on $S_k(\R^n)$. In effect, consider a sequence $(f_n)$ in $S_k(\R^n)$. If $f_n\xrightarrow{\mathcal{T}_{co}} f$, then by Lemma \ref{convergencias_lema} we have $f_n\xrightarrow{H(\C^n)} f$, thus by expression \eqref{germs_convergence} we conclude
$$\xymatrix{\iota_k(f_n)\,=\,[f_n] \ar[r]^{\I_{p}} & [f]\,=\,\iota_k(f)}$$
and this proves the claim. In particular,
\begin{equation}\label{inclusion_lemma_topology}
\iota_k^{-1}(\mathcal{T}_{\I})\,\subset\,\mathcal{T}_{co}
\end{equation}
where we have denoted by $\mathcal{T}_{\I}$ the topology on $\I_p$. Since by definition $\rho_k$ is continuous with respect to the quotient topology on $S_k(\R^n)$ and $\iota_k$ is a section of $\rho_k$, that is $\rho_k\circ\iota_k=id$, we have
\begin{equation}\label{inclusion_lemma_topology_II}
\mathcal{T}_{q}\,=\,\iota_k^{-1}\left(\rho_k^{-1}(\mathcal{T}_{q})\right)\,\subset\,\iota_k^{-1}(\mathcal{T}_{\I}).
\end{equation}
Because of inclusions \eqref{inclusion_lemma_topology} and \eqref{inclusion_lemma_topology_II}, we conclude that $\mathcal{T}_{q}\,\subset\,\mathcal{T}_{co}$ and we have the result.
\end{proof}
\bigskip

\begin{proof}[Proof of Proposition \ref{generic_prop_II}]
By Lemma \ref{rho_continuous_lemma} and the fact that $\mathcal{H}$ is an open subset of $S_l(\R^n)$, we have that $\rho_l^{-1}(\mathcal{H})$ is an open subset of $\mathcal{I}_l$. It rest to show that it is dense.

For this purpose, consider $[f]$ in $\mathcal{I}_l$. As a power series, $f$ has a radius of convergence $R>0$ and a unique splitting
$$f\,=\,f_l\,+\,f_{>l},\qquad f_l\,=\,\rho_l([f]).$$
Since $\mathcal{H}$ is dense in $S_l(\R^n)$, there is a sequence $(s_m)$ in $\mathcal{H}$ converging to $f_l$ locally uniformly. Consider the following sequence,
$$f_m\,=\,s_m+f_{>l},\qquad m\in\N.$$
By Cauchy-Hadamard formula, every $f_m$ has radius of convergence $R>0$.

Although the left hand side of the expression
$$f\,-\,f_m\,=\,f_l\,-\,s_m$$
is defined on the polydisc with radius $R$, because of the fact that the right hand side is a polynomial, it uniquely extends to the whole $\R^n$ and the sequence converges locally uniformly to zero. By Lemma \ref{convergencias_lema}, it does so on $H(\C^n)$ and taking the respective germs, because of expression \eqref{germs_convergence} we conclude
$$\xymatrix{[f_m] \ar[r]^{\I_{p}} & [f]},\qquad [f_m]\in\mathcal{I}_l,\quad\rho_l([f_m])=s_m\,\in\,\mathcal{H}.$$
Since the choice of $[f]$ was arbitrary, this finishes the proof.
\end{proof}

\section{Asymptotic and boomerang orbits}\label{Asymtotic_section}

Consider an orbit $\alpha$ in phase space $TE$ converging to $\mathbf{0}$ as time increases. Since the orbit is eventually contained in a compact set, by general theory it is defined on the interval $[t_0,\,+\infty)$ for some initial time $t_0$.

Suppose that $\alpha$ is not the constant orbit. In particular, by uniqueness of the solution, every point in the $\alpha$ orbit is different from $\mathbf{0}$, that is
\begin{equation}\label{planteo_alpha}
\alpha(t)\rightarrow \mathbf{0}\quad \mathrm{as} \quad t\rightarrow +\infty,\qquad\alpha(t)\,\neq\,\mathbf{0},\ \forall\,t\in [t_{\rm 0},\,+\infty).
\end{equation}
Because the energy is a continuous function which is constant along the orbit and $\mathbf{0}$ has zero energy, the orbit $\alpha$ has zero energy also and $p$ is not a strict minimum of the potential $U$. Thus, statement \eqref{planteo_alpha} holds for the configuration space trajectory as well,
\begin{equation}\label{planteo_x}
x(t)\rightarrow 0\quad \mathrm{as}\quad t\rightarrow +\infty,\qquad x(t)\,\neq\,0,\ \forall\,t\in [t_{0},\,+\infty),
\end{equation}
where we have denoted $\alpha(t)\,=\,(x(t),\,v(t))$.

\begin{defi}
The orbit $\alpha$ is asymptotic to $\Sigma\,\subset\, S^{n-1}$ in configuration space if the omega-limit of $(t\mapsto x(t)/\Vert x(t) \Vert)$ is contained in $\Sigma$,
$$\omega\left(\, x(t)/\Vert x(t)\Vert\,\right)\,\subset\,\Sigma.$$
\end{defi}

As before, denote by $f$ the restriction to the unit sphere of the first non-zero jet of the potential $U$ at the critical point $p$.

In this section, it will be proved that an orbit converging to $\mathbf 0$ in phase space is asymptotic to some component of the critical locus of $f$ in configuration space. Conversely, if additionally $f$ is Morse and zero is a regular value, then for every critical point of $f$ there is such an orbit. Moreover in the purely mechanical case, under the mentioned additional hypotheses on $f$, if the asymptote is a local minimum of $f$, then the orbit is \emph{boomerang} (see Definition \ref{Boomerang} below).

\begin{prop}\label{Prop_asymptote_criticalpoint}
The orbit $\alpha$ is asymptotic to some connected component of the region of the critical locus of $f$ where the function is less than or equal to zero.
\end{prop}
\begin{proof}
Because of statement \eqref{planteo_alpha} and the fact that $\alpha$ has zero energy, we have that the function $\tilde{H}$ is zero on $\pi_{\leq 0}^{-1}\left(\alpha\right)$. By Proposition \ref{McGehee_blowup_mechanical_system}, $\pi_{\leq 0}^{-1}\left(\alpha\right)$ is an orbit in $M$ up to some positive reparametrization. Concretely, there is a strictly increasing continuous map
\begin{equation}\label{increasing_time_eq}
\tau\mapsto t\,(\tau),\qquad t\,(\tau)\rightarrow +\infty\quad \rm as\quad \mathnormal{\tau\rightarrow \omega_+}
\end{equation}
such that
$$\gamma:[\tau_0,\,\omega_+)\rightarrow M,\qquad \gamma(\tau)\,=\,\pi_{\leq 0}^{-1}\left(\,\alpha(t\,(\tau))\,\right)$$
is an orbit in $M$ where $\omega_+$ is the supremum of its maximal interval of definition. In particular, the omega-limit set $\omega(\gamma)$ is contained in the critical boundary $\partial M_0$ which is a compact set hence the orbit $\gamma$ is eventually contained in a compact subset of $M$. By general theory, we conclude that $\omega_+=+\infty$. Moreover, because of Proposition \ref{near_recurrency_II}, the omega-limit set is a non-empty compact connected set actually contained in the fixed point set of the critical boundary $\partial M_0$.

In particular, by the concrete description of the fixed point set in Lemma \ref{concrete_description_fixpoint} and the definition of the McGehee coordinates \eqref{Coordinates_McGehee}, we have
$$\frac{x(\,t\,(\tau)\,)}{\Vert x(\,t\,(\tau)\,) \Vert}\,=\,q(\tau)\,\rightarrow\,\Sigma(f)\cap f^{-1}(-\infty,0]\quad
 \rm as\quad \mathnormal{\tau\rightarrow \omega_+}$$
where we have denoted by $\Sigma(f)$ the critical locus of $f$. By expression \eqref{increasing_time_eq}, we have the result.
\end{proof}

Under additional hypotheses on the function $f$, the converse of the previous proposition holds.

\begin{prop}
If $f$ is a Morse function and zero is not a critical value, then, for every critical point $q$ of $f$ such that $f(q)<0$, there is an orbit asymptotic to $q$ in configuration space. In particular, there are at least as many asymptotic orbits as critical points of $f$ with negative value.
\end{prop}
\begin{proof}
Let $q$ be a critical point of $f$ such that $f(q)<0$. By Lemma \ref{concrete_description_fixpoint}, consider the fixed point
$$\zeta=(0,q,\nu q)\in \partial M_0,\qquad \nu\,<\,0.$$
By Proposition \ref{Morse_Proposition}, there is an orbit $\gamma$ in the interior of $M$ converging to $\zeta$ as time goes to $+\infty$ and the function $\tilde{H}$ is zero along the orbit. By Proposition \ref{McGehee_blowup_mechanical_system}, $\pi_{\leq 0}(\gamma)$ is an orbit in $TE$ up to some positive reparametrization, converging to $\mathbf{0}$ as time goes to $+\infty$ (see the argument at the beginning of the section). In resume, there is a strictly continuous function
\begin{equation}\label{increasing_time_eq_5}
\tau\mapsto t\,(\tau),\qquad t\,(\tau)\rightarrow +\infty\quad \rm as\quad \mathnormal{\tau\rightarrow +\infty}
\end{equation}
with the property that
$$\alpha:[t_0,\,+\infty)\rightarrow TE,\qquad \alpha(t(\tau))\,=\,\pi_{\leq 0}\left(\,\gamma(\tau)\,\right)$$
is an orbit in $TE$ for some initial time $t_0$. By definition of the McGehee coordinates \eqref{Coordinates_McGehee} and \eqref{increasing_time_eq_5} we have
$$\lim_{t\to +\infty}\,\frac{x(t)}{\Vert x(t)\Vert}\,=\,\lim_{\tau\to +\infty}\,\frac{x(t(\tau))}{\Vert x(t(\tau))\Vert}\,=\,
\lim_{\tau\to +\infty}\,q(\tau)\,=\,q.$$
Therefore, $\gamma$ is an orbit asymptotic in configuration space with asymptote $q$ and we have the result.
\end{proof}



\textbf{In the remainder of this section}, we will assume that the magnetic potential is zero, that is, \textbf{the Lagrangian $L$ is mechanical}.

Consider the \emph{time reversed orbit} $\alpha^{-1}$ defined as follows,
$$\alpha^{-1}:\,(-\infty,\,-t_0]\,\rightarrow\, TE,\qquad \alpha^{-1}(t)\,=\,\left(x(-t),\,-v(-t)\right).$$
By direct verification, it is a solution of the motion equation \eqref{Newton_eq2} hence it is an orbit. Now, by expression \eqref{planteo_alpha} we have
\begin{equation}\label{planteo_alpha_reversed}
\alpha^{-1}(t)\rightarrow \mathbf{0}\quad \mathrm{as} \quad t\rightarrow -\infty,\qquad\alpha^{-1}(t)\,\neq\,\mathbf{0},\ \forall\,t\in(-\infty,\,-t_{0}].
\end{equation}

In what follows, we will consider the orbit $\alpha$ as well as its time reversed orbit $\alpha^{-1}$ defined on their respective maximal interval of definition.

\begin{defi}\label{Boomerang}
An asymptotic in configuration space orbit $\alpha$ is boomerang if there is a sequence $(z_n)_{n\in\N}$ in phase space $TE$ such that
$$\limsup_n\,\I(z_n)\,=\,\alpha\,\cup\,\{\mathbf{0}\}\,\cup\,\alpha^{-1}.$$
\end{defi}

Note that the following result does not follow by any means from the continuity of ODE solutions with respect to the initial conditions.

\begin{prop}
If $f$ is a Morse function and zero is not a critical value, then every asymptotic in configuration space orbit whose asymptote is a local minimum of $f$ is boomerang.
\end{prop}
\begin{proof}
This is a continuation of the proof of Proposition \ref{Prop_asymptote_criticalpoint}. Now, assume that the asymptote $q_*$ of the orbit $\alpha$ in configuration space is a local minimum of $f$. By Proposition \ref{Morse_Proposition}, we have
\begin{equation}\label{increasing_time_eq_3}
\gamma(\tau)\,\rightarrow\,\zeta_{\,in}\,=\,(0,\,q_*,\,\nu_*\,q_*)\quad {\rm as} \mathnormal{\quad \tau\rightarrow +\infty,\qquad \nu_*<\rm 0.}
\end{equation}

Define the time reversed orbit $\gamma^{-1}$ by the expression
$$\gamma^{-1}(\tau)\,=\,(r(-\tau),\,q(-\tau),\,-y(-\tau))$$
where it is defined and we have denoted $\gamma(\tau)\,=\,(r(\tau),\,q(\tau),\,y(\tau))$. Note that $\pi_{\leq 0}(\gamma^{-1})$ coincides with $\alpha^{-1}$ up to some positive reparametrization. By definition, it is clear that
\begin{equation}\label{increasing_time_eq_4}
\gamma^{-1}(\tau)\,\rightarrow\,\zeta_{\,out}\,=\,(0,\,q_*,\,-\nu_*\,q_*)\quad {\rm as} \mathnormal{\quad \tau\rightarrow -\infty.}
\end{equation}

By Lemma \ref{heteroclinic_orbit}, there is a subcritical heteroclinic orbit $\I_h$ connecting the fixed points $\zeta_{\,in}$ and $\zeta_{\,out}$. Let $\zeta_{\,h}\,=\,(0,\,q_{\,h},\,y_{\,h})$ be a point in $\I_h$ and take the sequence $(\zeta_{\,n})$ in the interior of $M$ converging to $\zeta_{\,h}$,
$$(1/n,\,q_{\,h},\,y_{\,h})\,=\,\zeta_{\,n}\,\rightarrow\,\zeta_{\,h},\qquad n\in\N.$$
Since the fixed points $\zeta_{\,in}$ and $\zeta_{\,out}$ are hyperbolic by Proposition \ref{Morse_Proposition}, we have
\begin{equation}\label{aux1_eq_limite}
L\,=\,\limsup_n\,\I(\zeta_{\,n})\,\subset\,W^s(\zeta_{\,in})\,\cup\,\I_h\,\cup\,W^u(\zeta_{\,out}),
\end{equation}
such that the following intersections are non-empty,
\begin{equation}\label{aux2_eq_limite}
L\,\cap\,\left(W^s(\zeta_{\,in})\,-\,\{\zeta_{\,in}\}\right)\,\neq\,\emptyset,\qquad
L\,\cap\,\left(W^u(\zeta_{\,out})\,-\,\{\zeta_{\,out}\}\right)\,\neq\,\emptyset.
\end{equation}
Here, $W^s(\zeta_{\,in})$ and $W^u(\zeta_{\,out})$ denote the stable and unstable manifolds in $M$ of the fixed points $\zeta_{\,in}$ and $\zeta_{\,out}$ respectively. Because the asymptote $q$ is a local minimum by hypothesis, by Proposition \ref{Morse_Proposition} again, we have
\begin{equation}\label{aux3_eq_limite}
\gamma\,=\,W^s(\zeta_{\,in})\,-\,\{\zeta_{\,in}\},\qquad \gamma^{-1}\,=\,W^u(\zeta_{\,out})\,-\,\{\zeta_{\,out}\}
\end{equation}
where the orbit $\gamma$ as well as its time reversed orbit $\gamma^{-1}$ are defined on their respective maximal interval of definition. Thus, by \eqref{aux1_eq_limite}, \eqref{aux2_eq_limite} and \eqref{aux3_eq_limite} we have the expression
\begin{equation}\label{aux4_eq_limite}
\limsup_n\,\I(\zeta_{\,n})\,=\,\gamma\,\cup\,\I_h\,\cup\,\gamma^{-1}.
\end{equation}

Note that, since the orbit $\I_h$ is subcritical, the point $\zeta_ {\,h}$ is subcritical as well, that is $\tilde{H}(\zeta_{\,h})<0$. Thus, the sequence $(\zeta_{\,n})$ is eventually subcritical and taking a subsequence if necessary, we may suppose that the entire sequence is so. Define the sequence $(z_n)$ in $TE$ by
$$z_n\,=\,\pi_{\leq 0}(\,\zeta_{\,n}\,).$$
By Proposition \ref{McGehee_blowup_mechanical_system} and expression \eqref{aux4_eq_limite} we conclude that
$$\limsup_n\,\I(z_n)\,=\,\pi_{\leq 0}\left(\limsup_n\,\I(\zeta_{\,n})\right)\,=\,\pi_{\leq 0}\left(\gamma\,\cup\,\I_h\,\cup\,\gamma^{-1}\right)\,=\,\alpha\,\cup\,\{\mathbf{0}\}\,\cup\,\alpha^{-1}$$
and the proof is finished.
\end{proof}

As a final comment, note that the passage from \eqref{aux1_eq_limite} to \eqref{aux4_eq_limite} in the previous proof is due to the fact that the unstable manifold $W^u(\zeta_{\,in})$ as well as the stable one $W^s(\zeta_{\,out})$, are both open sets in the critical and subcritical boundary $\partial M_{\leq 0}$. In general, this will not be the case and the projection by the blowup map of the stable and unstable manifolds in $M$ of a pair of fixed points in the critical boundary $\partial M_0$ connected by some subcritical or critical heteroclinic orbit, will define a sort of scattering structure at the equilibrium point in phase space. This is the analogue of the collision-ejection structure near a singularity described by Devaney in \cite{Devaney}.

\section{Total instability I: generic case}\label{Total_inst_section}

Recall that $p$ is a critical point of the potential $U$ and without loss of generality, we have assumed that $U$ is zero at the point $p$. In particular, $p$ is an equilibrium point in configuration space and $(p,0)\,=\,\mathbf{0}$ is a fixed point of the Euler-Lagrange flow in phase space.

The point $p$ is \emph{totally unstable} if there is a neighbourhood of $p$ such that no subcritical positive orbit can be confined to the neighbourhood in configuration space.

\begin{defi}\label{total_instability_def}
The point $p$ is totally unstable if there is a neighbourhood $B$ of $p$ in $E$ such that $\I^+(z)\not\subset TB$ for every $z\in TB$ verifying $H(z)<0$.
\end{defi}


In this section, it will be proved that total instability is a generic property. This will immediately follow from a stronger result, also proved in this section.

Denote by $f$ the restriction to the unit sphere of the first non-zero jet of the potential $U$ at the critical point $p$.

The next proposition is a \emph{tightness result} and has interest in itself. It says that, under generic hypotheses on the function $f$, any positive orbit in phase space with, a priori, critical or subcritical energy and confined in configuration space in a sufficiently small neighbourhood of $p$, necessarily gets arbitrarily close to $(p,0)$ in the future hence it has, a posteriori, critical energy.

In particular, a subcritical orbit cannot be confined in configuration space in a sufficiently small neighbourhood of $p$, that is to say, the equilibrium point $p$ is totally unstable.

\begin{prop}[Mechanical tightness]\label{mechanical_tightness}
Suppose that $f$ is a Morse function and zero is a regular value. Then, there is a neighbourhood $B$ of $p$ in configuration space such that, for any orbit $\I^+(z)$ in phase space verifying
$$\I^{+}(z)\,\subset\, TB,\qquad H(z)\leq 0,$$
its omega-limit set contains $(p,0)$. In particular, $H(z)=0$.
\end{prop}
\medskip

\begin{cor}\label{Mech_tightness_asymptotic orbit}
Under the hypotheses of Proposition \ref{mechanical_tightness}, if additionally the positive orbit $\I^+(z)$ converges and $z$ is not $(p,0)$, then it must be a zero energy orbit converging to $(p,0)$ in phase space and asymptotic to $q$ in configuration space such that $q$ is a critical point of $f$ verifying $f(q)<0$.
\end{cor}
\begin{proof}
By Proposition \ref{mechanical_tightness} and the convergence hypothesis, the orbit $\alpha$ converges to $(p,0)$. By Proposition \ref{McGehee_blowup_mechanical_system}, up to some positive reparametrization $\pi_{\leq 0}^{-1}\left(\alpha\right)$ is an orbit in $M$
$$\gamma(\tau)\,=\,(r(\tau),\,q(\tau),\,y(\tau)),$$
whose omega-limit is contained in the critical and subcritical boundary $\partial M_{\leq 0}$. By Propositions \ref{near_recurrency_II}, \ref{Morse_Proposition} and Lemma \ref{concrete_description_fixpoint}, the orbit $\gamma$ converges to some fixed point $(0,\,q_*,\,\nu_*\,q_*)$ as time goes to infinity, where $q_*$ is a critical point of $f$ verifying $f(q_*)< 0$. In particular, by definition of the McGehee coordinates \eqref{Coordinates_McGehee} we have
$$\lim_{t\to +\infty}\,\frac{x(t)}{\Vert x(t)\Vert}\,=\,\lim_{\tau\to +\infty}\,\frac{x(t(\tau))}{\Vert x(t(\tau))\Vert}\,=\,
\lim_{\tau\to +\infty}\,q(\tau)\,=\,q,$$
where $(\tau\mapsto t(\tau))$ is the aforementioned positive time reparametrization.

Finally, by Corollary \ref{Corollary_omegalimit_constraint} and the definition of $\tilde{H}$ in Proposition \ref{Extension_energy}, we conclude that the orbit $\alpha$ has zero energy and the result follows.
\end{proof}

In other words, for a critical or subcritical orbit confined near the point $p$ in configuration space, generically the only thing preventing it to be an asymptotic orbit is its convergence in phase space.

\begin{cor}[Total instability criterion]\label{total_inst_criterion}
Suppose that $f$ is a Morse function and zero is a regular value. Then, the point $p$ is totally unstable.
\end{cor}
\begin{proof}
This follows immediately from Proposition \ref{mechanical_tightness}.
\end{proof}

As mentioned in the introduction, by definition \ref{total_instability_def}, total instability is a local property in configuration space hence it only depends on the germ of the Lagrangian at $T_p E$. In this respect, considering the fact that in the smooth class the stability question depends non-trivially on the kinetic term, it is quite interesting that, at least in the real analytic class, a sufficient condition for total instability involves the first non-zero jet of the potential only without any mention neither to higher order terms of it nor to those of the kinetic and magnetic term.

Recall that $(E,g)$ is a real analytic manifold with real analytic Riemannian metric and this defines a real analytic kinetic term on phase space,
$$K(x,v)\,=\,\frac{g_x(v,v)}{2},\qquad (x,v)\,\in\,TE.$$
For every function germ $[U]$ and one\,-\,form germ $[\mu]$ at $p$, both real analytic, the Lagrangian
$$L=K+\mu-U$$
is defined on $TA$ where $A$ is a sufficiently small neighbourhood of $p$ in configuration space and determines an Euler-Lagrange flow there. It is in this sense that the notion of total instability depends on the Lagrangian.

Recall, from section \ref{Genericity_II}, the stalk of real analytic germs $\I_p$ at $p$. It is a local ring with maximal ideal $\mathfrak{m}$ and denote by $\mathcal{I}_k$ its $k$-power, that is $\mathcal{I}_k\,=\,\mathfrak{m}^k\lhd \I_p$.

The following is one of the main results of this paper.

\begin{teo}[Total instability is generic]
For every $k\geq 2$, there is an open and dense subset $\mathcal{A}$ of the ideal $\mathcal{I}_k$ such that, for every potential germ in $\mathcal{A}$ and every magnetic potential verifying the weak magnetism hypothesis \eqref{weak_magnetism}, the point $p$ is totally unstable.
\end{teo}
\begin{proof}
This follows immediately from Proposition \ref{generic_prop_II} and the total instability criterion, Corollary \ref{total_inst_criterion}.
\end{proof}
\medskip

\begin{proof}[Proof of Proposition \ref{mechanical_tightness}]
If $p$ is a strict minimum of the potential $U$, then there is neighbourhood $B$ of $p$ in configuration space such that $p$ is the only point $x\in B$ verifying $U(x)\leq 0$. In particular, the fixed point $\mathbf{0}$ is the only point $z\in TB$ verifying $H(z)\leq 0$ and the result follows. Thus, in the statement of the proposition, assume the additional hypothesis that $p$ is not a strict minimum of the potential.

Suppose that there is no such neighbourhood. Then, there is a sequence $(z_k)$ in phase space $TE$ such that for every natural $k$,
\begin{equation}\label{negation_tightness_prop}
\I^{+}(z_k)\subset B(p,1/k)\times \R^n\subset TE,\qquad H(z_k)\leq 0,\qquad \mathbf{0}\,\notin\,\omega(z_k).
\end{equation}

Since the energy $H$ is constant along the orbits, we have for every natural $k$,
$$H(x,\,v)\,=\,\frac{\Vert v\Vert^2}{2}\,+\,U(x)\,=\,H(z_k)\,\leq\,0,\qquad \forall\ (x,\,v)\,\in\,\I^+(z_k)$$
hence
$$\Vert v\Vert\,\leq\,\left(-2\,U(x)\right)^{1/2}\,\leq\,\left(-2\,m_U \right)^{1/2}\,=\,\mathsf{v},\qquad \forall\ (x,\,v)\,\in\,\I^+(z_k)$$
or equivalently,
\begin{equation}\label{common_compact_tightness_prop}
\I^+(z_k)\,\subset\,\overline{B(p,1)}\times \overline{B(0,\mathsf{v})}\,=\,K
\end{equation}
where we have denoted by $m_U$ the minimum of $U$ on the compact $\overline{B(p,1)}$ which is lower than or equal to zero since $z_k$ verifies the energy constraint in \eqref{negation_tightness_prop}. In resume, there is a compact set $K$ in phase space $TE$ containing every orbit $\I^{+}(z_k)$.

It follows immediately from \eqref{negation_tightness_prop} that, for every natural $k$,
$$\I^{+}(z_k)\subset\,H^{-1}(-\infty,0]\,-\,\{\mathbf{0}\}.$$
By Proposition \ref{McGehee_blowup_mechanical_system}, the blowup inverse $\pi_{\leq 0}^{-1}$ is defined on the right hand side of the previous expression and moreover, it is a flow equivalence there. In particular, the sequence $(\zeta_{\,k})$ of preimages of $(z_k)$ by the blowup is well defined and verifies
\begin{equation}\label{def_z'_tightness_prop}
\I^{+}(\zeta_{\,k})\subset\,\tilde{H}^{-1}(-\infty,0],\qquad \zeta_{\,k}\,=\,\pi_{\leq 0}^{-1}(z_k)
\end{equation}
for every natural $k$. Furthermore, by the same proposition, the blowup is a proper map hence the preimage of $K$ is a compact set $K'$ in $M$, containing every positive orbit $\I^{+}(\zeta_{\,k})$ by \eqref{common_compact_tightness_prop}, that is
\begin{equation}\label{common_compact_tightness_prop_2}
\I^+(\zeta_{\,k})\,\subset\,K',\qquad K'\,=\,\pi_{\leq 0}^{-1}(K)\,\subset\,\tilde{H}^{-1}(-\infty,0]
\end{equation}
for every natural $k$. Note that $K'$ contains $\partial M_{\leq 0}\,=\,\pi_{\leq 0}^{-1}(\mathbf{0})$ for $(p,0)\,=\,\mathbf{0}$ belongs to $K$ by definition \eqref{common_compact_tightness_prop},
\begin{equation}\label{K'_contains_boundary_tightness}
\partial M_{\leq 0}\,\subset\,K'.
\end{equation}

We claim that the superior limit of the positive orbits is contained in the critical and subcritical boundary,
\begin{equation}\label{limsup_tightness_prop}
\limsup_k\,\I^+(\zeta_{\,k})\,\subset\,\partial M_{\leq 0}.
\end{equation}

In effect, with respect to the McGehee coordinates, by \eqref{negation_tightness_prop} the radial coordinate of $\I^+(\zeta_{\,k})$ is bounded above by $1/k$ which tends to zero, thus the superior limit of the positive orbits is contained in the boundary $\partial M$. Because of expression \eqref{def_z'_tightness_prop} and the definition of $\partial M_{\leq 0}$ in Corollary \ref{compact_invariant_boundaries}, the claim follows.

Now, for every natural $k$, take a minimal invariant subset $\omega_k$ of the omega-limit set $\omega(\zeta_{\,k})$. Note that, by expression \eqref{negation_tightness_prop} and Proposition \ref{McGehee_blowup_mechanical_system}, we have
\begin{equation}
\partial M_{\leq 0}\,\cap\, \omega_k\,\subset\,\partial M_{\leq 0}\,\cap\, \omega(\zeta_{\,k})\,=\,\pi_{\leq 0}\left(\{\mathbf{0}\}\,\cap\,\omega(z_k)\right)\,=\,\emptyset.
\end{equation}
Nevertheless, by expressions \eqref{common_compact_tightness_prop_2} and \eqref{limsup_tightness_prop}, these minimal sets are contained in the compact set $K'$ and its superior limit is contained in the critical and subcritical boundary,
\begin{equation}\label{omega_in_suplim_eq}
\omega\,=\,\limsup_k\,\omega_k\,\subset\,\partial M_{\leq 0},\qquad \qquad \omega_k\,\subset\,K',\quad \forall\ k\in\N.
\end{equation}
Actually, recall that the closure of the union is taken before the intersection in the superior limit definition \eqref{limsup_definition_eq} and because
$$\omega_k\,\subset\,\omega(\zeta_{\,k})\,\subset\, \overline{\I^+(\zeta_{\,k})},$$
expression \eqref{omega_in_suplim_eq} follows. Therefore, by \eqref{K'_contains_boundary_tightness} and Proposition \ref{near_recurrency}, $\omega$ is a non-empty compact set contained in the fixed point set of $\partial M_{0}$.

By hypothesis, $f$ is a Morse function and zero is a regular value hence, by Proposition \ref{Morse_Proposition}, the fixed point set in $\partial M_{0}$ consist of isolated points only. Let $\zeta$ be a fixed point in $\omega$. By the same proposition, $\zeta$ is hyperbolic in the extended McGehee manifold $\hat{M}$, then by Hartman-Grobman Theorem, there is a neighbourhood $A$ of $\zeta$ in $\hat{M}$ whereat the McGehee flow and the flow of the linearized system at $\zeta$ are equivalent.

By Corollary \ref{near_recurrency_cor}, taking a subsequence if necessary, we may suppose that the sequence $(\omega_k)$ is contained in the neighbourhood $A\cap M$ of $\zeta$ in $M$.

However, the only recurrent set in $A\cap M$ consists of the single point $\zeta$ and we have that for every natural $k$,
$$\{\zeta\}\,=\,\omega_k\,\subset\,\omega(\zeta_{\,k}).$$
Evaluating the previous expression with the blowup map $\pi_{\leq 0}$, we conclude that $\mathbf{0}$ belongs to $\omega(z_k)$ which is absurd by \eqref{negation_tightness_prop} and this proves the first part of the statement. The second part easily follows from the fact that $H$ is continuous, constant along the orbits and zero at $\mathbf{0}$. This concludes the proof.
\end{proof}

\section{Total instability II: non-generic case}\label{non_generic}

This section is devoted to the case in which the restriction to the unit sphere of first non-zero jet of the potential at the critical point is not a Morse function or has critical value zero. In contrast with the criterion given in Corollary \ref{total_inst_criterion}, the criterion for this non-generic case proven here will involve higher order terms of the potential and the metric tensor coefficients as well.

Although the non-trivial dependence of the kinetic term on the stability of an equilibrium is known, this is the first time that it is considered in an instability criterion.

Another way of looking at this result, is that it bounds the field of possibilities in which to look for a counterexample, if there are any, to the total instability conjecture.


Recall from section \ref{blowup_definition_section}, after taking a local chart and identifying the space with its coordinates, that the configuration space $E$ is a ball of radius $r_E$ centered at the origin in Euclidean space,
$$E\,=\,B(0,r_E)\,\subset\, \R^n.$$
There is a real analytic function $U$ on $E$, the potential, together with a real analytic one-form $\mu$ on $E$, the magnetic potential, defining the real analytic Lagrangian
$$L:TE\rightarrow\R,\qquad L(x,v)\,=\,\frac{1}{2}\sum_{i,j=1}^n\,\gi_{ij}(x)\,v^i v^j\,+\,\mu_x(v)\,-\,U(x),\qquad (x,v)\in TE,$$
where $v^i$ denotes the $i$-th component of the vector $v$ in Euclidean space. The origin $p=0$ is a critical point of the potential $U$ and we have assumed without loss of generality that $U(p)=0$.

Suppose that $U$ is not a homogeneous polynomial and consider the first and second non-zero homogeneous terms in the Maclaurin series of the potential and the metric tensor coefficients,
$$U\,=\,U_{l_1}\,+\,U_{l_2}\,+\,\ldots,\qquad \gi_{ij}\,=\,\delta_{ij}\,+\,\gi^{(m)}_{ij}\,+\,\ldots.$$
In what follows, if the metric tensor is the Euclidean one, that is if the Lagrangian is Newtonian, then formally we can take $m$ as infinity. It is worth to emphasize that \textbf{no additional hypothesis is assumed on the first non-zero jet} $U_{l_1}$.

Throughout this section, it will be assumed that the degree value $d$ of the first non-zero jet of the magnetic potential at $p$ verifies a weaker magnetism hypothesis than the original, namely
\begin{equation}\label{weak_magnetism_II}
\Delta\,=\,d\,-l_1/2\, >\,\mu\,=\,\min\,\{l_2-l_1,\,m\}.
\end{equation}
The reason is simply because otherwise the following instability criteria will not decide.

Recall the globally defined function $\nu$ on $M$ in expression \eqref{function_nu},
$$\nu:M\rightarrow\R,\qquad \nu\,(r,q,y)\,=\,\left\langle q,\,y\right\rangle.$$

\begin{lema}\label{Lemma1_ch7}
There is a globally defined real analytic function $F$ on $M$ such that the derivative of $\nu$ along the McGehee flow reads as follows
$$\nu'\,=\,
\left(1+l_1/2\right)\,\Vert y_{tg}\Vert^2\,-\,l_1\,\tilde{H}\,+\,r^\mu\,F,\qquad \mu\,=\,\min\,\{l_2-l_1,\,m\}.$$
\end{lema}
\begin{proof}
A similar calculation as the one in Lemma \ref{Lyapunov_calculation} gives the result.
\end{proof}

A tedious but straightforward calculation gives the value of the function $F$ in the previous lemma at the fixed point set in the critical boundary.

\begin{lema}\label{Lemma2_ch7}
Consider the function $F$ and the value $\mu$ in Lemma \ref{Lemma1_ch7}. Consider also the fixed point $(0,q,\nu q)$ in the critical boundary $\partial M_0$. Then,
$$F\,(0,q,\nu q)\,=\,-\delta_\mu^{l_2-l_1}\,(l_2-l_1)\,U_{l_2}(q)\,+\,
\delta_\mu^m\,m\,U_{l_1}(q)\,\sum_{i,j=1}^n\,\gi^{(m)}_{ij}(q)\,q^i q^j,$$
where $q^i$ denotes the $i$-th component of the vector $q$ in Euclidean space.
\end{lema}

Note that the evaluation in the previous lemma only depends on $q$ and actually defines the \emph{criterion function} $\mathcal{C}:S^{n-1}\rightarrow\R$ on the unit sphere given by
$$\mathcal{C}(q)\,=\,-\delta_\mu^{l_2-l_1}\,(l_2-l_1)\,U_{l_2}(q)\,+\,
\delta_\mu^m\,m\,U_{l_1}(q)\,\sum_{i,j=1}^n\,\gi^{(m)}_{ij}(q)\,q^i q^j,\qquad q\in S^{n-1}.$$

Similarly as in section \ref{Total_inst_section}, the total instability criterion will follow from a stronger tightness result. Denote by $f$ the restriction to the unit sphere of the first non-zero jet $U_{l_1}$ of the potential.

\begin{prop}[Mechanical tightness II]\label{mechanical_tightness_nongeneric}
Suppose that $\mathcal{C}(q)>0$ on every critical point $q$ of $f$ such that $f(q)\leq 0$. Then, there is a neighbourhood $B$ of $p$ in configuration space such that, for any orbit $\I^+(z)$ in phase space verifying
$$\I^{+}(z)\,\subset\, TB,\qquad H(z)\leq 0,$$
its omega-limit set contains $\mathbf{0}$. In particular, $H(z)=0$.
\end{prop}

As immediate corollaries, we have the following criteria.

\begin{cor}[Non-generic total instability criterion]\label{nongeneric_criterion_prop}
If $\mathcal{C}(q)>0$ on every critical point $q$ of $f$ such that $f(q)\leq 0$, then the point $p$ is totally unstable.
\end{cor}

As far as our knowledge goes, \textbf{the following is the first criterion that covers the possibility of} a zero value critical point of the first-non zero jet of the potential restricted to the unit sphere, or equivalently, that there are \textbf{other critical points of the homogeneous polynomial $U_{l_1}$ besides the origin}.

\begin{cor}
In Newtonian mechanics, that is the case in which the metric tensor is the Euclidean one, the point $p$ is totally unstable if the second non-zero homogeneous term in the Maclaurin series of the potential is negative on the critical locus of the the first non-zero homogeneous term restricted to the unit sphere where the latter term is less than or equal to zero.
\end{cor}

It is interesting to note that, outside Newtonian mechanics, the criterion is in general not useful if the difference $l_2-l_1$ is greater than or equal to three. The following is a negative result.

\begin{cor}
If the Riemann curvature is non-null at $p$ and $l_2-l_1\geq 3$, then the criterion in Proposition \ref{nongeneric_criterion_prop} does not decide.
\end{cor}
\begin{proof}
With respect to normal coordinates at the origin, the metric tensor coefficients have the expansion
$$g_{ij}(x)\,=\,\delta_{ij}\,+\,\frac{1}{3}R_{kijl}\,x^k x^l\,+\,O(\Vert x \Vert^3)$$
where we have used Einstein's convention on repeated indices. In particular, if the Riemann tensor is not zero at the origin, then $m=2$ and
$$\gi^{(2)}_{ij}(q)\,q^i q^j\,=\,\frac{1}{3}R_{kijl}\,q^k q^l q^i q^j \,=\,0,$$
for every $q$ in the unit sphere, because of the antisymmetry identities on the Riemann tensor coefficients. In particular, the criterion function $\mathcal{C}$ is identically zero and this proves the result.
\end{proof}

The following is the analogue of Corollary \ref{Mech_tightness_asymptotic orbit} in the non-generic case. The proof is almost verbatim as the one of the mentioned corollary and a straightforward adaptation of the latter gives the former.

\begin{cor}\label{Mech_tightness_asymptotic orbit_II}
Under the hypotheses of Proposition \ref{mechanical_tightness_nongeneric}, if additionally the positive orbit $\I^+(z)$ converges and $z$ is not $(p,0)$, then it must be a zero energy orbit converging to $(p,0)$ in phase space and asymptotic in configuration space to the critical locus of the the first non-zero jet of the potential restricted to the unit sphere.
\end{cor}
\medskip

\begin{proof}[Proof of Proposition \ref{mechanical_tightness_nongeneric}]
Because of Lemmas \ref{concrete_description_fixpoint}, \ref{Lemma1_ch7}, \ref{Lemma2_ch7} and the definition of the criterion function, we have that 
$F$ is positive on the fixed point set of the critical boundary hence there is a neighbourhood $A\subset M$ of the mentioned fixed point set whereat $F$ is positive as well. In particular, because of Lemma \ref{Lemma1_ch7}, we have that
\begin{equation}\label{eq1_tightness_II}
\nu'(\zeta)>0,\qquad \forall\  \zeta\,\in\, A\cap \left(\tilde{H}^{-1}(-\infty, 0]-\partial M_{\leq 0}\right).
\end{equation}

Suppose that the point $p$ is not totally unstable. A verbatim argument as in the proof of Proposition \ref{mechanical_tightness} gives a sequence $(\omega_k)$ of minimal invariant compact sets such that for every natural $k$,
\begin{equation}\label{eq2_tightness_II}
\omega_k\,\subset\,A\cap \left(\tilde{H}^{-1}(-\infty, 0]-\partial M_{\leq 0}\right).
\end{equation}

Considering every $\omega_k$ as a dynamical system itself, by general theory, for every natural $k$ there is an ergodic measure $\mu_k$ with support in $\omega_k$. Since $\nu$ and $\nu'$ are continuous on the compact set $\omega_k$, because of expressions \eqref{eq1_tightness_II} and \eqref{eq2_tightness_II}, for every natural $k$ we have
$$0\,=\,\lim_{T\to+\infty}\,\frac{\nu(\varphi_T(\zeta))-\nu(\zeta)}{T}\,=\,\lim_{T\to+\infty}\,\frac{1}{T}\int_0^T d\tau\,\nu'(\varphi_\tau(\zeta))\,=\,
\int d\mu_k\,\nu'\,>\,0,$$
for $\mu_k\,$-\,almost every $\zeta$ in $\omega_k$, which is absurd and the result is proved.
\end{proof}

\appendix

\section{Abstract McGehee blowup}\label{abstract_blowup_appendix}

In this section we construct the \emph{abstract McGehee manifold}, an intrinsic object that naturally adopts the usual McGehee coordinates once a local coordinate chart on the configuration manifold is chosen. From this object, the analogue of Proposition \ref{McGehee_blowup_mechanical_system} in the intrinsic context naturally follows, realizing the \emph{abstract McGehee blowup}.

\begin{lema}[Extension]\label{extension_Lemma}
Consider a non-zero real analytic function $w$ defined on the ball $B(0,r_w)\subset\R^n$ such that its first non-zero jet at the origin is the degree $d$ homogeneous polynomial $w_d$. Then, there is a unique real analytic function $w_{>d}$ on the real analytic manifold $(-r_w,r_w)\times S^{n-1}$ such that
$$w(r\,q)\,=\,r^d\,w_d(q)\,+\,r^{d+1}\,w_{>d}(r,q),\qquad (r,q)\,\in\, (-r_w,r_w)\times S^{n-1}.$$
\end{lema}
\begin{proof}
It is clear that such a function exists and is unique outside the embedded sphere $\{0\}\times S^{n-1}$, that is for $r\neq 0$. It rest to show that it can be extended real analytically to the sphere. By continuity, the extension will be unique.

Since $w$ is real analytic at the origin, it has a power series representation centered at the mentioned point with radius of convergence $R>0$. Formally, $w$ can be written as
$$w(x)\,=\,w_d(x)\,+\,\sum_{|\alpha|=d+1}\,h_\alpha(x)\,x^\alpha$$
where $h_\alpha$ are power series centered at the origin and we have used the multi-index notation. By the Cauchy-Hadamard formula, each $h_\alpha$ has radius of convergence $R$ hence these are real analytic functions on the neighbourhood $B(0,R)$.

Because the unit sphere is a real analytic submanifold, the restriction of the real analytic functions $x^\alpha$ to the sphere are real analytic. Since the composition and linear combination of real analytic functions is real analytic as well, we have that the function $w_{>d}$ defined on $(-R,R)\times S^{n-1}$ as
$$w_{>d}\,(r,q)\,=\,\sum_{|\alpha|=d+1}\,h_\alpha(r\,q)\,q^\alpha,\qquad (r,q)\,\in\, (-R,R)\times S^{n-1}$$
is real analytic. The proof is complete.
\end{proof}

Recall the map $\pi:M\rightarrow T\R^n$ defined by the McGehee coordinates, expressions \eqref{map_pi} and \eqref{Coordinates_McGehee} in section \ref{blowup_definition_section}.

\begin{lema}\label{Coherence_charts}
Let $B$ and $B'$ be open sets in $\R^n$ containing the origin and consider a real analytic diffeomorphism $\phi:B\rightarrow B'$ such that $\phi(0)\,=\,0$. Then, there is a unique real analytic diffeomorphism $\tilde{\phi}$ such that the following diagram commutes,
$$\xymatrix{\pi^{-1}(TB) \ar[d]_\pi \ar@{.>}[rr]^{\tilde{\phi}} & & \pi^{-1}(TB') \ar[d]^\pi \\
TB\ar[rr]^{d\phi} & & TB'}.$$
Moreover, if $d_0 \phi$ is an isometry, that is an element in the group $O(n)$, then on the boundary $\partial M$ of $\pi^{-1}(TB)$, we have
\begin{equation}\label{rotation_tensor}
\tilde{\phi}(0,q,y)\,=\,(0,\,d_0\phi(q),\,d_0\phi(y)),\qquad (0,q,y)\in \partial M.
\end{equation}
\end{lema}
\begin{proof}
Since $d\phi$ is a real analytic diffeomorphism and $\pi$ is so outside $\partial M$, there is a unique real analytic diffeomorphism $\tilde{\phi}$ making the diagram commute outside the boundary, that is in the region $r>0$. It rest to show that the map $\tilde{\phi}$ extends as a real analytic diffeomorphism to $r=0$. By continuity, the extension will make the diagram commute as well.

There is a ball of radius $\varepsilon>0$ centered at the origin and contained in $B$. By Lemma \ref{extension_Lemma}, there is a real analytic function $h$ such that
$$\phi(r\,q)\,=\,r\,d_0\phi(q)\,+\,r^2\,h(r,q),\qquad (r,q)\in (-\varepsilon,\varepsilon)\times S^{n-1}.$$
Since $d_0\phi$ is an isomorphism, the function given by
\begin{equation}\label{eq1_Lemma_phi_tilde}
\Lambda(r,q)\,=\,\Vert d_0\phi(q)\,+\,r\,h(r,q)\Vert,\qquad (r,q)\,\in\, (-\varepsilon,\varepsilon)\times S^{n-1}
\end{equation}
is real analytic and strictly positive for a sufficiently small $\varepsilon>0$. Therefore, defining
\begin{equation}\label{eq2_Lemma_phi_tilde}
(\tilde{r},\tilde{q},\tilde{y})\,=\,\tilde{\phi}\,(r,q,y),\qquad (r,q,y)\in \pi^{-1}(TB)-\partial M,
\end{equation}
we necessarily have the real analytic extension to the region $(-\varepsilon,\varepsilon)\times S^{n-1}\times \R^n$,
\begin{equation}\label{eq3_Lemma_phi_tilde}
\tilde{r}=\Lambda r,\qquad \tilde{q}\,=\,\frac{1}{\Lambda}\,(d_0\phi(q)\,+\,r\,h(r,q)),\qquad \tilde{y}\,=\,\frac{1}{\Lambda^{k/2}}\,d_{rq}\phi(y).
\end{equation}
In particular, there is a unique real analytic extension of the map $\tilde{\phi}$ to the region $r=0$. A verbatim argument for the inverse map $(d\phi)^{-1}=d(\phi^{-1})$ shows the existence of $\widetilde{\phi^{-1}}$ with the above properties and because of uniqueness,
$$\widetilde{\phi}\,\widetilde{\phi^{-1}}\,=\,\widetilde{\phi^{-1}}\,\widetilde{\phi}\,=\,\widetilde{id}\,=\,id.$$
Hence $\tilde{\phi}$ has a real analytic inverse. The second part of the statement follows directly from \eqref{eq1_Lemma_phi_tilde},\eqref{eq2_Lemma_phi_tilde} and \eqref{eq3_Lemma_phi_tilde} and we have the result.
\end{proof}

Consider a real analytic Riemannian manifold $(E,g)$ and a point $p\in E$. Denote by $\mathcal{A}$ the family of local charts $(A,\eta)$ centered at $p$ such that the pullback by $\eta$ of the usual inner product at the origin coincides with the metric at $p$, that is
$$\eta:A\rightarrow\R^n,\qquad \eta(p)\,=\,0,\qquad \eta^*\,\langle\cdot,\cdot\rangle_0\,=\,g_p,\qquad p\in A\subset E.$$
Consider the open subset of $E$ covered by these charts,
$$E_{\mathcal{A}}\,=\,\bigcup_{(A,\eta)\in\mathcal{A}}\,A.$$

\begin{prop}\label{abstract_manifold}
There is a real analytic manifold $\mathcal{M}$ with boundary and a surjective continuous map $\epsilon:\mathcal{M}\rightarrow TE_{\mathcal{A}}$ that is a real analytic diffeomorphism outside the boundary such that, for every local chart $(A,\eta)$ in $\mathcal{A}$, there is a real analytic diffeomorphism over its image $\tilde{\eta}$ making the following diagram commute
\begin{equation}\label{commutative_diagram_main}
\xymatrix{\mathcal{M} \ar[rr]^\epsilon & & TE_{\mathcal{A}} & \\
& \ \ \ \ \ \epsilon^{-1}(TA) \ar@{^{(}->}[ul] \ar[rr]^\epsilon \ar[dl]_{\tilde{\eta}} & & \ \ TA \ar[dl]^{d\eta} \ar@{^{(}->}[ul] \\
M \ar[rr]^\pi & & T\R^n & }.
\end{equation}
\end{prop}
\begin{proof}
This is a formal construction. Consider a pair of local charts $(A_1,\eta_1)$ and $(A_2,\eta_2)$ in $\mathcal{A}$. Denote by $A$ the intersection $A_{12}=A_1\cap A_2$ and by $\eta$ and $\eta'$ the restrictions of $\eta_1$ and $\eta_2$ to $A$ respectively. We have the commutative diagram
$$\xymatrix{ & A  \ar[dl]_\eta \ar[dr]^{\eta'} & \\
\eta(A) \ar[rr]^{\phi_{\eta \eta'}} & & \eta'(A) }.$$
Note that the map $\phi_{\eta \eta'}$ is a real analytic diffeomorphism verifying
$$\phi_{\eta \eta'}(0)\,=\,0,\qquad d_0\phi_{\eta \eta'}\,\in\,SO(n).$$
Taking the tangent functor on the previous diagram and considering the fact that $d\eta$ and $d\eta'$ are real analytic diffeomorphisms, by Lemma \ref{Coherence_charts} there are surjective maps $\epsilon_{\eta}$, $\epsilon_{\eta'}$ and a real analytic diffeomorphism $\widetilde{\phi}_{\eta \eta'}$ such that the following diagram commutes
\begin{equation}\label{commutative_diagram}
\xymatrix{ \pi^{-1}(T\,\eta(A))  \ar[dd]_\pi \ar@{.>}[rr]^{\widetilde{\phi}_{\eta \eta'}} \ar@{.>}[dr]^{\epsilon_\eta} & & \pi^{-1}(T\,\eta'(A)) \ar[dd]^\pi \ar@{.>}[dl]_{\epsilon_{\eta'}} \\
& TA  \ar[dl]_{d\eta} \ar[dr]^{d\eta'} & \\
T\,\eta(A)  \ar[rr]_{d\phi_{\eta \eta'}} & & T\,\eta'(A)}.
\end{equation}
Note that the maps $\epsilon_{\eta}$, $\epsilon_{\eta'}$ are continuous and real analytic diffeomorphisms outside the boundary, properties inherited from $\pi$.

By the uniqueness in Lemma \ref{Coherence_charts}, the correspondence $\phi\mapsto\tilde{\phi}$ is functorial, that is it verifies
$\widetilde{\phi_1\phi_2}=\widetilde{\phi_1}\widetilde{\phi_2}$ and $\widetilde{id}=id$. In particular, the cocycle condition is verified, specifically, given local charts $(A_1,\eta_1)$, $(A_2,\eta_2)$ and $(A_3,\eta_3)$ in $\mathcal{A}$, the following identity is verified on $A_{123}=A_1\cap A_2 \cap A_3$,
\begin{equation}\label{cocycle_condition}
\widetilde{\phi}_{\eta_1 \eta_2}\,\widetilde{\phi}_{\eta_2 \eta_3}\,\widetilde{\phi}_{\eta_3 \eta_1}\,=\,id.
\end{equation}

Define an equivalence relation $\sim$ on the set
\begin{equation}\label{disjoint_union}
\bigsqcup_{(A,\eta)\in\mathcal{A}}\,\pi^{-1}(T\,\eta(A)),
\end{equation}
by considering the points $\zeta\in \pi^{-1}(T\,\eta_1(A_1))$ and $\zeta'\in \pi^{-1}(T\,\eta_2(A_2))$ equivalent if $\zeta'=\widetilde{\phi}_{\eta_1 \eta_2}(\zeta)$. By definition, it is reflexive and symmetric and because of the cocycle condition \eqref{cocycle_condition}, it is transitive.

Since the equivalence relation is given by the real analytic diffeomorphisms $\widetilde{\phi}_{\eta_1 \eta_2}$, the quotient set inherits a real analytic manifold structure hence the following object is a real analytic manifold with boundary
$$\mathcal{M}\,=\,\bigsqcup_{(A,\eta)\in\mathcal{A}}\,\pi^{-1}(T\,\eta(A))\,/\,\sim.$$

Because of the commutative diagram \eqref{commutative_diagram}, there is a well-defined map
$$\epsilon:\mathcal{M}\rightarrow TE_\mathcal{A},\qquad [(\zeta,(A,\eta))]\mapsto\epsilon_\eta(\zeta).$$
We have made explicit that the union \eqref{disjoint_union} is disjoint and a typical element is a pair $(\zeta,(A,\eta))$, so that the fact that the map is well-defined is evident. This map inherits the properties of the collection of maps $\epsilon_\eta$, namely that it is surjective, continuous and a real analytic diffeomorphism outside the boundary.

Again, because of the commutative diagram \eqref{commutative_diagram}, it is clear that, given $(A,\eta)\in \mathcal{A}$, the map
$$\tilde{\eta}:\epsilon^{-1}(TA)\rightarrow \pi^{-1}(T\,\eta(A)),\qquad [(\zeta,(A,\eta))]\mapsto \zeta$$
is a well-defined real analytic diffeomorphism and makes the following diagram commute,
$$\xymatrix{\epsilon^{-1}(TA) \ar[rr]^\epsilon \ar[d]_{\tilde{\eta}} & & \ \ TA \ar[d]^{d\eta} \\
\pi^{-1}(T\,\eta(A)) \ar[rr]^\pi & & T\,\eta(A)}.$$
This concludes the proof.
\end{proof}

\begin{defi}
The manifold $\mathcal{M}$ in Proposition \ref{abstract_manifold} will be called the abstract McGehee manifold.
\end{defi}

The main objects constructed in section \ref{blowup_definition_section} have their corresponding intrinsic analogue in the manifold $\mathcal{M}$ is such a way that these are recovered from the intrinsic ones once a local chart in $\mathcal{A}$ is chosen.

Indeed, because of the transformation \eqref{rotation_tensor}, there is a real analytic function
$$\hat{H}:\partial\mathcal{M}\rightarrow\R$$
such that, for every local chart $(A,\eta)$ in $\mathcal{A}$, we have on the boundary $\partial \mathcal{M}$ that
$$\hat{H}\,=\,\tilde{H}\circ \tilde{\eta},$$
where $\tilde{H}$ is the function defined in Proposition \ref{Extension_energy}. In particular, with respect to any local chart in $\mathcal{A}$, the restriction $f$ of the first non-zero jet of the potential to the unit sphere is unique up to some rotation of the sphere. Thus, the degree of the first non-zero jet of the potential is intrinsic itself.

By definition, with respect to some local chart in $\mathcal{A}$, outside the boundary $\partial M$, the function $\tilde{H}$ differs from $H\circ \pi$ only by a positive factor. Thus, because of Proposition \ref{abstract_manifold}, the regions $(H\circ \epsilon)^{-1}(-\infty,0)$ and $(H\circ \epsilon)^{-1}(0)$ are intrinsic in $\mathcal{M}-\partial\mathcal{M}$. At the boundary $\partial\mathcal{M}$, by continuity these regions are intrinsic as well. In particular, the critical and subcritical boundaries
$$\partial\mathcal{M}_0=\partial\mathcal{M}\cap (H\circ \epsilon)^{-1}(0),\qquad \partial\mathcal{M}_{<0}=\partial\mathcal{M}\cap (H\circ \epsilon)^{-1}(-\infty,0)$$ are intrinsic.

Similarly, with respect to some local chart $(A,\eta)$ in $\mathcal{A}$, outside the boundary $\partial M$, the vector field $\tilde{X}$ defined in Proposition \ref{Extension_vector} differs from $(\pi^{-1})_*(X)$ only by a positive factor. Hence, by Proposition \ref{abstract_manifold} and Lemma \ref{flow_equivalence}, there is an intrinsic flow equivalence class on $\mathcal{M}-\partial\mathcal{M}$. Because the transformation \eqref{rotation_tensor} is a symmetry of the equations \eqref{McGehee_eq} at $r=0$, at the boundary $\partial\mathcal{M}$ there is an intrinsic vector field whose pushout by $\tilde{\eta}$ is $\tilde{X}$. In particular, the dynamics on $\partial\mathcal{M}_{\leq 0}$ is conjugated to the one on $\partial M_{\leq 0}$.

The only a priori non-trivial part verifying the previous claim that the transformation \eqref{rotation_tensor} is a symmetry of the equations \eqref{McGehee_eq} at $r=0$, is perhaps the gradient term. This is because the function $\nu$ is clearly invariant under the transformation and the rest of the terms are linear. For the gradient term, denote by $\Lambda$ the rotation $d_0\phi$ and denote by $\mathsf{x}$ the new coordinates verifying $x=\Lambda\,\mathsf{x}$. By the chain rule we have
$$\frac{\partial\,(U\circ \Lambda)}{\partial\,\mathsf{x}^a}(\mathsf{x})\,=\,\sum_b\,\frac{\partial\,U}{\partial\, x^b}(\Lambda\,\mathsf{x})\,\frac{\partial\,\Lambda^b}{\partial\, \mathsf{x}^a}(\mathsf{x})\,=\,
\sum_b\,\Lambda^b_a\,\frac{\partial\,U}{\partial\, x^b}(\Lambda\,\mathsf{x})\,=\,
\sum_b\,(\Lambda^{-1})^a_b\,\frac{\partial\,U}{\partial\, x^b}(\Lambda\,\mathsf{x})
$$
where we have used that $\Lambda$ is orthogonal and made the abuse of notation identifying the function $\Lambda$ with its matrix. Therefore,
$$\nabla\,(U\circ \Lambda)\,=\,\Lambda^{-1}\,(\nabla U)\circ \Lambda$$
and the claim is proved.

In view of the previous observations, the analogue of Proposition \ref{McGehee_blowup_mechanical_system} holds for the intrinsic objects as well. In the following statement, we consider the energy function restricted to $TE_{\mathcal{A}}$,
$$H:TE_{\mathcal{A}}\rightarrow \R.$$

\begin{prop}\label{abstract_McGehee_blowup}
Suppose that the critical point $p$ of the potential $U$ is not a strict minimum and $U(p)=0$. Then, the restriction of the map $\epsilon$ to the critical and subcritical energy regions
$$\epsilon_{\leq 0}:\,(H\circ \epsilon)^{-1}(-\infty,0]\rightarrow H^{-1}(-\infty,0]$$
is a surjective continuous proper map with the following properties:
\medskip

\begin{enumerate}
\item $\epsilon_{\leq 0}^{-1}\left((p,0)\right)\,=\,\it \partial \mathcal{M}_{\leq \rm 0}\,\neq\,\emptyset$.
\medskip

\item The following restriction is a homeomorphism
$$\epsilon_{\leq 0}:\,(H\circ \epsilon)^{-1}(-\infty,0]\,-\,\partial \mathcal{M}_{\leq 0}\rightarrow H^{-1}(-\infty,0]\,-\,\{(p,0)\}.$$

\item The following restriction is a real analytic diffeomorphism
$$\epsilon_{< 0}:\,(H\circ \epsilon)^{-1}(-\infty,0)\,-\,\partial \mathcal{M}_{< 0}\,\rightarrow\, H^{-1}(-\infty,0).$$

\item Dynamically, it maps an orbit to another orbit preserving their orientations and the resulting orbit map is surjective. In particular, the map in item \ref{item_homeo_McGehee_blowup} is a topological equivalence of flows.
\end{enumerate}
\end{prop}
\begin{proof}
This follows directly from Propositions \ref{abstract_manifold} and \ref{McGehee_blowup_mechanical_system}.
\end{proof}

\begin{defi}
The map in Proposition \ref{abstract_McGehee_blowup} is the abstract McGehee blowup for the Lagrangian system at $p$.
\end{defi}

\section{Accumulation at the critical boundary}\label{flow_eq_appendix}

Note that in the following lemma, neither the completeness of the manifolds nor of the vector fields is assumed.

\begin{lema}\label{flow_equivalence}
Let $\psi:N\rightarrow N'$ be a diffeomorphism and $X$ a vector field on $N$. Consider a real-valued positive smooth function $f$ on $N'$, and define the vector field $X'$ on $N'$,
$$X'(p')\,=\,f(p')\,\psi_*(X)(p'),\qquad f(p')>0,\qquad p'\in N'.$$
Then, the corresponding flows are equivalent, that is $\psi$ maps bijectively orbits in orbits preserving their orientation.
\end{lema}
\begin{proof}
Consider a point $p$ in $N$ and $p'=\psi(p)$. The corresponding flows $\eta$ and $\eta'$ on $N$ and $N'$ respectively, have the following maximal intervals of definition at the points $p$ and $p'$ respectively
$$\left(t\mapsto \eta_p(t)\right),\qquad t\in (\omega_-,\,\omega_+),$$
$$\left(t'\mapsto \eta'_{p'}(t')\right),\qquad t'\in (\omega'_-,\,\omega'_+).$$

Define the function $t\mapsto t'(t)$ by the formula
\begin{equation}\label{formula_times}
t'(t)\,=\,\int_0^t\,\frac{ds}{f(\psi\circ \eta_p(s))},\qquad t\in (\omega_-,\,\omega_+).
\end{equation}
Note that it is an orientation preserving diffeomorphism. We claim that
\begin{equation}\label{flow_equality}
\eta'_{p'}\left(t'(t)\right)\,=\,\psi\circ \eta_p(t),\qquad t\in (\omega_-,\,\omega_+).
\end{equation}
Indeed, it follows immediately from
$$\frac{d\left(\psi\circ \eta_p\right)}{dt}\,=\,d\psi\left(\frac{d\eta_p}{dt}\right)\,=\,d\psi\left(X\circ\eta_p\right)\,=\,\psi_*(X)\left(\psi\circ\eta_p\right),$$
formula \eqref{formula_times} and the uniqueness of solutions. In particular, we have shown that
$$t'\,(\omega_-,\,\omega_+)\,\subset\,(\omega'_-,\,\omega'_+).$$

It rest to show that the previous inclusion is actually an equality and from this it will immediately follow that
$$\psi\left(\I_p\right)\,=\,\I_{p'},$$
thus proving the result since \eqref{formula_times} is a positive reparametrization and the point $p$ was arbitrarily chosen.

Denote by $(t'_-,\,t'_+)\,=\,t'\,(\omega_-,\,\omega_+)$ and suppose that $t'_+\,<\,\omega'_+$. In particular, the supremum $t'_+$ is finite and by \eqref{flow_equality} we have
\begin{equation}\label{limit_flows}
\lim_{t\to \omega_+}\,\eta_p(t)\,=\,\psi^{-1}\left(\eta'_{p'}(t'_+)\right).
\end{equation}
Thus, $\eta_p(t)$ is eventually contained in a compact set hence $\omega_+\,=\,+\infty$ by general theory. By formula \eqref{formula_times} we have
\begin{equation}\label{convergence_integral}
\int_0^{+\infty}\,\frac{ds}{f(\psi\circ \eta_p(s))}\,=\,t'_+\,<\,+\infty.
\end{equation}
However, by expression \eqref{limit_flows}, there is a time $t_*>0$ such that
$$0\,<\,\frac{1}{2}\frac{1}{f(\eta'_{p'}(t'_+))}\,\leq \,\frac{1}{f(\psi\circ \eta_p(t))},\qquad \forall\,t\geq t_*,$$
which is absurd by the convergence of the integral \eqref{convergence_integral}, thus proving $t'_+\,=\,\omega'_+$. A similar argument shows that $t'_-\,=\,\omega'_-$ and this completes the proof.
\end{proof}

Recall that the map $\pi$ in \eqref{map_pi} is a continuous map on the manifold $M$ with boundary, which is a real analytic diffeomorphism on the interior.  Recall also that the preimage of $\mathbf{0}$ by $\pi$ coincides with the boundary $\partial M$.

\begin{cor}\label{Corollary_omegalimit_constraint}
Consider an orbit in the interior of $M$ whose omega-limit (alpha-limit) intersects the boundary $\partial M$. Then, the function $\tilde{H}$ is zero along the orbit and the aforementioned intersection is contained in the critical boundary $\partial M_0$.
\end{cor}
\begin{proof}
Let $\gamma$ be an orbit in the interior of $M$ whose omega-limit (alpha-limit) $\omega(\gamma)$ intersects $\partial M$. Since the map $\pi$ is a diffeomorphism on the interior of $M$, by definition of the vector field $\tilde{X}$ in Proposition \ref{Extension_vector} and the previous Lemma \ref{flow_equivalence}, we have that $\pi(\gamma)$ is an orbit in $TE-\{\mathbf{0}\}$ up to some positive reparametrization, whose omega-limit (alpha-limit) contains the point $\mathbf{0}$. Since the energy $H$ is continuous on $TE$, constant along any orbit and $H(\mathbf{0})\mathnormal{ =\rm 0}$, we have that $H$ is zero on $\pi(\gamma)$. Recalling the definition of $\tilde{H}$ on the interior of $M$ in Proposition \ref{Extension_energy}, we conclude that $\tilde{H}$ is zero along the orbit. By Proposition \ref{Extension_energy} again, $\tilde{H}$ continuously extends to the boundary hence $\tilde{H}$ is zero on the omega-limit (alpha-limit) of the orbit as well. By definition of the critical boundary, the corollary is proved.
\end{proof}

\section*{Acknowledgments}
The author is grateful to José Job Flores for having simulated the equations and produced such beautiful figures.

\end{document}